\documentclass[12pt]{article}
\usepackage[utf8]{inputenc}
\usepackage[T1]{fontenc}
\usepackage[english]{babel}
\usepackage{amsmath,amssymb,amsfonts,amsthm,mathptmx}
\usepackage[inner=2.4cm,outer=2.65cm,bottom=3cm]{geometry}

\usepackage{hyperref}

\theoremstyle{remark}
\newtheorem{remark}{Remark}[section]
\theoremstyle{definition}
\newtheorem{theorem}{Theorem}[section]
\newtheorem{definition}[theorem]{Definition}

\newtheorem{proposition}[theorem]{Proposition}
\newtheorem{lemma}[theorem]{Lemma}
\newtheorem{corollary}[theorem]{Corollary}

\newtheorem{algorithm}[theorem]{Scheme}

\allowdisplaybreaks
\usepackage{graphicx}

\DeclareMathOperator{\R}{\mathbb{R}}

\DeclareMathOperator{\C}{\mathcal{C}}

\DeclareMathOperator{\ra}{\rightarrow}
\DeclareMathOperator{\de}{\text{d}}
\DeclareMathOperator{\tr}{tr}

\newcommand{\f}[1]{{\pmb{ #1}}}
\DeclareMathOperator{\di}{\nabla\cdot}

\newcommand{\tu}{\tilde{\f u}}

\newcommand{\tq}{\tilde{\f q}}

\newcommand{\ov}[1]{\overline{{#1}}}

\renewcommand{\t}{\partial_t}

\newcommand{\vv}{\tilde{\f v}}
\newcommand{\dd}{\tilde{\f d}}

\newcommand{\tP}{\tilde{\Phi}}

\newcommand{\tn}{\tilde{n}}

\newcommand{\fn}[1]{\f {{#1}}^{j}}
\newcommand{\fnn}[1]{\f {{#1}}^{j-1}_h}
\newcommand{\fii}[1]{\f {{#1}}^{j}}
\newcommand{\fiii}[1]{\f {{#1}}^{j-1}}
\newcommand{\fh}[1]{\f {{#1}}^{j-1/2}}
\newcommand{\fno}[1]{\f {{\overline{#1}}}^{k}_h}
\newcommand{\fnu}[1]{\f {{\underline{#1}}}^{k}_h}
\newcommand{\fnuo}[1]{\f {{\overline{\underline{#1}}}}^{k}_h}

\newcommand{\fnf}[1]{\f {{{#1}}}^{k}_h}

\newcommand{\td}{{d}_t  }

\renewcommand{\o}{\otimes}
\newcommand{\sumi}{k \sum_{j=1}^J }
\newcommand{\sumii}[1]{k\sum_{j=1}^J\left [{ #1} \right ]}

\newcommand{\tPj}{{\tP}^j}

\newcommand{\tqj}{{\tq}^j}

\newcommand{\PL}{\mathcal{P}_{{\mathbb L}^2 } }

\newcommand{\vvi}{\tilde{\f v}^j}
\newcommand{\ddi}{\tilde{\f d}^j}
\newcommand{\ddj}{\tilde{\f d}^j}
\newcommand{\ddh}{\tilde{\f d}^{j-1/2}}

\newcommand{\vvii}{\tilde{\f v}^{j-1}}
\newcommand{\ddii}{\tilde{\f d}^{j-1}}

\title{Numerical Analysis for Nematic Electrolytes}

\author{\v{L}ubom\'ir Ba\v{n}as\footnote{Fakult\"at f\"ur Mathematik,
Postfach 100 131,
D-33501 Bielefeld, Germany. email: {\tt banas@math.uni-bielefeld.de}}, \qquad 
Robert Lasarzik\footnote{Weierstra\ss -Institut,
Mohrenstra\ss e 39,
D-10117 Berlin, Germany. email: {\tt robert.lasarzik@wias-berlin.de}}, \qquad
Andreas Prohl\footnote{
Mathematisches Institut,
Universit\"{a}t T\"{u}bingen,
Auf der Morgenstelle 10,
D-72076 T\"{u}bingen, Germany. email: {\tt prohl@na.uni-tuebingen.de}}}

\date{April 22, 2020}
\begin{document}

\maketitle
\begin{abstract}
We consider a system of nonlinear PDEs modeling nematic electrolytes, and construct a dissipative solution with the help of its implementable, structure-inheriting  space-time discretization. 
Computational studies are performed to study the mutual effects of electric, elastic, and viscous effects onto the molecules in a nematic electrolyte. 
\end{abstract}


\tableofcontents
\section{Introduction}
We consider a nonlinear system of PDEs to model electrokinetics in \textit{nematic electrolytes}, and show convergence of an implementable discretization to its solution.
\textit{Electrokinetics} is a term describing electrically driven flows, either of a fluid with respect to a solid surface (electroosmosis), or of particles dispersed in a fluid (electrophoresis).
For electrokinetics to occur, it is essential that electric charges of opposite polarities are spatially separated such that an electric field can trigger motion in a fluid. 
Recent studies show that if an {\em anisotropic} electrolyte replaces an isotropic one, the resulting electrokinetic flows may show very different responses to an induced electric field; see~\cite{experi}, and also Figure \ref{fig_stat_dip_phi} in Section \ref{sec:comp}.
Additionally, mechanisms triggering electrokinetics in isotropic electrolytes are far more restrictive, \textit{i.e.},
 an anisotropic medium lifts the constraints on the electric properties of the transported particles and additionally allows for alternating currents to induce a static flow field in the medium (see~\cite{experi} or~\cite{noel}),  and also Figure \ref{fig_appl1_vel_ac} in Section \ref{sec:comp}.

\medskip

In this work, we show the solvability for a model proposed in~\cite[(2.51)--(2.52), (2.55)--(2.56), (2.65)]{noel} (in simplified form) via an implementable (finite-element based) approximation, which is then used for computational studies. The nonlinear PDE system uses 
\begin{itemize}
\item the simplified
Ericksen--Leslie equations for the director field $\f d$ representing the spatio-temporal distribution of average orientations of elongated molecules in the 
liquid-crystalline phase, which is coupled with 
\item the  Nernst--Planck--Poisson system to model phenomena due to  the electrolyte. 
\end{itemize}
Let $\Omega\subset \R^d$, for $d=2,3$ be a bounded convex Lipschitz domain. The
PDE system is as follows:
\begin{subequations}\label{simp}
\begin{align}
\t \f  v + ( \f v \cdot \nabla) \f v - \nu \Delta \f v + \nabla \f d^\top \bigl( \Delta \f d + \varepsilon_a (\f d \cdot \nabla \Phi) \nabla \Phi \bigr) +(n^+-n^-)\nabla \Phi + \nabla p &{}=0 \,,\label{simp:v}\\
\di  \f v &{}=0\,, \label{simp:incomp}\\
\t \f d + (\f v \cdot \nabla ) \f d - ( {I} - \f d \otimes \f d )  \bigl( \Delta \f d + \varepsilon_a (\f d \cdot \nabla \Phi) \nabla \Phi \bigr)&{}=0\,,\label{simp:d}\\
|\f d| &{}=1\,,\label{simp:norm}\\
- \di \left ( \varepsilon (\f d ) \nabla \Phi \right ) &{}= n^+-n^- \,,\label{simp:Phi}\\
\t n^{\pm} + (\f v \cdot \nabla ) n^{\pm} - \di \Bigl( \varepsilon (\f d) \left ( \nabla n^{\pm} \pm n^{\pm} \nabla \Phi \right ) \Bigr) &{}=0\label{simp:c}\,,
\end{align}
\end{subequations}
where $\f v:\ov\Omega \times [0,T] \ra \R^d$  denotes the macroscopic velocity of the nematic fluid, $\f d: \ov\Omega \times [0,T] \ra \R^d $ the local orientation of the nematic liquid crystalline molecules, $\Phi : \ov\Omega \times [0,T] \ra \R $ the electric potential, $ n^{\pm}:\ov\Omega \times [0,T] \ra \R$ the concentrations of positive and negative ions in the liquid crystal, and $p : \ov\Omega \times [0,T] \ra \R$ the pressure of the nematic electrolyte resulting from the incompressibility constraint~\eqref{simp:incomp}.
The matrix
$\varepsilon (\f d ) : = {I} + \varepsilon_a \f d \otimes \f d $ for $\varepsilon_a >0$ is called the dielectric permittivity matrix, which describes
 the relation between the electric displacement $\f D$ of the nematic electrolyte and the electric field $\f E= - \nabla \Phi$: for static dielectric constants measured in the direction of the molecular orientation $\f d$ ($\varepsilon_{\|} $) and perpendicular to it ($\varepsilon_{\bot}$, which we normalized to $1$), this relationship is given by $ \f D =- \varepsilon(\f d ) \nabla \Phi $, where $\varepsilon_a = \varepsilon_{\|}-\varepsilon_{\bot} \geq 0$.  
The system is supplemented with initial data 
\begin{align*}
\f v (0)=\f v_0 \,, \qquad \f d (0)= \f d_0 \quad \text{with} \quad |\f d_0|=1 \,, \qquad n^{\pm}(0)=n^{\pm}_0\in[0,1] \qquad \text{in }\Omega\,,\\
\intertext{and boundary conditions}
\f v = \f 0 \,, \qquad   
\f d = \f d_1
  \,, \qquad  \f n \cdot \Bigl (\f v n^{\pm}- \varepsilon(\f d) \bigl[ \nabla n^{\pm}\pm n^{\pm} \nabla \Phi \bigr] \Bigr ) = 0 \,, \qquad  \f n \cdot  \varepsilon(\f d) \nabla \Phi=0 \qquad \text{on }\partial \Omega \,.
\end{align*}
where the initial and boundary conditions for the director are assumed to fulfill the usual compatibility conditions, \textit{e.g.}, $ \f d _0 = \f d_1 $ on $\partial \Omega$. 

\medskip

The equations (\ref{simp}) are deduced from corresponding ones in~\cite{noel}, and are expected to describe relevant physical effects of the original model. Applied simplifications here include
\begin{itemize}
\item the choice of equal elastic constants in the Oseen-Frank elastic energy, and omitting body forces, as well as inertia effects acting on the director field,
\item setting to zero all Leslie constants in the dissipation potential, except from 
the one that corresponds to the classical Newtonian part of the stress tensor,
\item that electrokinetic effects are initiated from only two species of particles (with related densities $n^{\pm}$),
\item that the interaction matrix in the Nernst-Planck-Poisson part is set to be $\varepsilon(\f d)$, where all appearing constants are set equal to $1$ --- apart from  $\varepsilon_a$, which scales inherent anisotropy and coupling effects.
\end{itemize}
 The goal of this work is to establish a {\em practically} useful solvability concept for (\ref{simp}), where a sequence of functions is generated via an implementable space-time discretization, and the solution of (\ref{simp}) is the limit of  a properly selected convergent subsequence. There are always two main obstacles for such a result:
\begin{itemize}
\item[1.] (\textbf{A structure-inheriting discretization scheme \& stability}) 
It turns out that the construction of 
a sequence of approximate  solutions of practical schemes (here: obtained via finite element method), each of
which, in particular, inherit the properties (\ref{simp:incomp}) and (\ref{simp:norm}) in proper sense, and contain $[0,1]$-valued approximate concentrations,  as well as relevant Lyapunov structures  is still not sufficient to construct a {\em weak solution} of (\ref{simp})  from it as the limit of a proper, convergent subsequence when discretization parameters tend to zero ($d=3$). 
\item[2.] (\textbf{Convergence \& solution concept}) Instead, only a {\em measure-valued solution} is known to exist even in this case for the Ericksen-Leslie system (which is a sub-problem of (\ref{simp}); see
\cite{masswertig}) --- for whose {\em practical} construction no implementable scheme is known so far.
We here show convergence to a {\em dissipative solution} instead. 
\end{itemize}
In particular, we present a structure-preserving space-time discretization for (\ref{simp}) which satisfies {\em all} properties outlined in 1.~in Section \ref{sec:dis}, and show the 
{\em practical} constructability of a {\em dissipative solution} of (\ref{simp}) through it --- see Section \ref{dissipativesolutions} for a definition of this solution concept --- 
as described in item 2. We provide further details of these main results in this work in the 
following discussion.

\medskip

The analysis of models for nematic electrolytes so far is rare in the literature: an interesting approach is~\cite{nemaelecana}, where the authors show \textit{a priori} estimates and weak sequential compactness properties for a  model similar to~\eqref{simp}. In their model, the (pointwise) property of $\f d$ to be a unit vector field (see (\ref{simp:norm})) is approximated by a Ginzburg-Landau-type singular logarithmic potential term that is added to the free energy functional; this additional term is then crucial to validate relevant bounds for the director field. Unfortunately, this additional term blows up for $| \f d| \nearrow 1$, thus leaving open convergence of this model to (\ref{simp}); additionally, it has been
pointed out in \cite{nemaelecana} that it is not obvious how to construct such
approximate sequences that satisfy the assumed properties \cite[(2.26)]{nemaelecana}.

A first sub-problem of (\ref{simp}) is the Navier--Stokes-Nernst--Planck--Poisson system, which corresponds to formally setting $\f d$ constant in~\eqref{simp:v}, and ignoring
(\ref{simp:d})--(\ref{simp:norm}). For this sub-problem, complete analytic resp.~numerical works are available, which prove the global existence of a {\em weak solution}
(see~\textit{e.g.}~\cite{schmuck,const}), as well as their {\em practical} constructability by a
finite element-based, structure-preserving space-time discretization in~\cite{numap}: 
approximate space-time solutions generated from corresponding time-iterates of this scheme satisfy a discrete energy law, as well as a discrete maximum principle (for charges),
from which we may identify the limit of a proper, convergent (sub-)sequence of approximate solutions (for numerical parameters independently tending to zero) as a weak solution.

A second sub-problem of (\ref{simp}) are the simplified Ericksen--Leslie equations
(\ref{simp:v})--(\ref{simp:norm}), where we set $\Phi\equiv n^{\pm} \equiv 0$. We mention several results on the local existence of classical solutions or global existence of classical solutions under smallness conditions (see~\cite{Pruess2,hong,localoseen,linwangdim3} for a similar setting) --- but
our goal here is a global solution concept that copes with possible singular behaviors, and its practical constructability.
For $d=2$, a {\em weak solution} is constructed in \cite{twodim}, thanks to known
properties  of the singularity set of solutions of the harmonic map heat flow to the unit sphere ${\mathbb S}^2$, local energy arguments, and a continuation procedure in time to cope with the elastic stress tensor in (\ref{simp:v}).  Unfortunately, a corresponding existence result for a {\em weak solution} so-far is not known to hold for (the practically relevant case) $d=3$ and general (initial) data, which is why a Ginzburg--Landau penalization is chosen to approximate (\ref{simp:norm}) in (\ref{simp}). {Here, a {\em weak solution} may be constructed (see \cite{linliu1}) for every positive penalization parameter; moreover, different structure-inheriting numerical methods of varying complexity are available in the literature which construct a {\em weak solution} for vanishing discretization parameters: while the first work \cite{liuwal} required 
Hermite-type finite element methods to validate a discrete energy estimate, later ones
\cite{prohl,noel2} only require mixed methods to serve this purpose.
Again,
 passing to the limit with the penalization parameter to validate (\ref{simp:norm}) 
  is open to yield a {\em weak solution} of the simplified Ericksen--Leslie equations.

Instead, a {\em measure-valued solution} is constructed in this way in \cite{masswertig}
for the full Ericksen--Leslie system equipped with the Oseen--Frank energy,
satisfying (\ref{simp:norm}) almost everywhere: its construction 
considers ([sequences of] weak solutions of) the Ginzburg-Landau penalization {\em first, and then} tends the penalization
parameter to zero to efficiently cope with the extra viscous stress tensor in (\ref{simp:v}) in terms of a generalized gradient Young measures. This construction strategy of {\em first} tending discretization parameters to zero in available, structure-inheriting schemes (see \cite{prohl,noel2}) for the Ginzburg-Landau penalization, and only {\em afterwards} tending the penalization parameter to zero clearly excludes 
a practical construction of a {\em measure-valued solution}. 
Regarding this generalized solution concept, a relevant property of it is the
 weak-strong (or rather measure-valued-strong) uniqueness~\cite{weakstrong}, \textit{i.e.}, {\em measure-valued solutions} coincide with the local strong solution emanating from the same initial data  as long as such a strong solution exists.  
 
 A practical shortcoming of the relaxed solution concept in terms of parametrized measures is its complexity; in fact,
the  first moment of a measure-valued solution is the physically relevant quantity in (\ref{simp}) 
which fulfills the so-called {\em dissipative formulation} (see Definition~\ref{def:diss} below, and~\cite{diss} for details). To get this formulation, the solution concept is not relaxed in terms of parametrized measures%
, but the weak formulation of equation~\eqref{simp:v} is relaxed to a {\em relative energy inequality} (see~\eqref{relencont} below). 
The relations of the different solution concepts  for the full Ericksen--Leslie system equipped with the Oseen--Frank energy can be summarized as follows: global {\em weak solutions} exist for the Ginzburg--Landau penalization to approximate the norm restriction (\ref{simp:norm}). In the limit of this approximation, these solutions converge to a {\em measure-valued solution}:
the  first moment of the measure-valued solution is then a {\em dissipative solution} \cite{diss}, which also coincides with the local strong solution as long as the latter exists.

The concept of a dissipative solution was first introduced by P.-L.~Lions in the context of the Euler equations~\cite[Sec.~4.4]{lionsfluid}, with ideas originating from the Boltzmann equation~\cite{LionsBoltzman}. 
It is also applied in the context of incompressible viscous
electro-magneto-hydro\-dy\-na\-mics (see%
~\cite{raymond})
 and equations of viscoelastic diffusion in polymers~\cite{viscoelsticdiff}.
Our first goal in this work is to construct a dissipative solution to~\eqref{simp} via a {\em practical} scheme (see~\eqref{dis} in Section~\ref{sec:dis}). For this purpose, related iterates have to inherit relevant properties of~\eqref{simp}, including a discrete energy law, a discrete unit length property for the director field, and a discrete maximum principle for the charges (see Theorem~\ref{thm:disex}). Upon unconditionally passing to the limit with respect to  the discretization parameters then generates a ($6$-tuple of) limiting functions which may be identified as a dissipative solution of~\eqref{simp}. We remark that the
proposed scheme seems as well to be the first for nematic materials (\textit{i.e.}, including a convection term) which preserves the norm restriction $|\f d|=1$ at every nodal point of the triangulation. 
Another  new ingredient in this article then is to show that the solution to the fully discrete scheme fulfills an \textit{approximate relative energy inequality} (see Proposition~\ref{prop:disrel}), which eventually establishes that a proper limit of this sequence of approximate solutions is a dissipative solution of~\eqref{simp}. 
As a by-product, we show strong convergence to the unique classical solution, as long as this more regular solution exists.

The paper is organized as follows: in the following section, we collect some notations and preliminaries. Section~\ref{sec:cont} is dedicated to the continuous system and collects associated \textit{a priori} estimates and the definition of a dissipative solution. Section~\ref{sec:dis} introduces the fully discrete scheme, its solvability, associated \textit{a priori} estimates, the approximate relative energy inequality, and the convergence to a dissipative solution. 
Section~\ref{sec:comp} discusses computational experiments.

\section{Notation and preliminaries\label{secc:not}}
We denote by $\pmb{\mathcal{V}}:=\{ \f v \in \mathcal{C}_{c}^\infty(\Omega;\R^d)| \di \f v =0\}$ 
the space of smooth solenoi\-dal functions with compact support. By ${\mathbb H} $ and ${\mathbb V}$
we denote the closure of $\pmb{\mathcal{V}}$ with respect to the norm of $\f L^2(\Omega):= {\mathbb L}^2 $  and $ \f H^1( \Omega):= {\mathbb W}^{1,2}$,
respectively.
Note that ${\mathbb H}$ can be characterized by ${\mathbb H} = \{ \f v \in {\mathbb L}^2 | \di \f v =0 \text{ in }\Omega\, , \f n \cdot \f v = 0 \text{ on } \partial \Omega \} $, where the first condition has to be understood in the distributional sense and the second condition in the sense of the trace in ${H}^{-1/2}(\partial \Omega )$. 
The dual space of a Banach space ${\mathbb X}$ is always denoted by ${\mathbb X}^*$ and is equipped with the standard norm; the duality pairing is denoted by $\langle\cdot, \cdot \rangle$. We use the standard notation $( {\mathbb H}^1_0)^*= {\mathbb H}^{-1}$. 
By $\mathbb W^{1,2}/_{\R}$ we denote the functions $f \in \mathbb W^{1,2} $ with $\int_\Omega f \de \f x = 0$. 
We define $ | \f a |_{\varepsilon(\f d)}^2 = \f a \cdot \varepsilon(\f d) \f a $ for $\f a$, $\f d\in \R^d$.
 The inner product in ${\mathbb L}^2$ is denoted by brackets, \textit{i.e.,} $ (\cdot,\cdot)$, and the associated norm is $\| \cdot \|_{{\mathbb L}^2}$. 
 We define the dyadic product of a vector and a matrix by $( \f a \otimes \f A )_{ijk}= \f a _i \f A_{kl}$, and the cross product for two matrices, as well as of a vector and a matrix using the well-known Levi-Cevita symbol $\Upsilon \in \R^{d\times d\times d} $ by
 $ (\f A \times \f B )_{ijk}  = \Upsilon_{ilm} \f A_{jl}\f B_{mk}$, as well as $( \f a \times \f A )_{ij} = \Upsilon_{ilm} \f a_l \f A_{mj}$ , where $ \f a \in \R^d$ and $\f A $, $\f B\in \R^{d\times d} $. For simplicity, we denote the matrix vector multiplication without a sign, \textit{i.e.}, $( \f A \f a) _i = \f A _{ij} \f a_j $, where $ \f a \in \R^d$ and $\f A \in \R^{d\times d} $.
\subsection{Discrete time derivative}
Given a time-step size $k>0$, and a sequence $\{ \varphi^j\}$ for $0\leq j\leq J$ in some Banach space ${\mathbb X}$, we set $\td \varphi ^j := k^{-1} ( \varphi^j- \varphi^{j-1})$ for $j\geq 1$ and $\td \varphi^0=0$. 

The following lemma provides a tool to mimic the Gronwall inequality for the relative energy inequality on the discrete level. 

\begin{lemma}\label{lem:disgron}
Consider sequences $\{ f^j\}_{0\leq j\leq J}\subset \R, \{ g_1^j\}_{0\leq j\leq J}$, $ \{ g_2^j\}_{0\leq j\leq J}$, $\{ y^j\}_{0\leq j\leq J}\subset  \R^+_0  $.  If we have
\begin{align}
\td y^j + f^j \leq g_1^j y^j + g_2^j y^{j-1} \quad \text{for } 1\leq j\leq J\,,\label{grondiff}
\end{align}
then it holds true for $k$ sufficiently small  that
\begin{align*}
- k \sum_{j=0}^{J-1} \td \phi^{j+1} \left ( y^j \prod _{l=1}^j \frac{1}{\omega^l}\right )  + k \sum_{j=1}^{J-1} \phi^j \frac{f^j}{1-kg_1^j} \left ( \prod _{l=1}^j\frac{1}{\omega^l}\right )  \leq \phi(0) y^0 
\end{align*}
for all $ \phi\in\C^\infty_c([0,T)) ${ with } $\phi \geq 0 $, and $ \phi'\leq 0$ on $[0,T]$, where $\phi^j= \phi(jk) $ and $\td \phi^{j+1}= ( \phi^{j+1}-\phi^j)/k$ for $0\leq j \leq J$ with $J = \lfloor (T/k)\rfloor$ as well as  $ \omega^j := \frac{1+kg_2^j}{1-kg_1^j}$.
\end{lemma}
\begin{proof}
From~\eqref{grondiff}, we find 
\begin{align*}
y^j - \omega^j y^{j-1} \leq - \frac{k}{1-kg_1^j} f^j \,.
\end{align*}
By some elementary calculation, we observe (where we understand $ \prod_{l=1}^0 1/\omega^l $ as $1$)
\begin{align*}
- k \sum_{j=0}^{J} \td \phi^{j+1} \left ( y^j \prod_{l=1}^j \frac{1}{\omega^l}\right ) ={}& \phi^0 y^0 - \phi^{J+1} y^{J} \prod_{l=1}^J  \frac{1}{\omega^l} + \sum_{j=1}^J \phi^j \prod_{l=1}^j \frac{1}{\omega^l} \left ( y^j - \omega^jy^{j-1} \right ) \\
\leq{}& \phi^0 y^0 - k\sum_{j=1}^J \phi^j\left ( \frac{f^j}{(1-kg_1^j)} \right ) \prod_{l=1}^j \frac{1}{\omega^l}  \,.
\end{align*}
Note that the term $\phi^{J+1} $ vanishes since $ \phi\in\C^\infty_c([0,T)) $. 

\end{proof}
\subsection{Finite element spaces}
Let $\mathcal{T}_h = \bigcup_{\ell} K_\ell$ be a quasi-uniform triangulation of $\Omega\subset \R^d$ ($d=2,3$) into triangles or tetrahedrons $K_\ell$ of maximal diameter $h>0$; see~\cite{brenner}. 
We additionally assume:
\begin{itemize}
\item[{\bf (A1)}]  $\mathcal{T}_h$ is a strongly acute triangulation, or for $d=2$ a Delaunay triangulation.
\end{itemize}
A triangulation is strongly acute, if  the sum of opposite angles to the common side of any two adjacent triangles is $\leq \pi - \theta$ with $\theta >0$ independent of $h$; see~{\em e.g.}~\cite{ciarletraviart,nochettoverdi}.
For strongly acute meshes, we may verify the $M$-matrix property for the Nernst--Planck--Poisson sub-system, which establishes its unique solvability, and a discrete maximum principle. Moreover, we often use $\mathcal{N}_h = \{ \f x _l \}_{1 \leq l \leq L} $, which is the set of all nodes of $\mathcal{T}_h$, on which we validate a unit-length property of iterates for the director, for example.

Let ${\mathbb P}_l(K; {\mathbb R}^d)$ denote the set of ${\mathbb R}^d$-valued polynomials in $d$ variables of degree $\leq l$ on a triangle/tetrahedron $K \in {\mathcal T}_h$. We introduce the following spaces
\begin{align*}
{Y}_h ={}& \{ y \in
   \C
  (\ov \Omega  )\cap \mathbb L^2_0 \, : \, y \bigl\vert _{K} \in {\mathbb P}_1 ( K ) \quad \forall K \in \mathcal{T}_h \} \\
{Z}_h ={}& \{ y \in
   \C
  (\ov \Omega  ) \, : \, y \bigl\vert _{K} \in {\mathbb P}_1 ( K ) \quad \forall K \in \mathcal{T}_h \} \\
\f Y_h ={}& \{ \f v \in \C_0(\ov \Omega ; \R^d ) \, : \, \f v \bigl\vert_{K} \in {\mathbb P}_1 ( K; \R^d ) \quad \forall K \in \mathcal{T}_h \} \\
\f Z_h ={}& \{ \f v \in \C(\ov \Omega ; \R^d ) \, : \, \f v \bigl\vert_{K} \in {\mathbb P}_1 ( K; \R^d ) \quad \forall K \in \mathcal{T}_h \} \\
\f B^l_h ={}& \{ \f v \in [{\mathbb H}_0^1]^d \,:\, \f v \bigl\vert_K \in {\mathbb P}_l(K;\R^d)  \quad \forall K\in\mathcal{T}_h \}\\
\f X_h ={}& \f Y_h \cup \f B_h^3 \\
M_h ={}& \{ q \in {\mathbb L}^2_0  \cap \C(\ov\Omega)\,:\, q \bigl|_K \in {\mathbb P}_1 (K) \quad \forall K\in \mathcal{T}_h\}\,,
\end{align*}
where $\C_0(\Omega ;\R^d) := \{ \f v \in \C(\ov\Omega) \,:\, \f v = {\bf 0} \text{ on }\partial\Omega \}$ and ${\mathbb L}^2_0  := \{ q\in {\mathbb L}^2 \,:\, (q,1)=0\}$. 
We assume the discrete inf-sup of Babu\v{s}ka-Brezzi condition to be fulfilled, \textit{i.e.}, there exists a constant $C>0$ independent of $h$, {\em s.t.}
\begin{align*}
\sup _{\f v \in \f X_h } \frac{( \di \f v , q ) }{\| \nabla \f v \|_{{\mathbb L}^2}} \geq C \| q \|_{{\mathbb L}^2} \quad \forall\, q\in M_h\,.
\end{align*}
A well-known example complying to this condition is given by the MINI-element given by $\f X_h$ and $M_h$. 
Define 
\begin{align*}
\f V_h := \{ \f v \in \f X_h \, :\, ( \di \f v , q) = 0  \quad \forall\, q\in M_h \} \,.
\end{align*}
We remark the following compatibility condition of spaces from~\cite{numap} is valid above that accounts for coupling effects in the electro-hydro\-dynamical system:
\begin{itemize}
\item[{\bf (A2)}]  $Y_h \cap L^2_0 (\Omega) \subset M_h$.
\end{itemize}
\begin{lemma}[Inverse inequality]\label{lem:inv}
For the considered triangulations ${\mathcal T}_h$, there exists  $C>0$ such that
\begin{align*}
\| \nabla y \|_{{\mathbb L}^p} \leq C h^{-\gamma} \| y \|_{{\mathbb L}^q} \quad 
\qquad \forall\, y  \in Y_h\,,
\end{align*}
where $\gamma = 1 +\min\{0, d / q -d/p\}  $, and $ p,q\in [1,\infty] $.
\end{lemma}
This result is a special case of~\cite[Thm.~4.5.11]{brenner}.
We use the nodal interpolation operator $\mathcal{I}_h: \mathcal{C}(\ov \Omega ) \ra  Y_h $ such that
\begin{align*}
\mathcal{I}_h ( y) := \sum_{\f z \in \mathcal{N}_h} y (\f z) \varphi _{\f z} \,,
\end{align*}
where $\varphi_{\f z} \in Y_h$ denotes the basis function associated to the 
nodal point $\f z \in \mathcal{N}_h $. 
For functions $y_1, y_2\in\C(\ov\Omega)$, we define mass-lumping
\begin{align*}
\left ( y_1,y_2\right )_h := \int_\Omega \mathcal{I}_h \left ( y_1 y_2\right ) \de \f x = \sum_{\f z \in \mathcal{N}_h} y_1(\f z ) y_2( \f z) \int_\Omega \varphi_{\f z} \de \f x  \,, \qquad \text{ and } \qquad \| y_1 \|_h^2 : = ( y_1, y_1)_h\,.
\end{align*}
For all $ y_1, y_2 \in Y_h$, there exists $C>0$ independent of $h>0$, such that~\cite{raviart, ciavaldini}
\begin{align}
\| y_1 \|_{{\mathbb L}^2} \leq \| y_1 \|_h \leq C \| y_1 \|_{{\mathbb L}^2} \,, 
\qquad \text{and} \qquad | ( y_1 , y_2)_h - ( y_1 , y_2) | \leq C h \| y_1 \|_{{\mathbb L}^2} \| \nabla y_2 \|_{{\mathbb L}^2}  
\,.\label{dismass}
\end{align}
 
 We will need several properties of the interpolation operator.
\begin{lemma}[Interpolation estimate]\label{lem:interpolation}
Let $p\in [1,\infty]$. There exists a constant $C>0$ such that
\begin{align*}
\| y - \mathcal{I}_h( y) \|_{{\mathbb W}^{\ell, p}} \leq{}& C h \| \nabla^{\ell+1} y \|_{{\mathbb L}^p} \quad \forall\, y \in {\mathbb W}^{\ell,p}  \qquad (\ell \in \{0,1\})\, .
\end{align*}
Furthermore, for $1 \leq p<\infty$ and $y \in {\mathbb L}^p$ there holds 
$\| y - \mathcal{I}_h y \|_{{\mathbb L}^p} \ra{} 0$ for $h \ra 0$\,.
\end{lemma}
The result is standard and  can be found for example  in~\cite[Thm.~4.6.19]{brenner}.

\begin{lemma}\label{lem:masslumping}
 Let $p\in [1,\infty )$ and $q \in [1,\infty] $.
There exists a constant $C>0$ such that 
\begin{align*}
\| ( I - \mathcal{I}_h) ( f_h g ) \|_{{\mathbb L}^p} \leq C h  \| f_h \| _{{\mathbb L}^{pq}} \left (  \|\nabla  g \|_{{\mathbb W}^{1,r }}+ \| \nabla g \|_{{\mathbb L}^{s}}\right ) \,,
\end{align*}
for all $f_h\in Y_h$ and $g\in \C^\infty_c(\Omega)$
with $r = \min\{p,  dpq/((p+d)q-d)\}$ and $s = pq/(q-1)$ for $q>1$ and $s = \infty $ for $q=1$.
Note that for $q>1$, we may estimate further $  \| \nabla g \|_{{\mathbb L}^{s}}\leq \|\nabla  g \|_{{\mathbb W}^{1,r}}$. 

\end{lemma}
\begin{proof}

From Lemma~\ref{lem:interpolation}, we observe 
\begin{align*}
\| ( I - \mathcal{I}_h) ( f_h g ) \|_{{\mathbb L}^P} \leq C h^2 \| \nabla^2( f_h g ) \|_{{\mathbb L}^p} = C h^2 \sum_{T\in\mathcal{T}_h}\| \nabla^2( f_h g ) \|_{{\mathbb L}^p(T)} \,.
\end{align*}
The subsequent argumentation will be done of a certain element of the triangulation $\mathcal{T}_h$, summation over all element provides the assertion on the whole domain. 
Due to the product rule, we find $ \nabla ^2 ( f_h g) = \nabla ^2 f _h g +2 \nabla f_h \otimes \nabla  g + f \nabla ^2 g $, where the first term vanishes since $f_h$ is a polynomial of degree $\leq 1 $.
Applying H\"older's inequality for the two remaining terms, we find
\begin{align*}
C h^2 \| \nabla^2( f_h g ) \|_{{\mathbb L}^p(T)} \leq C h^2 \left ( \| \nabla f_h \| _{{\mathbb L}^{pq}(T)} \|\nabla  g \|_{{\mathbb L}^{pq'}(T)} +  \|  f_h \| _{{\mathbb L}^{pqd/(d-pq)}(T)} \|\nabla^2  g \|_{{\mathbb L}^{dpq'/(d+pq')}(T)} \right )
\end{align*}
 where we set $pqd/(d-pq)= \infty$ and $dpq'/(d+pq')=p$ as soon as $pq \geq d$.   
Summing over all elements, standard embeddings, and Lemma~\ref{lem:inv} imply
 \begin{align*}
\| ( I - \mathcal{I}_h) ( f_h g ) \|_{{\mathbb L}^p} \leq{}& C h^2 \| \nabla f_h \| _{{\mathbb L}^{pq}}  \left (  \|\nabla^2  g \|_{{\mathbb L}^{dpq'/(d+pq')}}  + \| \nabla g \|_{{\mathbb L}^{pq'}}\right ) \\\leq{}& C h \|  f_h \| _{{\mathbb L}^{pq}} \left (  \|\nabla^2  g \|_{{\mathbb L}^{dpq'/(d+pq')}}  + \| \nabla g \|_{{\mathbb L}^{pq'}}\right )  \,.
\end{align*}

\end{proof}

By ${\mathcal P}_{{\mathbb L}^2}$, we denote the standard ${\mathbb L}^2$-projection ${\mathcal P}_{{\mathbb L}^2} : {\mathbb L}^2 \ra  \f Z_h$, which is denoted in the same way for matrices. 
The ${\mathbb L}^2$-projection onto the finite element space $Y_h$ and $\f Y_h$ are denoted accordingly, whenever the underlying finite element space will be clear in the context. 
For quasi-uniform meshes ${\mathcal T}_h$, 
the ${\mathbb H}^1$-stability of ${\mathcal P}_{{\mathbb L}^2}$ is well-known (see \textit{e.g.}~\cite{bramblepasciaksteinbach}), and the following error estimate is valid (see \textit{e.g.}~\cite{brenner}),
\begin{align*}
\| \f y - {\mathcal P}_{{\mathbb L}^2} (\f y) \|_{{\mathbb L}^2} \leq  C h^{\ell}  \| \nabla^{\ell} \f y \|_{{\mathbb L}^2} \qquad (\ell =1,2)\,.
\end{align*}
We also use the {projection ${\mathcal P}_h: C(\overline{\Omega}) \rightarrow Y_h$} via
$\bigl( \phi - {\mathcal P}_h (\phi),y)_h = 0$ for all $y \in Y_h$.  We use the
discrete Lapacian $\Delta_h
: [{\mathbb H}^1_0]^d \rightarrow 
 \f Y_h
 $, where 
\begin{equation}\label{disklapl}
(-\Delta_h \f \phi, y) = (\nabla \f \phi, \nabla \f y) \qquad \forall\, \f y \in \f Y_h 
\,.
\end{equation}
For a sum $\f d = \ov{\f d} + \ov{\f d}_1 $, with $\ov{\f d} \in \mathbb H^1_0 $ and $\ov{\f d}_1\in \mathbb W^{2,2} $  we denote accordingly
$\Delta _h  \f d = \Delta_h \ov{\f d}  + \mathcal{P}_{\mathbb L^2} \Delta  \ov{\f d}_1 $, where $\mathcal{P}_{\mathbb L^2}$ denotes the $\mathbb L^2$-projection onto the finite element space $\f Z_h$.

A discrete Sobolev interpolation inequality is a consequence of
\cite[Lemma 4.4]{heywood},
\begin{equation}\label{discrete_sobolev}
\Vert \nabla \f y\Vert_{{\mathbb L}^4} \leq C \Vert \nabla \f y\Vert_{{\mathbb L}^2}^{\frac{4-d}{4}} 
\Vert  \Delta_{ h} \f y\Vert_{{\mathbb L}^2}^{\frac{d}{4}} \qquad \forall\, \f y \in \f Y_h\,.
\end{equation}
This holds for homogeneous Dirichlet boundary conditions, for imhomogeneous Dirichlet boundary conditions, we find in combination with the standard Gagliardo--Nirenberg inequality
\begin{align*}
\Vert \nabla ( \ov{\f d} + \mathcal{I}_h[\ov{\f d}_1])\Vert_{{\mathbb L}^4}\leq{}& \Vert \nabla \ov{\f d}\Vert_{{\mathbb L}^4} + C \Vert \nabla \ov{\f d}_1 \Vert_{{\mathbb L}^4}\\   \leq{}& C \Vert \nabla \ov{\f d}\Vert_{{\mathbb L}^2}^{\frac{4-d}{4}} 
\Vert {\Delta_{ h} \ov{\f d}}\Vert_{{\mathbb L}^2}
^{\frac{d}{4}} 
+ C  \Vert \nabla \ov{\f d}_1 \Vert_{{\mathbb L}^2}^{\frac{4-d}{4}} 
\Vert   \ov{\f d}_1 \Vert_{{\mathbb W}^{2,2}}
^{\frac{d}{4}} 
\qquad \forall\, \ov{\f d} \in \f Y_h \text{ and }  \ov{\f d}_1\in \mathbb{W}^{2,2}\,.
\end{align*}
Additionally, we note that the boundedness of the discrete Laplacian of $\ov{\f d}$ follows from the boundedness of the discrete Laplacian of $\ov{\f d}+\ov{\f d}_1$, \textit{i.e.}, 
\begin{align*}
\Vert {\Delta_{ h} \ov{\f d}}\Vert_{{\mathbb L}^2} \leq{}& \Vert {\Delta_{ h} \ov{\f d}} + \mathcal{P}_{\mathbb L^2} \Delta  \ov{\f d}_1 \Vert_{{\mathbb L}^2 } + \Vert \mathcal{P}_{\mathbb L^2}  \Delta  \ov{\f d}_1 \Vert_{{\mathbb L}^2 }\\ \leq{}& \Vert \Delta_{ h} \f d\Vert_{{\mathbb L}^2 } + \Vert   \Delta  \ov{\f d}_1 \Vert_{{\mathbb L}^2 } \qquad \qquad\qquad\qquad\forall\, \ov{\f d} \in \f Y_h \text{ and }  \ov{\f d}_1\in \mathbb{W}^{2,2}\,
\end{align*}
with $\Delta _h  \f d = \Delta_h \ov{\f d}  + \mathcal{P}_{\mathbb L^2} \Delta  \ov{\f d}_1 $ as defined above.
This will help us to infer some $h$-dependent bound  for the director on the discrete level. 

\section{Continuous system\label{sec:cont}}
 The main obstacle which prevents the construction of a {\em weak solution} ($d=3$) even for a
sub-problem of (\ref{simp}) --- the simplified Ericksen--Leslie equations
(\ref{simp:v})--(\ref{simp:norm}), where we set $\Phi\equiv n^{\pm} \equiv 0$ --- is the  extra elastic stress tensor in the Navier--Stokes equation, \textit{i.e.}, the fourth term in~\eqref{simp:v}.  
This highly nonlinear term is difficult to be identified in the limit for solutions of an approximate scheme, due to limited
regularity estimates.
In~\cite{masswertig}, it was found that using a suitable regularization
 in the equation in order to pass to the limit in the extra elastic stress tensor may lead to undesired error terms coming from the chosen regularization procedure: the oscillatory effects introduced herewith do not vanish in the limit and give rise to an additional defect measure (see also Remark~\ref{rem:measure} below). This is circumvented by approximating the system via a Galerkin approximation with point-wise norm constraint; see also Section \ref{sec:dis}. 
Unfortunately, it still seems not possible to pass to the limit in the weak formulation of the Navier--Stokes-like equation even in this case due to the fact that the extra elastic stress tensor is a nonlinear function of
 $\nabla \f d$, the associated sequence of which only converges weakly. 
However, it is possible to pass to the limit in the {\em dissipative solution} framework given in Section \ref{dissipativesolutions} by only exploiting this weak convergence property of gradients of approximate director fields for the  extra elastic stress tensor: for this solution concept, only weakly lower semi-continuity is needed at this place to retain the relevant {\em relative energy inequality} in Section \ref{convergence1} --- which then settles the construction of a {\em dissipative solution} of (\ref{simp}). 
---
We start this section with a collection of relevant properties of a classical solution of~\eqref{simp}.

\subsection{\textit{A priori} estimates}
\begin{theorem}\label{thm:cont}
Let $T>0$, and $(\f v ,\f d , \Phi, n^{\pm})$ be a classical solution of~\eqref{simp}. Then the following energy equations and norm restrictions are fulfilled for any $0\leq t\leq T$, 
\begin{enumerate}
\item[i)\label{itemi}] energy conservation
\begin{multline}
\frac{1}{2} \left ( \| \f v \|_{{\mathbb L}^2}^2 + \| \nabla \f d \|_{{\mathbb L}^2}^2 + \int_\Omega | \nabla \Phi |_{\varepsilon(\f d)}^2 \de \f x \right ) \Bigg|_0^t \\
+ \int_0^t\left ( \nu \| \nabla \f v \|_{{\mathbb L}^2}^2 
+ \left \| \f d \times \left ( \Delta \f d + \varepsilon_a ( \nabla \Phi \cdot \f d ) \nabla \Phi \right ) \right \|_{{\mathbb L}^2}^2 \right ) \de s 
\\
+\int_0^t \left ( \| n^+-n^- \| _{{\mathbb L}^2}^2 +\int_\Omega (n^++n^-) | \nabla \Phi |_{\varepsilon(\f d )}^2 \de \f x \right ) \de s = 0\,,
\end{multline}
\item[ii)] \label{itemii} charge conservation
\begin{multline*}
\frac{1}{2}\left ( \| n^+  \|_{{\mathbb L}^2}^2+ \| n^- \|_{{\mathbb L}^2}^2\right ) \Bigg|_0^t  + \int_0^t\int_\Omega | \nabla n^+|_{\varepsilon(\f d)}^2 + | \nabla n^-|_{\varepsilon(\f d)}^2 \de \f x \de s \\ 
+\frac{1}{2}\int_0^t \left ( n^+-n^-,[n^+]^2-[n^-]^2\right ) \de s  = 0\,,
\end{multline*}
\item[iii)] \label{itemiii} norm restriction
\begin{align*}
| \f d (\f x ,t ) | = 1 \qquad \text{for a.e.~}(\f x ,t) \in \Omega \times (0,T)\,,
\end{align*}
\item[iv)] \label{itemiiii} maximum principle
\begin{align*}
0\leq  n^{\pm}(\f x ,t) \leq 1  \qquad \text{for a.e.~}(\f x ,t) \in \Omega \times (0,T)\,,
\end{align*}
\item[v)]elliptic regularity
\begin{align}
\| \Phi  \|_{L^\infty(0,T;{\mathbb L}^p)} \leq C  \qquad \text {for some }p>2\,.\label{addregPhi}
\end{align}
\end{enumerate}

\end{theorem}
\begin{proof}
In order to prove the energy equality~\ref{itemi}i), we multiply~\eqref{simp:v} by $\f v$ and integrate over $\Omega$ to obtain
\begin{align}
\frac{1}{2}\frac{\de}{\de t}  \| \f v \|_{{\mathbb L}^2}^2 + \nu \| \nabla \f v \|_{{\mathbb L}^2}^2 + \Bigl( \nabla \f d ^T \bigl( \Delta \f d + \varepsilon_a ( \nabla \Phi \cdot \f d ) \nabla \Phi \bigr) , \f v \Bigr) + \Bigl(( n^+-n^-) \nabla \Phi , \f v \Bigr) =0\,.\label{en:v}
\end{align}
Multiplying~\eqref{simp:d} by $- \Delta \f d - \varepsilon_a (\f d \cdot \nabla \Phi) \nabla \Phi$ and integrating over $\Omega$  gives
\begin{multline}
\frac{1}{2}\frac{\de}{\de t }\| \nabla \f d \|_{{\mathbb L}^2}^2 - \varepsilon_a \Bigl( \t \f d \cdot\nabla \Phi , \f d \cdot \nabla \Phi \Bigr) 
-\Bigl(  ( \f v \cdot \nabla) \f d , \Delta \f d + \varepsilon_a (\f d \cdot \nabla \Phi) \nabla \Phi \Bigr) \\ +  \bigl \| \f d \times \left ( \Delta \f d + \varepsilon_a ( \nabla \Phi \cdot \f d ) \nabla \Phi \right ) \bigr \|_{{\mathbb L}^2}^2 =0\,.\label{en:d}
\end{multline}
Multiplying~\eqref{simp:Phi} by $n^+-n^--\t \Phi$, adding~\eqref{simp:c} multiplied by $\pm\Phi$ and  integrating over $\Omega$, we observe
\begin{multline}
- \Bigl( \t \nabla \Phi, \varepsilon(\f d ) \nabla \Phi \Bigr) + \Bigl( n^+-n^- , \t \Phi \Bigr) + \Bigl( \t (n^+-n^-) , \Phi \Bigr)\\-  \Bigl( \nabla  (n^+-n^- ) ,\varepsilon(\f d ) \nabla \Phi \Bigr)  + \|n^+-n^-\|_{{\mathbb L}^2}^2   \\ - \bigl( n^+-n^- , \f v \nabla \Phi \bigr ) + \Bigl(\nabla [n^+-n^-] ,\varepsilon(\f d ) \nabla \Phi \Bigr) + \int_\Omega (n^++n^-)| \nabla \Phi|_{\varepsilon(\f d )}^2 \de \f x=0 \,.\label{en:Phi}
\end{multline}
By the product formula, 
\begin{align*}
\left ( n^+-n^- , \t \Phi \right ) + \Bigl( \t (n^+- n^-) , \Phi \Bigr) = \t \left ( n^+-n^- ,  \Phi \right ) \,,
\end{align*}
and
\begin{align*}
- \Bigl( \t \nabla \Phi, \varepsilon(\f d ) \nabla \Phi \Bigr) - \varepsilon_a \bigl( \t \f d \cdot\nabla \Phi , \f d \cdot \nabla \Phi \bigr)  = - \frac{1}{2}\t \int_\Omega| \nabla \Phi |_{\varepsilon(\f d )}^2 \de \f x \,.
\end{align*}

Adding the three equations~\eqref{en:v},~\eqref{en:d}, and~\eqref{en:Phi}, we find
\begin{multline}\label{en:last}
\frac{1}{2}\frac{\de}{\de t} \left ( \| \f v \|_{{\mathbb L}^2}^2 + \| \nabla \f d \|_{{\mathbb L}^2}^2 +2 \left ( n^+-n^- ,  \Phi \right )- \int_\Omega | \nabla \Phi |_{\varepsilon(\f d)}^2 \de \f x \right ) \\
+\left ( \nu \| \nabla \f v \|_{{\mathbb L}^2}^2 
+ \left \| \f d \times \left ( \Delta \f d + \varepsilon_a ( \nabla \Phi \cdot \f d ) \nabla \Phi \right ) \right \|_{{\mathbb L}^2}^2 \right )  
\\
+\|n^+-n^-\|^2_{{\mathbb L}^2}+\int_\Omega (n^++n^-)| \nabla \Phi |_{\varepsilon(\f d )}^2 \de \f x  = 0\,.
\end{multline}
Integrating~\eqref{en:last} in time and  using equation~\eqref{simp:Phi}  gives the assertion~i).

In order to prove assertion~\ref{itemii}ii), we multiply equation~\eqref{simp:c} by $n^{\pm} $,  integrate over $\Omega$, to find
\begin{align*}
\frac{1}{2} \frac{\de }{\de t} \left ( \| n^{+}\|_{{\mathbb L}^2}^2+\| n^{-}\|_{{\mathbb L}^2}^2\right )  + \int_\Omega | \nabla n^{+} |_{\varepsilon(\f d) } ^2 +| \nabla n^{-} |_{\varepsilon(\f d) } ^2 \de \f x \\
+ \frac{1}{2}\Bigl (\varepsilon(\f d) \nabla \Phi , \nabla ( n^{+})^2 - \nabla (n^-)^2\Bigr )  = 0 \,.
\end{align*}
Adding~\eqref{simp:Phi} multiplied  by $\frac{1}{2}\bigl((n^+)^2-(n^-)^2\bigr)$ and integrating over $\Omega$, leads via another integration over $ (0,T)$ to the assertion.

The unit norm restriction of the director~\ref{itemiii}iii) is implied by 
multiplying equation~\eqref{simp:d} by $( | \f d |^2 -1)\f d$. Integrating the resulting equation over $\Omega\times (0,T)$, we observe that
\begin{align*}
\frac{1}{4} \int_\Omega \t (|\f d|^2-1 )^2 + (\f v \cdot \nabla )(  | \f d |^2-1 )^2 \de \f x = 0 \,.
\end{align*}
Since the initial value fulfills $| \f d_0|=1$ a.e.~in $\Omega$ and~\eqref{simp:incomp} is valid, we find that $| \f d |=1 $ a.e.~in $\Omega \times (0,T)$.

Standard maximum principles are applied to prove~\ref{itemiiii}iv). 
Indeed, multiplying~\eqref{simp:c} for the positive charges by $(n^+-1)_+$ and for the negative charges by $(n^--1)_+$, and integrating in space, we find
\begin{multline}
\frac{1}{2} \frac{{\rm d}}{{\rm d}t} \left ( \| (n^+-1)_+\|_{{\mathbb L}^2}^2 +  \| (n^--1)_+\|_{{\mathbb L}^2}^2\right ) 
+ \int_\Omega | \nabla ( n^+-1)_+ |_{\varepsilon(\f d)}^2 + | \nabla ( n^--1)_+ |_{\varepsilon(\f d)}^2 \de \f x 
 \\ + \Bigl ( \varepsilon(\f d) \nabla \Phi , n^+ \nabla (n^+-1)_+ - n^-\nabla (n^--1)_+ \Bigr )=0\,.\label{positivityNpm}
\end{multline}
The last term on the left-hand side can be transformed using~\eqref{simp:Phi} to
\begin{multline*}
\Bigl ( \varepsilon(\f d) \nabla \Phi , n^+ \nabla (n^+-1)_+ - n^-\nabla (n^--1)_+ \Bigr )\\= \Bigl( \varepsilon(\f d) \nabla \Phi , \nabla \bigl[ \frac{1}{2}(n^+-1)^2_+ -\frac{1}{2} (n^--1)_+^2 + (n^+-1)_+ - (n^--1)_+\bigr]  \Bigr )
\\=
 \Bigl (n^+-1 -(n^--1),  \frac{1}{2} (n^+-1)^2_+ - \frac{1}{2}(n^--1)_+^2 + (n^+-1)_+ - (n^--1)_+  \Bigr)\geq 0 \,.
\end{multline*}
The inequality follows by observing that the right-hand side may be written as $ (a^+-a^-,f(a^+)-f(a^-))$ with $a^{\pm}= n^{\pm}-1$ and due to the monotony of the function $ f :a \mapsto (1/2)(a)_+^2+(a)_+$. 
Integrating~\eqref{positivityNpm} in time implies due to the condition on the initial condition the upper bound in~\ref{itemiiii}iv).  
Using this $L^\infty$-bound on the charges, we may show their non-negativity.
Multiplying~\eqref{simp:c} by $-(n^{\pm})_-= - max \{ - n^{\pm} , 0\} $, we find 
\begin{align}
\frac{1}{2} \frac{{\rm d}}{{\rm d}t} \| (n^{\pm})_ -\| ^2_{{\mathbb L}^2} + \int_\Omega | \nabla (n^{\pm})_- |_{\varepsilon(\f d)}^2 \de \f x  \pm \left ( \varepsilon(\f d) \nabla \Phi , \frac{1}{2}\nabla \bigl[ (n^{\pm})_-^2\bigr]\right ) = 0 \,.\label{chargepositivity}
\end{align}
For the last term on the left-hand side, we observe due to~\eqref{simp:Phi} and the upper bound on $n^{\mp}$ in~\ref{itemiiii}iv)
\begin{align*}
\pm \Bigl ( \varepsilon(\f d) \nabla \Phi ,  \nabla ( (n^{\pm})_-^2)\Bigr )  =  \left ( ( n^{\pm})_-^2 , n^{\pm}\right ) - \left (  (n^{\pm})_-^2  ,  n^{\mp}\right )\geq  \| ( n^{\pm})_- \|_{{\mathbb L}^3}^3-\| (n^{\pm})_ -\| ^2_{{\mathbb L}^2} \,.
\end{align*}
Reinserting this into~\eqref{chargepositivity} implies
\begin{align*}
\frac{{\rm d}}{{\rm d}t} \| (n^{\pm})_ -\| ^2_{{\mathbb L}^2} \leq \| (n^{\pm})_ -\| ^2_{{\mathbb L}^2} 
\end{align*}
and via Gronwall's inequality the lower bound of~\ref{itemiiii}iv). 

The additional regularity of $\Phi$ follows from elliptic regularity theory (see~\cite[Theorem~1]{RegPhi}).

\end{proof}
\begin{remark}
By the formulation of the equation, it is implied that the mass of the charges is conserved. Indeed, integrating over $\Omega$ the equation~\eqref{simp:c} and using the associated boundary conditions implies that
\begin{align*}
\int_\Omega n^{\pm} (t) \de \f x = \int_\Omega n^{\pm}(0) \de \f x \,.
\end{align*}
Integrating the equation~\eqref{simp:Phi} over $\Omega$ even implies that 
\begin{align*}
\int_\Omega n^+_0 \de \f x = \int_\Omega n^+ ( t) \de \f x = \int_\Omega n^- ( t) \de \f x = \int_\Omega n^-_0 \de \f x \,,
\end{align*}
which is a hidden compatibility condition for the initial values of the charges. 
\end{remark}
We remark that the verification of Theorem~\ref{thm:cont} involves \textit{nonlinear} functions of the classical solution of~\eqref{simp} to be multiplied with~\eqref{simp}, which prevents an immediate corresponding argumentation in Section~\ref{sec:dis}, where a finite element-based space-time discretization is discussed. 

\subsection{Dissipative solutions}\label{dissipativesolutions}
The concept of a dissipative solution heavily relies on the formulation of an appropriate relative energy. This relative energy serves as a natural comparison tool for two different solutions $\f u:=(\f v , \f d, \Phi, n^\pm)$ and $\tu:=(\vv,\dd,\tP,\tilde{n}^\pm)$. 
One possibility to interpret the corresponding {\em relative energy inequality} is as a variation of the energy equality. 
Thus in comparison to weak solutions which fulfill the equation in a generalized sense, the dissipative solution rather fulfills the energy dissipation  mechanism in a weakened sense. 
We decide to use the variation of the energy principle~i) in Theorem \ref{thm:cont}, therefore, the charges are not present in the relative energy. 
It is also possible to derive a relative energy inequality for the energy principle~ii) in Theorem \ref{thm:cont}, but this is not necessary, since~$n^{\pm}$ inherits enough regularity in the limit to perform the calculations to get~ii), and thus the relative energy inequality in the limit rigorously. 
This is also due to the fact that the weak and dissipative solution coincide, if the solution inherits enough regularity to be unique (compare to~\cite{maxdiss}). The charges are regular enough such that the equations~\eqref{simp:c} may be tested with the solution $n^{\pm}$ in the limit.

We introduce the underlying Banach spaces  $\mathbb{X}$ and $\mathbb{Y}$ to denote 
$\f u:=(\f v , \f d, \Phi, n^\pm)\in \mathbb{X}$, if
\begin{subequations}\label{reg:weak}
\begin{align}
\f v \in{}& L^\infty(0,T; {\mathbb H})  \cap L^2(0,T; {\mathbb V})\,,
\\
\f d \in {}& L^\infty(0,T; {\mathbb W}^{1,2})\cap W^{1,2}(0,T;{\mathbb L}^{3/2}) \,, \\ 
\nonumber
\Phi \in{}& L^\infty(0,T;{\mathbb W}^{1,2}/_{\R})
\cap L^\infty(0,T; {\mathbb W}^{1,p}) \quad \text{for some }p>2 \,,\\
n^{\pm} \in{}& L^\infty (0,T;{\mathbb L}^\infty) \cap L^{2}(0,T;{\mathbb H}^1)\cap W^{1,2}\bigl(0,T;({\mathbb H}^1)^*\bigr)\,, \\ \nonumber
\end{align}
\end{subequations}
and $\tu:=(\vv,\dd,\tP,\tilde{n}^\pm) \in \mathbb{Y}$, if $ \tu \in \mathbb{X}$  and additionally
\begin{align}
\begin{split}
\vv \in{}& \C^1 ([0,T];{\mathbb H})\cap L^4(0,T;
{\mathbb L^\infty})\,,
\\
\dd \in {}& L^4(0,T;{\mathbb W}^{1,4})\cap \C^{1}([0,T];\mathbb{W}^{1,2} \cap \mathbb{L}^\infty) \cap L^4(0,T;{\mathbb W}^{2,3}) \cap L^2(0,T; \mathbb{W}^{3,2}) \,, \\ \nonumber
\tP \in{}& L^2(0,T;{\mathbb W}^{2,2} \cap \mathcal{C}^{1}([0,T]; {\mathbb W}^{1,3})\cap L^8(0,T;{\mathbb W}^{1,\infty}) \,,\\
\tilde{n}^{\pm} \in{}& L^{1}(0,T;{\mathbb W}^{1,3}) 
\,.
\end{split}
\end{align}

As will be detailed in Definition \ref{def:diss} below, the space $\mathbb{X}$ will be the solution space, and $ \mathbb{Y}$ the space of test functions. 
The relative energy $\mathcal{R}:\mathbb{X}\times {\mathbb{Y}} \ra \R$  is
 defined 
 for a.e.~$t \in (0,T)$ by 
\begin{align}
\begin{split}
\mathcal{R}(\f u| \tu)={}&   \frac{1}2 \| \nabla ( \f d - \dd)  \|_{{\mathbb L}^2}^2
   + \frac{1}{2} \left \| \f v - \vv \right \| _{{\mathbb L}^2}^2 + \frac{1}{2} \int_\Omega | \nabla (\Phi - \tP) |_{\varepsilon(\f d)}^2 \de \f x 
  \,
  \end{split}
  \label{relEnergy}
\end{align}
and the relative dissipation $\mathcal{W}:\mathbb{X}\times \mathbb{Y} \ra \R$ for a.e.~$t \in (0,T)$ by
\begin{align}
\begin{split}
\mathcal{W}(\f u| \tu ) ={}&
\nu \| \nabla (\f v - \vv) \|_{{\mathbb L}^2}^2 + \| \f d \times  \f q  -\dd \times \tq \|_{{\mathbb L}^2}^2 
\\
&+ \int_\Omega ( n^++n^-)  | \nabla (\Phi - \tP) |_{\varepsilon(\f d )}^2 
\, \de \f x 
+ \left  \| (n^+-n^-)-(\tilde{n}^+-\tilde{n}^-) \right \|_{{\mathbb L}^2}^2
  \, . 
\end{split}\label{relW}
\end{align}
We introduce the potential $\mathcal{K}: \mathbb{Y}\ra \R$ for a.e.~$t \in (0,T)$ via
\begin{align*}
\mathcal{K}(\tu)= C\Big  (& \| \vv \|^4_{{\mathbb L}^\infty} + \| \tq\|^4_{{\mathbb L}^3} + \| \nabla \dd\|_{{\mathbb L}^3}^4
 + \bigl\| \dd\times ( ( \vv \cdot \nabla ) \dd + \tq ) \bigr\|_{{\mathbb W}^{1,3}} \\
 &+ \| \dd\times ( ( \vv \cdot \nabla ) \dd + \tq ) \|_{{\mathbb L}^\infty}  
+ \| \t \dd \|_{L^\infty(\Omega)}^{4/3} + \| \nabla \tP \|_{{\mathbb L}^\infty}^8 \\&
+ \| \nabla ^2  \tP \|_{{\mathbb L}^2} ^{2} + \| \nabla \tilde{n}^- \|_{{\mathbb L}^3} + \| \nabla \tilde{n}^+\|_{{\mathbb L}^3} + \| \t \nabla \tP\|_{{\mathbb L}^3} +1 
\Big )\,,
\end{align*}
 which measures the regularity of the test function $\tu \in \mathbb{Y}$,
and finally the solution operator $\mathcal{A}: \mathbb{Y}\ra \mathbb{X}^*$,
 which incorporates the classical formulation of  system~\eqref{simp} evaluated at the test functions $(\vv ,\dd,\tP,\tn^{\pm})$ by
\begin{align}
\langle \mathcal{A}(\tu), \bullet \rangle  = \left \langle  \begin{pmatrix}
   \t \vv + (\vv \cdot \nabla ) \vv - \nu \Delta \vv + (\nabla \dd )^T \left ( \Delta \dd +\varepsilon \nabla \tP \cdot\dd \nabla \tP \right )  + \nabla \tP (\tn^+-\tn^-)  \\
 \t \dd + (\vv \cdot \nabla ) \dd - \dd \times ( \dd \times \tq) 
\\
\t \tn^{\pm}+ (\vv \cdot \nabla ) \tn^{\pm} - \di \left (\nabla \tn^{\pm} \pm \tn^{\pm} \nabla \tP \right ) 
\end{pmatrix}
, \bullet \right \rangle 
\label{A}\,,
\end{align}
where $ \tq$ is given by $ 
 \tq :={} - \Delta \dd - \varepsilon_a (\nabla \tP \cdot \dd ) \nabla \tP \in L^2(0,T;{\mathbb L}^3)$, and $\tP$ solves $ - \di ( \varepsilon(\f d) \nabla \tP ) = \tn^+-\tn^-$. The mapping $\mathcal{A}: \mathbb{Y}\ra \mathbb{X}^*$  measures `how well' the test function $\tu \in {\mathbb Y}$  `approximately solves'  the problem, 
it is a well-defined mapping, this can be read of the regularity requirements of $\mathbb{X}$ and $\mathbb{Y}$.

\begin{definition}\label{def:diss}
The function 
$\f u:=(\f v , \f d, \Phi, n^\pm)\in \mathbb{X}$ 
is called a dissipative solution to the system~\eqref{simp}, if
there exists a $\f q\in L^2 (0,T; ( \mathbb H^1_0 \cap \mathbb L^{p/(p-2)})^*) $
with
 $\f d \times \f q \in{} L^2(0,T; {\mathbb L}^2) $, where $p$ is given in~\eqref{reg:weak}  such that
 \begin{align}
\int_0^T 
 \left \langle \t n^{\pm} , e^{\pm} \right \rangle
 \de t - \int_0^T \int_\Omega \f v n^{\pm} \cdot \nabla e^{\pm} \de \f x \de t + \int_0^T \int_\Omega \varepsilon(\f d) (\nabla n^{\pm}\pm   n^{\pm}   \nabla \Phi ) \cdot \nabla e^{\pm} \de \f x \de t &{}= 0 \,, \label{weak:c}
\intertext{for all $e ^{\pm} \in L^2(0,T;{\mathbb H}^1) $ and $ n^\pm \in [0,1] $ a.e.~in $\Omega \times (0,T)$, as well as }
- \di \bigl ( \varepsilon (\f d) \nabla \Phi \bigr ) =  n^+-n^- \quad\text{a.e.~in } \Omega \times (0,T)  \,,\label{weak:Phi}
\\
 \t \f d + (\f v \cdot \nabla )\f d -\f d \times (\f d\times  \f q)  = 0 \quad\text{a.e.~in } \Omega \times (0,T)  \,,\label{weak:d}
\end{align}
where  
\begin{align}
\int_0^T \left \langle  \f q , \f b \right \rangle \de t = \int_0^T \int_\Omega \nabla \f d \cdot \nabla \f b -\varepsilon_a   \nabla \Phi (\f d \cdot \nabla \Phi ) \cdot \f b\de \f x \de t\,\label{weak:q}
\end{align}
for all $\f b \in \C^\infty_c( \Omega \times (0,T); {\mathbb R}^d)$. 
The norm restriction is fulfilled almost everywhere,\textit{ i.e.},  $ | \f d ( \f x , t)| = 1$  a.e.~in $\Omega \times (0,T)$, as well as $ \tr(\f d ) = \f d_1 \in H^{1/2}(\partial \Omega)$, 
and the {\em relative energy inequality}
\begin{multline}
\mathcal{R}(\f u | \tu)(t) +\frac{1}{2} \int_0^t \mathcal{W}(\f u| \tu) e^{\int_s^t\mathcal{K}(\tu) \de \tau} \de s \leq   \mathcal{R}(\f u_0| \tu(0))
 e^{\int_0^t \mathcal{K}(\tu) \de s } \\+ \int_0^t \left \langle  \mathcal{A}(\tu) , \begin{pmatrix}
 \vv - \f v \\ 
  \tq -\f q + 
\f A(  \tP) ( \dd - \f d ) \\
\tP-\Phi  
\end{pmatrix} \right \rangle e^{\int_s^t\mathcal{K}(\tu) \de \tau} \de s  \\
+\int_0^t  \left ( \frac{1}{2} \| n^+-\tn^+\|_{{\mathbb L}^2}^2 +\frac{1}{2} \| n^--\tn^-\|_{{\mathbb L}^2}^2 \right )e^{\int_s^t\mathcal{K}(\tu) \de \tau} \de \tau  \,,\label{relencont}
\end{multline}
holds for a.e.~$t\in (0,T)$ and all test functions $\tu\in \mathbb{Y}$, where $ \tr(\dd ) = \f d _1 $, 
\begin{align*}
\f A(\tP) := \left (\varepsilon_a \nabla \tP \otimes \nabla \tP \right )\,, 
\quad
\tq := - \Delta \dd - \varepsilon_a (\nabla \tP \cdot \dd ) \nabla \tP \,, \quad \text{and } - \di \left ( \varepsilon(\dd) \nabla \tP \right ) = \tn^+-\tn^-   \,. 
\end{align*}.

\end{definition}
\begin{theorem}
Let $\Omega \subset \R^d$ for $d= 2,3$ be a bounded convex Lipschitz domain. Let $\bigl(\f v_0,\f d_0, n^{\pm}_0\bigr)
 \in {\mathbb V} \times \bigl[{\mathbb W}^{2,2} \bigr]^3 \times \bigl[ {\mathbb L}^\infty\bigr]^2$, with $|\f d_0 |= 1$ and $n^{\pm}_0\in[0,1]$ a.e.~in $\Omega$ such that $\int_\Omega n_0^+-n_0^-\de \f x = 0$. 
 We additionally assume that there exists a $\bar{\f d}_1\in \mathbb W^{2,2}(\Omega)$ such that $ \tr (\bar{\f d}_1) = \f d _1 = \tr(\f d_0)$. 
Then there exists a dissipative solution according to Definition~\ref{def:diss}. 
\end{theorem}
We are going to prove the theorem by the convergence of a fully discrete, implementable scheme in Section~\ref{sec:dis}. 

\begin{remark}
The variational derivative $\f q$ may be interpreted via
\begin{align*}
\int_0^T \left ( \f d \times   \f q , \f b \right )  \de t = \int_0^T \int_\Omega  \f d \times \nabla \f d \cdot \nabla \f b -\varepsilon_a   \f d \times \nabla \Phi (\f d \cdot \nabla \Phi ) \cdot \f b\de \f x \de t\,,
\end{align*}
since~\eqref{weak:q} may be tested by $ \f b = \f d \times \f h $ for $ \f h \in \C_0^\infty (\Omega \times (0,T))$. Note that in this formulation, all integrals are well defined, since $ \f d \times \f q \in  L^2(0,T; \mathbb L^2 )$ and $\Phi \in L^\infty (0,T ; \mathbb W^{1,p})$, where $p>2$, but $p$ can be arbitrarily close to $2$. 
\end{remark}
\begin{remark}[Continuity in time]
Considering functions $\f u= ( \f v , \f d , \Phi, n^{\pm})\in \mathbb{X}$, we may deduce additional regularity in time. 
From the regularity of $\f d$, we observe by a standard result (see for instance~\cite[Lemma~6]{simon}) that $\f d \in \C_w([0,T];\mathbb{W}^{1,2}) $. Using compact embeddings and the uniform bounds on $
\f d$ implies $ \f d \in \C([0,T]; \mathbb{L}^p)$ for any $p\in [1,\infty)$. 
For the charges we find from the standard embedding $ L^2(0,T;\mathbb{W}^{1,2})\cap W^{1,2}(0,T; (\mathbb{W}^{1,2})^*) \hookrightarrow \C([0,T]; \mathbb{L}^2)$ that $ n^{\pm} \in \C([0,T]; \mathbb{L}^2)$. From the uniform  boundedness, we even observe  $ n^{\pm} \in \C([0,T]; \mathbb{L}^p)$ for any $p\in (1,\infty)$. 
For the electric field, $\Phi$, we may deduce $ \Phi \in \C([0,T]; \mathbb{W}^{1,2}) $ by the following calculation, which employs~\eqref{simp:Phi}: 
\begin{align*}
\int_\Omega | \nabla \Phi (t) - \nabla \Phi(t_n) |^2_{\varepsilon( \f d(t_n))} \de \f x ={}& \Bigl ( \nabla \Phi(t),  \bigl[ \varepsilon \bigl(\f d(t_n) \bigr) - \varepsilon\bigl(\f d (t)\bigr) \bigr]  \nabla \bigl(\Phi( t) - \Phi (t_n)\bigr) \Bigr) \\&  + \Bigl( n^+ ( t) - n^-(t) - \bigl[n^+(t_n ) - n^-(t_n)\bigr], \nabla \bigl(\Phi (t) -  \Phi(t_n)\bigr) \Bigr) 
\\
\leq{}& \frac{1}{2}\int_\Omega | \nabla \Phi (t) - \nabla \Phi(t_n) |^2_{\varepsilon( \f d(t_n))} \de \f x\\&  +C   \| \nabla \Phi(t) \|_{\mathbb{L}^p}^2 \| \varepsilon(\f d(t_n) ) - \varepsilon(\f d (t)) \|_{\mathbb{L}^{2p/(p-2)}} ^2\\& + C\Bigl( \| n^+(t)-n^+(t_n)\|_{\mathbb{L}^2} + \| n^-(t)-n^-(t_n)\|_{\mathbb{L}^2} \Bigr) \,,
\end{align*}
where we used the uniform coercivity of the matrix $\varepsilon$, and that $p$ is given according to~\eqref{reg:weak}. 

We note that we do not claim any continuity in time for the velocity field, since we lack any uniform control on its time-derivative. In the dissipative solution framework, this additional regularity is not needed to give sense to the initial values as in the weak solution framework. 
If we would instead show (additionally) that~\eqref{simp:v} is fulfilled in a measure-valued sense, we would gain the additional regularity $\f v \in \C_w([0,T]; \mathbb{H})$. 
\end{remark}
\begin{remark}[Measure-valued formulation]\label{rem:measure}
Already our initial formulation~\eqref{simp} relies on an inte\-gra\-tion-by-parts formula, which we took from~\cite{noel}.
In our simplified case, for smooth functions, it takes the form
\begin{align}
\begin{split}
\nabla \f d^T \left ( \Delta \f d + \varepsilon_a ( \f d \cdot \nabla \Phi ) \nabla \Phi \right ) + ( n^+-n^-) \nabla \Phi 
={}& \nabla \f d^T \Delta \f d +  \varepsilon_a \nabla \f d ^T  ( \f d \cdot \nabla \Phi ) \nabla \Phi - \di ( \varepsilon(\f d) \nabla \Phi) \nabla \Phi 
\\
={}& \di \left ( \nabla \f d ^T \nabla \f d\right ) - \frac{1}{2}\nabla | \nabla \f d|^2 +  \varepsilon_a \nabla \f d ^T  \nabla \Phi   ( \f d \cdot \nabla \Phi ) \\&- \di \left ( \nabla \Phi \otimes \varepsilon( \f d) \nabla \Phi \right )  + \nabla ^2 \Phi \varepsilon ( \f d ) \nabla \Phi 
\\={}& \di \left ( \nabla \f d ^T \nabla \f d -  \nabla \Phi \otimes \varepsilon( \f d) \nabla \Phi \right )  \\&- \frac{1}{2}\nabla \left ( | \nabla \f d|^2 - | \nabla \Phi |_{\varepsilon(\f d)}^2 \right ) \,,
\end{split}
\label{intbypart}
\end{align}
where the term in the last line can be incorporated into a reformulation of the pressure. 
Using this reformulation of~\eqref{simp:v} allows to show the weak-sequential stability of a measure-valued solution concept,  where the relative energy inequality~\eqref{relencont} of Definition~\ref{def:diss} would be replaced by 
\begin{multline*}
\int_0^T  \nu\left ( \nabla \f v , \nabla \varphi\right ) - \left (  \f v , \t\varphi \right ) - \left ( \f v \otimes \f v , \nabla \varphi \right )  - \left ( \nabla \f d ^T \nabla \f d - \nabla \Phi \otimes \varepsilon( \f d ) \nabla \Phi , \nabla \varphi \right ) - \left \langle \f m , \nabla \varphi \right \rangle \de t \\= \int_\Omega \f v_0 \varphi \de \f x  \,
\end{multline*}
for all $ \varphi \in \C_c^\infty([0,T)) \otimes \pmb{\mathcal{V}}$, and the energy inequality
\begin{multline}
\frac{1}{2} \left ( \| \f v \|_{{\mathbb L}^2}^2 + \| \nabla \f d \|_{{\mathbb L}^2}^2 + \int_\Omega | \nabla \Phi |_{\varepsilon(\f d)}^2 \de \f x + \langle \f m , I \rangle  \right ) \Bigg|_0^t 
+ \int_0^t\left ( \nu \| \nabla \f v \|_{{\mathbb L}^2}^2 
+ \left \| \f d \times \f q  \right \|_{{\mathbb L}^2}^2 \right ) \de s 
\\
+\int_0^t \left ( \| n^+-n^- \| _{{\mathbb L}^2}^2 +\int_\Omega (n^++n^-) | \nabla \Phi |_{\varepsilon(\f d )}^2 \de \f x \right ) \de s \leq  0\,,
\end{multline}
for a.e.~$t\in (0,T)$, 
where $\f m \in L^\infty_{w^*}(0,T;\mathcal{M}( \ov \Omega ; M^{d\times d}_+)$ and $M^{d\times d}_+ $ denotes the set of symmetric semi-positive matrices. We set $\f m (0) =0$. 
For this formulation, one could also gain the additional regularity $\f  v \in W^{1,2}(0,T; (\mathbb{W}^{1,4}\cap \mathbb{V})^*)$ such that $ \f v \in \C_w([0,T];\mathbb{H})$. 

We decided against this formulation, since it seems difficult to numerically keep track of the defect measure $\f m$. Additionally, to show convergence to this formulation, basically requires that an integration-by-parts formula similar to~\eqref{intbypart} holds on the approximate level, as long as structure preserving approximations (complying to a discrete energy principle) are concerned. 
We will rather focus on a scheme that preserves the structure of the continuous system, but where an according integration-by-parts formula is not known to hold.
\end{remark}
\subsection{Relative energy inequality} 
The construction of a {\em dissipative solution} for (\ref{simp}) in Section \ref{convergence1} heavily relies on an (approximate) relative energy inequality for the approximate problem in Section \ref{sec:dis}.
For convenience, we prove the relative energy inequality formally. 
\begin{proposition}[Relative energy inequality]\label{prop:cont}
Let $\f u \in {\mathbb X}$ be a classical  solution of (\ref{simp}).
Then it fulfills Definition~\ref{def:diss}.
\end{proposition}

\begin{proof}
The only thing we have to prove is that $\f u$ fulfills the relative energy inequality~\eqref{relencont}. 
All calculations hold for all $t\in [0,T]$. 
 For any $\tu \in {\mathbb Y}$, we
 decompose the relative energy into the related two energy parts, as well as a mixed part:
\begin{align}
\begin{split}
\mathcal{R}( \f u | \tu) ={}& \frac{1}{2}\left ( \| \f v \|_{{\mathbb L}^2}^2 + \| \nabla \f d \|_{{\mathbb L}^2}^2 + \int_\Omega | \nabla \Phi |_{\varepsilon(\f d)}^2 \de \f x  \right )  +\frac{1}{2}\left ( \| \vv \|_{{\mathbb L}^2}^2 + \| \nabla \dd \|_{{\mathbb L}^2}^2 + \int_\Omega | \nabla \tP |_{\varepsilon(\f d)}^2\de \f x  \right ) \\
&-\bigl ( \nabla \f d, \nabla \dd\bigr) -( \f v , \vv)  
- \bigl( \nabla \Phi , \varepsilon(\f d) \nabla \tP \bigr)
\,.
\end{split}
\end{align}
Similarly, we obtain for the relative dissipation
\begin{align*}
\mathcal{W}(\f u| \tu ) ={}&
\mathcal{W}(\f u | \bf 0 ) + \mathcal{W}(\bf 0|\tu) - 2 \nu \left ( \nabla \f v, \nabla \vv\right ) - 2 \left ( \f d \times \f q , \dd \times \tq\right ) 
\\&-  2 \left (n^+-n^-,\tn^+-\tn^- \right )-
2 \int_\Omega ( n^+-n^-) \nabla \Phi \cdot \varepsilon(\f d )\nabla \tP \de \f x 
  \, . 
\end{align*}

The energy inequality for the solution $\f u $ is given by
\begin{multline*}
\frac{1}{2} \left ( \| \f v \|_{{\mathbb L}^2}^2 + \| \nabla \f d \|_{{\mathbb L}^2}^2 + \int_\Omega | \nabla \Phi |_{\varepsilon(\f d)}^2 \de \f x \right ) \Big|_0^t 
+ \int_0^t\left ( \nu \| \nabla \f v \|_{{\mathbb L}^2}^2 
+ \left \| \f d \times \f q \right \|_{{\mathbb L}^2}^2 \right ) \de s 
\\
+\int_0^t\left ( \| n^+-n^- \| _{{\mathbb L}^2}^2 + \int_\Omega  (n^++n^-) | \nabla \Phi |_{\varepsilon(\f d )}^2 \de \f x \right )\de s \leq 0\,.
\end{multline*}
For the test function, we observe that
\begin{multline*}
\frac{1}{2}\left ( \| \vv \|_{{\mathbb L}^2}^2 + \| \nabla \dd \|_{{\mathbb L}^2}^2 + \int_\Omega | \nabla \tP |_{\varepsilon(\dd)}^2\de \f x  \right )+ \frac{1}{2}\int_\Omega \left (| \nabla \tP |_{\varepsilon(\f d)}^2 - | \nabla \tP |_{\varepsilon(\dd)}^2 \right )\de \f x  \\
+ \int_0^t\left ( \nu \| \nabla \vv \|_{{\mathbb L}^2}^2 
+ \left \| \dd \times\tq \right \|_{{\mathbb L}^2}^2 \right ) \de s 
\\
+\int_0^t\left (  \| \tn^+-\tn^- \| _{{\mathbb L}^2}^2 + \int_\Omega (\tn^++\tn^-)| \nabla \tP |_{\varepsilon(\dd )}^2 \de \f x\right )  \de s =
 \int_0^t \left \langle\mathcal{A}(\tu) , \begin{pmatrix}
 \vv\\ \tq\\ \tP
\end{pmatrix} \right  \rangle \de s \,.
\end{multline*}
Multiplying~\eqref{simp:v} by $\vv$, integrating,  and mimicking the same calculations for the test function tested by $\f v$ implies  
\begin{multline*}
- \bigl( \f v , \vv\bigr) \Big |_0^t = \int_0^t \bigl( ( \f v \cdot \nabla) \f v ,\vv \bigr) + \bigl( ( \vv\cdot\nabla) \vv , \f v \bigr) +2 \nu( \nabla \f v , \nabla \vv )  - \left \langle \mathcal{A}(\tu) , \begin{pmatrix}
\f v \\ 0\\0
\end{pmatrix}\right \rangle  \de s
\\-\int_0^t \bigl ( ( \vv \cdot \nabla) \f d , \f q  \bigr ) + \Bigl ( ( \f v \cdot \nabla ) \dd , \tq \Bigr )- \Bigl ((n^+-n^-) \nabla \Phi , \vv \Bigl ) - \Bigl ( (\tn^+-\tn^-) \nabla \tP , \f v \Bigr )  \de s \,.
\end{multline*}
Multiplying~\eqref{simp:d} by $\tq$, integrating,  and mimicking the same calculations for the test function tested by $\f q$ implies 
\begin{multline}
- \left ( \nabla \f d , \nabla \dd \right ) \Big|_0^t = \int_0^t \bigl ( ( \f v \cdot \nabla) \f d , \tq \bigr) + \left ( (\vv\cdot\nabla) \dd , \f q \right ) + \left ( \f d \times \f q , \f d \times \tq\right ) + \left ( \dd \times \tq , \dd \times \f q \right ) \de s 
\\- \varepsilon _a  \int_0^t \Bigl( \t \f d \nabla \tP , \nabla \tP \cdot \dd \Bigr) + \Bigl( \t \dd \cdot \nabla \Phi , \nabla \Phi \cdot \f d \Bigr) \de s - \int_0^t\left \langle \mathcal{A}(\tu) \begin{pmatrix}
0\\\f q \\0
\end{pmatrix}\right  \rangle \de s \,.
\end{multline}

For the remaining term in the relative energy, we use  equation~\eqref{simp:Phi}, 
\begin{align*}
- \bigl( \nabla \Phi ,\varepsilon(\f d) \nabla \tP \bigr)
={}&-
\left (n^+-n^-  ,  \tP\right ) 
\end{align*}

Multiplying~\eqref{simp:Phi} by $\t \tP$ and adding~\eqref{simp:c} multiplied by $\tP$, as well as mimicking the same for $\tu$ leads to
\begin{align*}
- \Bigl( (n^+-n^-), \tP \Bigr) \Bigg|_0^t ={}&- \int_0^t  \Bigl( \bigl(\f v (n^+-n^-) , \nabla \tP \bigr)+ \bigl(\vv(\tn^+-\tn^-),\nabla \Phi \bigr) \Bigr) \de s 
\\& -\int_0^t \Bigl( \varepsilon(\f d ) \nabla \Phi  , \nabla\t \nabla \tP \Bigr) - \Bigl( \t \bigl( \varepsilon(\dd) \nabla\tP\bigr) , \nabla\Phi \Bigr) \de s \\
& + \int_0^t \Bigl( \varepsilon(\f d ) \nabla (n^+-n^-) ,\nabla \tP \Bigr) + \Bigl(\varepsilon(\dd) \nabla (\tn^+-\tn^-) , \nabla \Phi \Bigr) \de s 
\\& +\int_0^t \Bigl( \varepsilon(\f d) \nabla \Phi  , (n^++n^-)  \nabla \tP \Bigr) +\Bigl( \varepsilon(\dd) \nabla \tP , (\tn^++\tn^-) \nabla \Phi \Bigr)  \de s 
\\
&  - \int_0^t\left \langle  \mathcal{A}(\tu), \begin{pmatrix}
\f 0\\ \f 0\\ \tP
\end{pmatrix}\right \rangle \de s\,.
\end{align*}
For the different terms on the right-hand side, we infer
\begin{align*}
\Bigl( \varepsilon(\f d ) \nabla n^{\pm} , \nabla  \tP \Bigr) +& \Bigl(\varepsilon(\dd) \nabla \tn^{\pm} , \nabla \Phi \Bigr)  
\\={}& \Bigl( \varepsilon(\dd ) \nabla n^{\pm} , \nabla  \tP \Bigr) + \Bigl(\varepsilon(\f d ) \nabla \tn^{\pm} , \nabla \Phi \Bigr) \\&+ \Bigl( (\varepsilon(\f d )-\varepsilon(\dd)) \nabla n^{\pm}, \nabla \tP \Bigr) + \Bigl( \bigl(\varepsilon(\dd)-\varepsilon(\f d) \bigr) \nabla \tn^{\pm} , \nabla \Phi \Bigr) 
\\={}& 
\Bigl( \tn^+-\tn^-,n^{\pm}\Bigr) +  \left ( n^+-n^-, \tn^{\pm}\right )
 \\&+ \Bigl( \bigl(\varepsilon(\f d )-\varepsilon(\dd)\bigr) \nabla (n^{\pm}- \tn^{\pm}) , \nabla \tP \Bigr ) + \Bigl( \bigl(\varepsilon(\dd)-\varepsilon(\f d) \bigr) \nabla \tn^{\pm} , \nabla( \Phi - \tP) \Bigr) \,,
\end{align*}
and
\begin{multline*}
 \Bigl ( \varepsilon(\f d) \nabla \Phi  , (n^++n^-)  \nabla \tP \Bigr) +\Bigl ( \varepsilon(\dd) \nabla \tP , (\tn^++\tn^-) \nabla \Phi \Bigr )
\\
= 2\Bigl ( \varepsilon(\f d) \nabla \Phi  , (n^++n^-)  \nabla \tP \Bigr)  + \Bigl( \varepsilon(\dd) (\tn^++\tn^+)- \varepsilon(\f d) (n^++n^-), \nabla \tP \o\nabla \Phi \Bigr )\, .
\end{multline*}
The next term is
\begin{align*}
 \Bigl ( \varepsilon(\f d ) \nabla \Phi  , \t \nabla \tP \Bigr ) - \Bigl ( \t ( \varepsilon(\dd) \nabla\tP) , \nabla\Phi \Bigr ) 
 ={}&  \Bigl ( (\varepsilon(\f d )- \varepsilon(\dd) ) \nabla \Phi  , \t \nabla \tP \Bigr) - \Bigl ( \t  \varepsilon(\dd)\nabla \tP  , \nabla\Phi \Bigr)
 \\
 ={}& \Bigl ( (\varepsilon(\f d )- \varepsilon(\dd) ) (\nabla \Phi -\nabla\tP) , \t \nabla \tP \Bigr )\\&+\Bigr( (\varepsilon(\f d )- \varepsilon(\dd) ) \nabla\tP , \t \nabla \tP \Bigr) - \Bigl ( \t  \varepsilon(\dd)\nabla \tP  , \nabla\Phi \Bigr)\,.
\end{align*}

From calculating the derivative of
\begin{align*}
 \frac{1}{2} \Bigl (   \varepsilon( \f d ) - \varepsilon(\dd) , \nabla \tP \otimes \nabla \tP  \Bigr )\Bigg|_0^t
={}&  \varepsilon_a \int_0^t  \Bigl(    \t \f d , \nabla \tP (\f d \cdot \nabla \tP) \Bigr) -\Bigl( \t \dd , \nabla \tP (\dd \cdot \nabla \tP) \Bigr ) \de s \\&+\frac{1}{2}\int_0^t\Bigl (  \t \left ( \nabla \tP \otimes \nabla \tP \right ),   \varepsilon( \f d ) - \varepsilon(\dd)    \Bigr )\de s \,,
\end{align*}
we find for the terms incorporating $\varepsilon_a$ that
\begin{align*}
& - \int_0^t\varepsilon_a \Bigl ( \f d \cdot \nabla \Phi , \nabla \Phi \t \dd \Bigr ) + \varepsilon_a \Bigl ( \dd \cdot \nabla \tP, \t \f d \cdot \nabla \tP \Bigr)  \de s+  \frac{1}{2} \Bigl(   \varepsilon( \f d ) - \varepsilon(\dd) , \nabla \tP \otimes \nabla \tP  \Bigr ) \Bigg\vert_0^t 
\\&  -\int_0^t \Bigl( (\varepsilon(\f d )- \varepsilon(\dd) ) \nabla\tP , \t \nabla \tP \Bigr) - \Bigl( \t  \varepsilon(\dd)\nabla \tP  , \nabla\Phi \Bigr)\de s \\
& =- \int_0^t\varepsilon_a \Bigl( \f d \cdot \nabla \Phi , \nabla \Phi \t \dd \Bigr) + \varepsilon_a \Bigl( \dd \cdot \nabla \tP, \t \f d \cdot \nabla \tP \Bigr)   \de s
\\&\quad+  \varepsilon_a \int_0^t  \Bigl(    \t \f d , \nabla \tP (\f d \cdot \nabla \tP)\Bigr) -\Bigl(\t \dd , \nabla \tP (\dd \cdot \nabla \tP) \Bigr) \de s
 +\int_0^t\Bigl( \t  \varepsilon(\dd)\nabla \tP  , \nabla\Phi \Bigr)\de s \\
 &= \varepsilon_a \int_0^t \Bigl( \t\f d - \t \dd, \nabla \tP \left (\nabla \tP \cdot ( \f d -\dd ) \right )\Bigr) + \Bigl( \t \dd ( \nabla \Phi - \nabla \tP ), \nabla \tP \cdot \dd - \nabla\Phi \cdot \f d \Bigr) \de s  \\
 & \quad + \varepsilon_a \int_0^t \Bigl( \t\dd \cdot \nabla \tP , ( \f d - \dd ) \cdot ( \nabla \tP - \nabla \Phi) \Bigr) \de s \,.
\end{align*}

Putting the pieces together, we obtain the inequality
\begin{align*}
\mathcal{R}(\f u| \tu) \Big|_0^t +{}& \int_0^t \mathcal{W}(\f u | \tu) \de s \\
\leq{}& \int_0^t \bigl( ( \f v \cdot \nabla) \f v ,\vv \bigr) + \bigl( ( \vv\cdot\nabla) \vv , \f v \bigl)   \de s
\\&+\int_0^t\bigl ( ( \f v \cdot \nabla) \f d , \tq \bigr ) + \bigl ( (\vv\cdot\nabla) \dd , \f q \bigr) -  \bigl ( ( \vv \cdot \nabla) \f d , \f q  \bigr) - \bigl( ( \f v \cdot \nabla ) \dd , \tq \bigr)   \de s 
\\
&+ \int_0^t   \bigl ( \f d \times \f q , \f d \times \tq\bigr) + \bigl( \dd \times \tq , \dd \times \f q \bigr)- 2  \bigl ( \dd \times \tq , \f d \times \f q \bigr ) \de s 
\\
 &+ \int_0^t \Bigl( (n^+-n^-) \nabla \Phi , \vv \Bigr) + \Bigl( (\tn^+-\tn^-) \nabla \tP , \f v \Bigr) \de s \\
 & - \int_0^t \Bigl( \bigl(\f v (n^+-n^-) , \nabla \tP \bigr)+ \bigl(\vv (\tn^+-\tn^-),\nabla \Phi \bigr) \Bigr)  \de s 
\\
&+ \int_0^t   \Bigl( \bigl(\varepsilon(\f d )-\varepsilon(\dd) \bigr) \bigl (  \nabla (n^+ - n^-)-\nabla (\tn^+-\tn^-) \bigr) , \nabla  \tP \Bigr) \de s 
\\& +\int_0^t \Bigl( \bigl(\varepsilon(\dd)-\varepsilon(\f d) \bigr)\bigl(  \nabla (\tn^+-\tn^-)+ \t \nabla \tP \bigr ) , \nabla (\Phi -\tP) \Bigr ) \de s 
\\&
+\int_0^t\Bigl ( \varepsilon(\dd) (\tn^++\tn^-)- \varepsilon(\f d) (n^++n^-), \nabla \tP \o\bigl(\nabla( \Phi - \tP)\bigr) \Bigr)
 \de s 
\\
&+ \varepsilon_a \int_0^t \Bigl ( \t \dd ( \nabla \Phi - \nabla \tP ), \nabla \tP \cdot \dd - \nabla\Phi \cdot \f d \Bigr ) + \Bigl( \t\dd\cdot \nabla \tP , ( \f d - \dd ) \cdot \nabla (\tP - \Phi) \Bigr) \de s \,.
\\
&+  \int_0^t \varepsilon_a \Bigl ( \t(\f d -  \dd)\cdot \nabla \tP ,   \nabla \tP\cdot ( \f d -\dd ) \Bigr ) + \left \langle \mathcal{A}(\tu), \begin{pmatrix}
\vv-\f v \\\tq-\f q \\\tP-\Phi
\end{pmatrix}\right \rangle \de s \,.
\end{align*}

For the difference of the time-derivatives of $\f d $ and $\dd$, we observe
\begin{align*}
\t \f d -{}& \t \dd + \mathcal{A}(\tu) \cdot ({\bf 0} ,{\bf 1},0)^\top= ( \vv \cdot \nabla) \dd - ( \f v \cdot\nabla)\f d + \f d \times( \f d \times\f q) - \dd \times (\dd \times \tq) \\={}& \bigl(( \vv -\f v ) \cdot \nabla \bigr) \dd - (\f v \cdot\nabla) ( \dd - \f d ) + \f d \times \bigl( \f d \times \f q - \dd \times \tq\bigr) + (\f d - \dd ) \times ( \dd \times \tq)  \,.
\end{align*}
Due to the norm restriction on $\f d$ and the test function $\dd$, we find that $\nabla | \f d|^2 = 0 = \nabla | \dd |^2 $.
Thus, we observe by some manipulations that
\begin{align*}
&\bigl( ( \f v \cdot \nabla) \f d , \tq \bigr) + \left ( (\vv\cdot\nabla) \dd , \f q \right ) -  \left ( ( \vv \cdot \nabla) \f d , \f q  \right ) - \left ( ( \f v \cdot \nabla ) \dd , \tq \right )  
\\
&= \Bigl( \f d \times \bigl[( \f v \cdot \nabla) \f d\bigr] ,\f d \times  \tq \Bigr) + \Bigl(\dd \times ( (\vv\cdot\nabla) \dd ), \dd \times \f q \Bigr) \\
&\quad -  \Bigl( \f d \times (( \vv \cdot \nabla) \f d) ,\f d \times  \f q  \Bigr ) - \Bigl( \dd \times (( \f v \cdot \nabla ) \dd) ,\dd \times  \tq \Bigr)  
\\
&=
\Bigl( \dd \times ( ( \vv \cdot\nabla) \dd , ( \dd -\f d ) \times ( \f q - \tq) \Bigr ) + \Bigl( \f d\times \left ( ( (\f v - \vv) \cdot \nabla) ( \f d - \dd ) \right ) , \f d \times \tq\Bigr)  \\
&\quad + \Bigl( ( \f d -\dd ) \times\left (  (( \f v -\vv) \cdot \nabla )\dd \right ), \f d \times \tq\Bigr) \\ & \quad + \Bigl( \dd \times \bigl[ ( \f v -\vv) \cdot\nabla ) \dd \bigr] , ( \f d - \dd ) \times \tq\Bigr) 
+ \Bigl( \f d \times \bigl[ ( \vv \cdot\nabla ) (\f d - \dd )\bigr], \dd \times \tq-\f d \times \f q\Bigr) \\
&\quad + \Bigl( ( \f d - \dd ) \times \left ( (\vv \cdot \nabla) \dd \right ) , \dd \times \tq - \f d \times \f q \Bigr) \\
&\quad + \Bigl( \f d \times \bigl[ ( \vv \cdot \nabla) ( \f d - \dd )\bigr] , ( \f d - \dd ) \times \tq\Bigr) + \Bigl((\f d - \dd ) \times \left ( ( \vv\cdot\nabla)\dd\right ) , ( \f d - \dd ) \times \tq\Bigr) \,
\end{align*}
and similarly 
\begin{eqnarray*}
&&  \left ( \f d \times \f q , \f d \times \tq\right ) + \left ( \dd \times \tq , \dd \times \f q \right )- 2  \left ( \dd \times \tq , \f d \times \f q \right ) \\
&&\qquad \qquad \qquad   =
  \left ( \f d \times \f q - \dd \times \tq , ( \f d - \dd ) \times \tq \right ) + \left ( \dd \times \tq , ( \dd - \f d) \times ( \f q - \tq )\right ) \,.
\end{eqnarray*}
We use (\ref{weak:q}), and integration by parts to estimate the terms including the difference in $\f q$ and $\tq$.
Indeed, for any smooth enough function $\f a$, we find
\begin{align*}
\left ( \f a , ( \dd-\f d ) \times ( \f q - \tq)\right ) ={}& - \left ( \nabla \f a , ( \dd -\f d ) \times \nabla[\f d - \dd ] \right ) - \left ( \f a \otimes I ,(
  \nabla [\dd - \f d]  )^\top\times  \nabla [\f d - \dd]
  \right ) 
\\{}&+\varepsilon_a \Bigl( \f a ,  ( \dd - \f d ) \times  \left [ \nabla \Phi (\nabla \Phi \cdot \f d) - \nabla \tP ( \nabla \tP \cdot \dd)\right ] \Bigr) \,.
\end{align*}
Additionally, we use integration by parts in the term incorporating the derivative of $n^{\pm}$ to find
\begin{eqnarray*}
&&\Bigl ( \bigl[ \varepsilon(\f d) - \varepsilon(\dd) \bigr]\left [ \nabla ( n^+-n^-) - \nabla (\tn^+-\tn^-)\right ], \nabla \Phi\Bigr ) \\
&&\quad = - \Bigl( ( n^+-n^-) -  (\tn^+-\tn^-), \di \left [ ( \varepsilon(\f d) - \varepsilon(\dd) ) \cdot 
\nabla \Phi\right ] \Bigr ) \\
&&\quad =  -  \Bigl( ( n^+-n^-) -  (\tn^+-\tn^-), \di \left [ (\varepsilon(\f d) - \varepsilon(\dd) \right ] \cdot 
\nabla \Phi\Bigr )  \\
&&\qquad -  \Bigl( ( n^+-n^-) -  (\tn^+-\tn^-), \left [ (\varepsilon(\f d) - \varepsilon(\dd) \right ] : 
\nabla^2 \Phi\Bigr )\,.
\end{eqnarray*}
Note that the boundary terms vanish, since $\f d$ and $\dd$ fulfill the same inhomogeneous Dirichlet boundary condition. 

Inserting this back into the relative energy inequality, we find
\begin{align*}
{}& \mathcal{R}\bigl(\f u| \tu \bigr)(t)  + \int_0^t \mathcal{W}(\f u | \tu) \de s \\
{}&\quad \leq
\mathcal{R}(\f u| \tu)(0)+ 
 \int_0^t \Bigl( \bigl[ (\f v-\vv) \cdot \nabla]( \f v-\vv) ,\vv\Bigr)   \de s
 \\
 &\qquad +\int_0^t \Bigl( ( \f d\times \left [ ( (\f v - \vv) \cdot \nabla) ( \f d - \dd ) \right ] , \f d \times \tq\Bigr )  + \Bigl( ( \f d -\dd ) \times\left[  \bigl(( \f v -\vv) \cdot \nabla \bigr)\dd \right ], \f d \times \tq\Bigr ) 
\de s \\ 
& \qquad +\int_0^t \Bigl ( \dd \times \left [ ( \f v -\vv) \cdot\nabla ) \dd \right ], ( \f d -\dd) \times \tq\Bigr) 
 +
  \Bigl ( \f d \times \left [ ( \vv \cdot\nabla ) (\f d - \dd )\right ], \dd \times \tq-\f d \times \f q\Bigr) \de s 
 \\
 &\qquad +\int_0^t  \Bigl ( ( \f d - \dd ) \times \left [ (\vv \cdot \nabla) \dd \right ] , \dd \times \tq - \f d \times \f q \Bigr )
 + \Bigl ( \f d \times \left [ ( \vv \cdot \nabla) ( \f d - \dd )\right] , ( \f d - \dd ) \times \tq\Bigr ) 
 \de s 
 \\&\qquad + \int_0^t  \Bigl ((\f d - \dd ) \times \left [ ( \vv\cdot\nabla)\dd\right ] , ( \f d - \dd ) \times \tq\Bigr)+  \Bigl ( \f d \times \f q -\dd\times \tq, (\f d-\dd) \times \tq\Bigr )  \de s 
 \\&\qquad -\int_0^t
  \Bigl ( \nabla ( \dd \times (( \vv \cdot \nabla) \dd + \tq) ), ( \dd -\f d ) \times \nabla \left  [\f d - \dd \right ] \Bigr ) \de s 
\\&\qquad -\int_0^t   
   \Bigl( ( \dd \times (( \vv \cdot \nabla) \dd + \tq)) \otimes I ,( \nabla (\dd - \f d)  )^\top \times  \nabla (\f d - \dd)\Bigr) \de s 
\\{}&\qquad +\varepsilon_a \int_0^t \Bigl ( \dd \times (( \vv \cdot \nabla) \dd + \tq) , ( \dd - \f d ) \times \left [ \nabla \Phi (\nabla \Phi \cdot \f d) - \nabla \tP ( \nabla \tP \cdot \dd)\right] \Bigr ) 
  \de s 
\\
 &\qquad + \int_0^t \Bigl (\left ( (n^+-n^-) -(\tn^+-\tn^-)\right ) (\nabla \Phi-\nabla\tP) , \vv \Bigr )\de s \\ &\qquad  + \int_0^t \Bigl ( (n^+-n^-) -(\tn^+-\tn^-) \nabla\tP , \vv -\f v \Bigr ) \de s 
\\&\qquad- \int_0^t   \Bigl (  (n^+ - n^-)- (\tn^+-\tn^-)  , \di \bigl(\varepsilon(\f d )-\varepsilon(\dd) \bigr) \cdot \nabla  \tP + (\varepsilon(\f d )-\varepsilon(\dd)) : \nabla^2 \tP \Bigr ) \de s 
\\&\qquad +\int_0^t \Bigl ((\varepsilon(\dd)-\varepsilon(\f d))\left (  \nabla (\tn^+-\tn^-)+ \t \nabla \tP \right ) , \nabla (\Phi - \tP) \Bigr ) \de s 
\\&\qquad
+\int_0^t\Bigl ( \bigl(\varepsilon(\dd) -\varepsilon(\f d)\bigr)(\tn^++\tn^-)+ \varepsilon(\f d)\left[(\tn^++\tn^-)- (n^++n^-) \right], \nabla \tP \o\nabla (\Phi - \tP) \Bigr)\de s 
\\&\qquad+\int_0^t\left  \langle  \mathcal{A}(\tu), \begin{pmatrix}
\vv-\f v \\\tq-\f q + \f A( \tP) ( \dd -\f d )  \\\tP-\Phi
\end{pmatrix} \right \rangle
 \de s 
\\&\qquad+ \varepsilon_a \int_0^t \Bigl ( \t \dd \cdot \nabla (\Phi -  \tP ), \nabla \tP \cdot \dd - \nabla\Phi \cdot \f d \Bigr) + \Bigl( \t\dd\cdot \nabla \tP , ( \f d - \dd ) \cdot \nabla (\tP -  \Phi) \Bigr) \de s 
\\&\qquad 
 + \int_0^t \Bigl(  \f A(\tP) , \left [\bigl(( \vv -\f v ) \cdot \nabla \bigr) \dd  + \f d \times ( \f d \times \f q - \dd \times \tq)  \right ]\otimes ( \f d - \dd )\Bigr) \de s 
\\&\qquad + \int_0^t  \Bigl (
\f A(\tP) , \left [ (\f d - \dd ) \times ( \dd \times \tq)  \right ]\otimes ( \f d - \dd ) \Bigr ) \de s  -\frac{1}{2}\int_0^t \Bigl (  \left (\f v\cdot\nabla\right )\f A(\tP)  : ( \f d - \dd ) \otimes ( \f d - \dd )\Bigr)
\de s  \,,
\end{align*}
where we defined
\begin{align*}
\f A(\tP) := \left (  \varepsilon_a  \nabla \tP \otimes \nabla \tP \right )\,.
\end{align*}

Estimating the right-hand side in terms of the relative energy and relative dissipation
leads to
\begin{align*}
\mathcal{R}(\f u| \tu) \Big|_0^t +{}\frac{1}{2} \int_0^t \mathcal{W}(\f u | \tu) \de s 
\leq {}&
  \mathcal{R}(\f u_0| \tu(0))+\int_0^t \left \langle \mathcal{A}(\tu), \begin{pmatrix}
\vv-\f v \\\tq-\f q + \f A( \tP) ( \dd -\f d )  \\\tP-\Phi
\end{pmatrix}\right \rangle 
 \de s 
\\&+ \int_0^t \mathcal{K}(\tu) \mathcal{R}(\f u|\tu) \de s +\frac{1}{2} \int_0^t \| n^+-\tn^+\|_{{\mathbb L}^2}^2 + \| n^+-\tn^+\|_{{\mathbb L}^2}^2 \de s 
\end{align*}
such that Gronwall's estimate implies~\eqref{relencont}.

\end{proof}
\section{Fully discrete system}\label{sec:dis}
 The aim is this work is a {\em practical} construction of a {\em dissipative solution} for (\ref{simp}). The related discussion in the introduction of existing schemes \cite{prohl,noel2} for
the simplified Ericksen--Leslie equations
(\ref{simp:v})--(\ref{simp:norm}) as a sub-problem of (\ref{simp})  shows that they are
not practical, since involved discretization and penalization parameters have to {\em independently} tend to zero to construct a related {\em measure-valued solution}. In this respect, a different effort is made in \cite{approx}, where a dissipative solution 
to this sub-problem is constructed with the help of a spatial  discretization, whose iterates satisfy an approximate relative energy inequality, and related director fields are (uniformly) bounded. While this scheme avoids the construction of a solution via (sequences generated by) a (Ginzburg-Landau) penalization approach, and thus is exempted from the above unpractical scenario of admissible parameter choices, it is still not practical, since it is only a semi-discretization in space. In this section, we propose a practical/implementable construction of a {\em dissipative solution} for (\ref{simp}), where related iterates inherit physically relevant properties, including 
a discrete sphere-property for director fields, that approximate concentrations take values in $[0,1]$ only, and a discrete/approximate (relative) energy (in)equality.
 For this purpose, the Scheme \ref{algo1}
below, with solutions in the finite element space
 $${\f U}_h := {\f V}_h \times [{\f Y}_h]^2 
 \times [Z_h] ^2 \times Y_h    \, .$$
uses different numerical tools which make this possible:
\begin{itemize}
\item mass lumping $(\cdot, \cdot)_h$ in equations (\ref{dis:N}) and strongly acute meshes ${\mathcal T}_h$ to validate an $M$-matrix property which ensures that discrete charges take values in $[0,1]$ only,
\item regularizing terms in (\ref{dis:v}) (scaled by $h^{\alpha}$) resp.~(\ref{dis:d}) via (\ref{dis:q}) (scaled by $h^{\beta}$) to limit the spatial variation of discrete velocities
resp.~director fields, and thus allow the $M$-matrix property for equations (\ref{dis:N}); cf.~Theorem \ref{thm:disex} below,
\item mass lumping $(\cdot, \cdot)_h$ in reformulation (\ref{dis:d}) of (\ref{simp:d}) to obtain unit-vector director fields at nodal points of the space-time mesh,
\item  a proper discretization of the (coupling) elastic stress tensor in (\ref{dis:v}) to allow for a discrete energy law for non-trivial fluid-flow velocities; cf.~Theorem \ref{thm:disex} below. This strategy requires the evaluation of (local) averages of
gradients of the discrete director field at nodal points of ${\mathcal T}_h$.
\item 
the introduction of a new variable $\f q$ for the variational derivative of the free energy functional with respect to $\f d$  in order to cope with the nonlinear contribution in this term.
\end{itemize}

 The discrete dynamics starts with initial data
 $\bigl(\f v^{0}, \f d ^{0}, [n^\pm]^{0}\bigr)\in {\f V}_h \times {{\f Z}_h} \times [Z_h]^2
  $, such that
$$\vert \f d^{0}(\f z)\vert = 1 \quad \forall \, \f z \in {\mathcal N}_h\,, \qquad
[n^\pm]^{0} \in [0,1]\qquad \left ( [n^+]^0-[n^-]^0, 1 \right ) = 0\,.  $$
Below, we denote $\f d ^{j-1/2}\  := \frac{1}{2} \bigl(\f d^j + \f d ^{j-1}\bigr)$ for brevity. 
Additionally, we define $ \overline{\f d}^j = \f d^j - \mathcal{I}_h [ \bar{\f d}_1]  $ 
 for $1\leq j \leq  J$ in order to work on linear finite element spaces, where $\bar{\f d}_1$ is the function of the assumptions on the boundary conditions in Theorem~\ref{thm:cont}. 
 Note that the discrete Laplacian is defined according to the Definition at the end of Section~\ref{secc:not}. 
 
  Let $ \alpha \in (0, (6-d)/d)$ and $\beta \in ( 2-(2d)/3, (4-d)^2/d) $. 
\begin{algorithm}\label{algo1} 
Let $\alpha, \beta>0$.
For every $1 \leq j \leq J$, find the solution $\bigl(\f v^{j}, \ov{\f d} ^{j}, {\f q}^{j}, [ n^\pm]^{j}, \Phi^{j} \bigr)\in \f  U_h$
for the given $6$-tuple $(\f v^{j-1}, \ov{\f d} ^{j-1}, {\f q}^{j-1}, [ n^\pm]^{j-1}, \Phi^{j-1})\in \f  U_h$, such
that for all $(\f a, \f b, \f c, e^{\pm}, g)\in  \f  U_h$ holds
\begin{subequations}\label{dis}
\begin{align}
\begin{split}
\left ( d_t \f v ^j, \f a\right  ) + \left ( \nabla \f v^j , \nabla \f a \right ) + h^{\alpha} \left ( \nabla d_t \f v^j , \nabla \f a \right ) + \left ( ( \f v^{j-1} \cdot \nabla ) \f v^j , \f a \right ) + \frac{1}{2} \Bigl ( [\di \f v ^{j-1}] \f v^j , \f a \Bigr )\qquad \\
+ \Bigl(  ([n^+]^j-[n^-]^j )  \nabla \Phi^j , \f a \Bigr) + \left ( \bigl[\PL(\nabla \f d ^{j-1})\bigr]^\top  \Bigl[ \f d ^{j-1/2} \times \bigl( \f d ^{j-1/2}  \times \f q ^{j} \bigr)\Bigr] , \f a \right )_h={}& 0\,,\end{split}\label{dis:v}
\\ {h^{\beta} \bigl( \Delta_h \f d^{j-1/2}, \Delta_h \f b \bigr)} +
 \left ( \nabla \fh d, \nabla \f b \right ) - \varepsilon_a \left ( \nabla \Phi^{j} ( \f d^{j-1} \cdot \nabla \Phi^j), \f b \right )- \left (  \f q^j ,\f b \right ) _h  ={}&0\,,\label{dis:q}\\ 
\left ( d_t \f d^j , \f c \right ) _h+ \Bigl ( \f d ^{j-1/2}\times \bigl[ \PL (\nabla \f d ^{j-1}) \fn v  \bigr] ,  \f d ^{j-1/2} \times \f c \Bigr)_h + \left (\f d ^{j-1/2}\times \f q^j , \f d ^{j-1/2}\times  \f c \right )_h={}&0\,,\label{dis:d}\\
\Bigl(d_t [ n ^{\pm}]^j , e^{\pm} \Bigr)_h  + \Bigl(  \varepsilon(\f d^{j})  \nabla [ n ^{\pm}]^j , \nabla e ^{\pm} \Bigr) 
\pm \Bigl([n^{\pm}]^j  \varepsilon(\f d^{j})  \nabla \Phi^{j} , \nabla e^{\pm} \Bigr)   
- \Bigl ( \f v^j [ n^{\pm}]^j, \nabla e^{\pm} \Bigr) 
={}&0\,,
\label{dis:N}
\\
\Bigl(  \varepsilon(\f d ^{j}) \nabla \Phi ^j , \nabla  g \Bigr)-\Bigl( [ n^+]^j- [n^-]^j, g \Bigr)_h= {}&0  \,.\label{dis:Phi}
\end{align}
\end{subequations}
\end{algorithm}
\begin{remark}
We use  the regularizing terms in~\eqref{dis:v} and~\eqref{dis:q}   to establish the $M$-matrix property of the linear system associated to~\eqref{dis:N}. In the computational studies elaborated in Section~\ref{sec:comp},  the positiveness and boundedness of the charges (which follows form the $M$-matrix property in our analysis) were observed in the experiments even without these additional regularizing terms. Possibly these regularizing terms are only needed in turbulent or oscillating regimes. 
In particular, the $h^\beta$ regularization in~\eqref{dis:q} is used to deduce 
$h$-independent bounds for $\Phi^j$ in stronger norms to validate an
$M$-matrix property in~\eqref{dis:N}.
In the $2d$-case, the $h^\beta$ regularization in~\eqref{dis:q}  is not needed to deduce the $M$-matrix property in~\eqref{dis:N}, but it is essential in the $3d$-case. 
\end{remark}

\subsection{Construction of a solution for Scheme \ref{algo1} via an auxiliary Scheme \ref{balgo2}}\label{discr0b}
We first consider an auxiliary problem (Scheme \ref{balgo2}), for which we show the existence of a $6$-tuple $\bigl(\f v^{j}, \ov{\f d} ^{j}, {\f q}^{j}, [ n^\pm]^{j}, \Phi^{j} \bigr)\in \f  U_h$ via Brouwer's fixed-point theorem; for a (slightly) restricted class of 
space-time meshes, we then show that this $6$-tuple already solves Scheme \ref{algo1}.

For this purpose, we introduce a non-increasing function $\phi \in C^{\infty}\bigl( [0,\infty); [0,1]\bigr)$ which satisfies
\begin{equation}\label{wia1}
\phi = \left \{  \begin{array}{l}
1  \quad \mbox{on } [0,1] \\
0  \quad \mbox{on } [2,\infty) 
\end{array}\,, \right. \qquad \mbox{{\em s.t.}} \qquad 
-\phi' \in \left \{  \begin{array}{l}
[0,2]  \quad \mbox{on } [0,1] \\
\{0\}  \quad \, \mbox{ on } [2,\infty) 
\end{array} \right.\, .
\end{equation}
For every $\gamma >0$, let $\phi_{\gamma}(s) := \phi(\gamma s)$ for all $s \in [0,\infty)$. 

\begin{algorithm}\label{balgo2} 
Fix
 $\alpha, {\beta}, \gamma
>0$.
For every $1 \leq j \leq J$, find the solution $\bigl(\f v^{j}, \ov{\f d }^{j}, {\f q}^{j}, [ n^\pm]^{j}, \Phi^{j} \bigr)\in \f  U_h$
 for the given $6$-tuple $(\f v^{j-1}, \ov{\f d }^{j-1}, {\f q}^{j-1}, [ n^\pm]^{j-1}, \Phi^{j-1})\in \f  U_h$, such
that for all $(\f a, \f b, \f c, e^{\pm}, g)\in  \f  U_h$ holds
\begin{subequations}\label{bdis}
 \begin{align}
\begin{split}
\left ( d_t \f v ^j, \f a\right  ) + \left ( \nabla \f v^j , \nabla \f a \right ) + h^{\alpha} \left ( \nabla d_t \f v^j , \nabla \f a \right ) + \left ( ( \f v^{j-1} \cdot \nabla ) \f v^j , \f a \right ) + \frac{1}{2} \Bigl ( [\di \f v ^{j-1}] \f v^j , \f a \Bigr )\qquad \\ 
+ \Biggl(   \Bigl[ 
\phi_{\gamma } \bigl( \vert [n^+]^j \vert \bigr) [n^+]^j
- 
 \phi_{\gamma } \bigl( \vert [n^-]^j \vert \bigr) [n^-]^j
 \Bigr]\nabla \Phi^j , \f a \Biggr) \qquad \\
+ \Biggl ( \bigl[\PL(\nabla \f d ^{j-1})\bigr]^\top  \Biggl[ 
\f d ^{j-1/2} 
 \times \Bigl(
 \f d ^{j-1/2}  \times \f q ^{j}
  \Bigr)\Biggr] , \f a \Biggr)_h={}& 0\,,\end{split}\label{bdis:v}
\\ {h^{\beta} \left( \Delta_h \f d^{j-1/2}, \Delta_h \f b\right)} +
 \left ( \nabla \fh d, \nabla \f b \right ) - \varepsilon_a \left ( \nabla \Phi^{j} ( \f d^{j-1} \cdot \nabla \Phi^j), \f b \right )- \left (  \f q^j ,\f b \right ) _h  ={}&0\,,\label{bdis:q}\\ 
\left ( d_t \f d^j , \f c \right ) _h+ \Bigl ( \f d ^{j-1/2}\times \bigl[ \PL (\nabla \f d ^{j-1}) \fn v  \bigr] ,  \f d ^{j-1/2} \times \f c \Bigr)_h + \left (
 \f d ^{j-1/2}\times \f q^j
  , \f d ^{j-1/2}\times  \f c \right )_h={}&0\,,\label{bdis:d}\\
\begin{split}
\Bigl(d_t [ n ^{\pm}]^j , e^{\pm} \Bigr)_h  
+  \Bigl( 
\varepsilon
(\f d^j)
 \nabla [ n ^{\pm}]^j , \nabla e ^{\pm} \Bigr) 
- \Bigl ( \f v^j 
\phi_{\gamma }\bigl(\vert [ n^\pm]^j\vert \bigr) [ n^\pm]^j
 , \nabla e^{\pm} \Bigr)\quad \\
\pm \Bigl(
\phi_{\gamma }\bigl(\vert [ n^\pm]^j\vert \bigr) [ n^\pm]^j
  \varepsilon
  (\f d^j)
   \nabla \Phi^{j} , \nabla e^{\pm} \Bigr) 
={}&0\,,
\end{split}
\label{bdis:N}
\\
 \Bigl(
   \varepsilon
   (\f d^j)
   \nabla \Phi ^j , \nabla  g \Bigr)-
 \Bigl(
  [ n^+]^j 
  - 
 [n^-]^j
 , g \Bigr)_h= {}&0  \,.\label{bdis:Phi}
\end{align} 
\end{subequations}
\end{algorithm}
\begin{lemma}\label{exist-lemma}
Let $
k \leq 1/ (8\gamma)
$
, and $k \leq k_0(\Omega)$ and $h \leq h_0(\Omega)$ be sufficiently small. There exists a solution $\bigl(\f v^{j}, \ov{\f d} ^{j}, {\f q}^{j}, [ n^\pm]^{j}, \Phi^{j} \bigr)\in \f  U_h$ of Scheme \ref{balgo2}.
\end{lemma}
\begin{proof}
For every $1 \leq j \leq J$, Scheme \ref{balgo2}  defines a continuous map  
$\pmb{\mathcal F}_j : {\f U}_h \rightarrow 
 {\f U}_h$  in a canonical way,  whose zero is the next iterate $\bigl(\f v^{j}, \ov{\f d} ^{j}, \f q^{j}, [ n^\pm]^{j}, \Phi^{j} \bigr)$; 
 to show its existence, we use Brouwer's fixed-point theorem  in the following form    \begin{equation}\label{coerc1}
\bigl\langle \pmb{\mathcal F}_j(\overline{\f w}), \overline{\f w}\bigr\rangle_{{\f U}_h} \geq 0
\qquad \forall\, \overline{\f w} \in \{\pmb{\phi} \in  {\f U}_h:\, 
\Vert \pmb{\phi}\Vert_{ {\f U}_h} \geq R_j\}\, ,
\end{equation}
for a number
$R_j \geq 0$. The following argumentation establishes this property, but complies
already to the energetic principle~\eqref{energylawex}, by formally choosing
$$\ov{\f w}=\bigl(\f a, \f b, \f c, e^{\pm}, g \bigr) =  \Bigl( \fn v , \td \fn d , \fn q , \Phi^j + \frac{[n^+]^j}{8}  , \frac{[n^-]^j}{8} - \Phi^j ,  - ( [ n^+ ]^j - [n^-]^j) -2  \td \Phi^j - \frac{\Phi^j}{k}  \Bigr)$$
in Scheme \ref{balgo2}. 
%
Testing~\eqref{bdis:v} by $\f v ^j$  implies 
\begin{multline}
\frac{1}{2}d_t \| \f v ^j \|^2_{{\mathbb L}^2} +  \frac{k}{2} \| d_t \f v _j \|^ 2_{{\mathbb L}^2}  + \nu\| \nabla \f v^j \|^2_{{\mathbb L}^2} + \frac{h^\alpha}{2} d_t \| \nabla \f v^j \|^2_{{\mathbb L}^2} + k\frac{h^\alpha}{2} \| \nabla d_t \f v^j \|^2_{{\mathbb L}^2}  
\\- \left ( \f d ^{j-1/2} \times \bigl[ \PL(\nabla  \f d^{j-1})\fn v \bigr] , \f d ^{j-1/2} \times \f q^j \right )_h  \\
+ \Bigl( \phi_{\gamma}\bigl(\vert [ n^+]^j\vert \bigr) [ n^+ ] ^j -\phi_{\gamma}\bigl(\vert [ n^-]^j\vert \bigr)  [ n^-]^j ,  \nabla \Phi ^j \cdot \f v^j \Bigr) 
= 0 \,.\label{disen:v}
\end{multline}
Testing~\eqref{bdis:q} by~$d_t \f d^j $, and adding this equation to~\eqref{bdis:d} tested by~$\f q ^j $ lead to
\begin{multline}
 \frac{h^{\beta}}{2}d_t \| \Delta_h \fn d \| ^2_{{\mathbb L}^2} +
\frac{1}{2}d_t \| \nabla \fn d \| ^2_{{\mathbb L}^2} 
- \varepsilon_a \left ( d_t \f d^j \nabla \Phi^j, \f d^{j-1} \cdot \nabla \Phi^j \right ) \\+ \left ( \f d ^{j-1/2} \times \bigl[ \PL(\nabla  \f d^{j-1})\fn v  \bigr], \f d ^{j-1/2} \times \f q^j \right )_h + \| \f d ^{j-1/2} \times \f q^j \|_h^2 =0\,.\label{disen:d}
\end{multline}
In the following we test~\eqref{bdis:N} by $\pm \Phi^j $ and add~\eqref{bdis:Phi} tested by $-([n^+]^j-[n^-]^j)-d_t \Phi^j 
$:
\begin{multline}
\Bigl ( d_t\left ( [ n^+]^j-[n^-]^j\right ) , \Phi^j \Bigr)_h+ \Bigl( \varepsilon(\f d ^j) \nabla \left ( [ n^+]^j-[n^-]^j\right ), \nabla \Phi^j \Bigr) 
\\- \Bigl(  \phi_{\gamma } \bigl(\vert [ n^+]^j\vert \bigr) [ n^+ ] ^j - \phi_{\gamma } \bigl(\vert [ n^-]^j\vert \bigr) [ n^-]^j ,  \f v^j \nabla \Phi ^j \Bigr) 
\\+ \Bigl ( (\phi_{\gamma } \bigl(\vert [ n^+]^j\vert \bigr) [n^+]^j +\phi_{\gamma } \bigl(\vert [ n^-]^j\vert \bigr) [ n^-]^j) \varepsilon(\f d^{j}) \nabla \Phi^j , \nabla\Phi^j \Bigr)\\ - \Bigl ( \varepsilon(\f d^j) \nabla \Phi^j , \nabla d_t \Phi ^j \Bigr) + \Bigl(  [ n^+]^j-[n^-]^j, d_t \Phi^j \Bigr)_h
\\
+ \| [ n^+]^j-[n^-]^j\|_h^2 - \Bigl( \varepsilon(\f d^j) \nabla \Phi^j , \nabla \left ([ n^+]^j-[n^-]^j\right )\Bigr) = 0\,.  \label{disen:N}
\end{multline}
When adding the three equations~\eqref{disen:v},~\eqref{disen:d}, and~\eqref{disen:N}, we observe that the terms in the second and third line of~\eqref{disen:v} cancel with the first term in the second line of~\eqref{disen:d} and the  two terms in the second line of~\eqref{disen:N}. 
Additionally, the second and the last term in~\eqref{disen:N} cancel.

For the following, it is crucial to observe the integration by parts formula
\begin{eqnarray}\label{starr1}
&&\Bigl( d_t\left ( [ n^+]^j-[n^-]^j\right ) , \Phi^j \Bigr)_h + \Bigl( [ n^+]^j-[n^-]^j, d_t \Phi^j \Bigr)_h  \\ \nonumber
&&\qquad - \Bigl( \varepsilon(\f d^j) \nabla \Phi^j , \nabla d_t \Phi ^j \Bigr)   - \varepsilon_a \Bigl ( d_t \f d^j \nabla \Phi^j, \f d^{j-1} \cdot \nabla \Phi^j \Bigr)  
\\ \nonumber
&&\qquad \qquad = \frac{1}{2} d_t \Bigl( \varepsilon(\f d^j) \nabla \Phi^j , \nabla \Phi^j \Bigr) + k \frac{1}{2} \left (\varepsilon(\f d^{j-1} ) , \left ( d_t \nabla \Phi^j \otimes d_t \nabla \Phi^j  \right ) \right ) + k\frac{\varepsilon_a}{2} \| \nabla \Phi^j \cdot d_t \f d ^j \|^2_{{\mathbb L}^2} \,.
\end{eqnarray}
Indeed, by the standard discrete integration by parts formula and equation~\eqref{bdis:Phi}, we find
\begin{eqnarray*}
&&\Bigl( d_t\left ( [ n^+]^j-[n^-]^j\right ) , \Phi^j \Bigr)_h + \Bigl( [ n^+]^j-[n^-]^j, d_t \Phi^j \Bigr)_h \\  
&&\qquad = d_t \Bigl([ n^+]^j-[n^-]^j , \Phi^j \Bigr)_h+ k \Bigl( d_t \left ( [ n^+]^j-[n^-]^j\right ) , d_t \Phi^j \Bigr)_h
\\
&&\qquad =d_t \Bigl(\varepsilon(\f d ^j) \nabla\Phi^j , \nabla\Phi^j \Bigr) + k \Bigl( d_t \left[ \varepsilon(\f d ^j) \nabla\Phi^j\right ]  , \nabla d_t \Phi^j \Bigr)\,
\end{eqnarray*}
and additionally, we observe
\begin{multline*}
k \Bigl( d_t \left [\varepsilon(\f d ^j) \nabla\Phi^j\right ]  , \nabla d_t \Phi^j \Bigr) - \Bigl ( \varepsilon(\f d^j) \nabla \Phi^j , \nabla d_t \Phi ^j \Bigr)  
= - \left ( \varepsilon(\f d^{j-1}) \nabla \Phi^{j-1} , \nabla d_t \Phi ^j \right )   
\\=\frac{1}{2k} \left [ \left ( \varepsilon(\f d ^{j-1}) \nabla \Phi^{j-1} , \nabla \Phi^{j-1} \right ) -  \left ( \varepsilon(\f d ^{j-1}) \nabla \Phi^{j} , \nabla \Phi^{j} \right ) + k^2  \left ( \varepsilon(\f d ^{j-1}) , d_t \nabla \Phi^j \otimes d_t \nabla \Phi^j \right ) \right]\,.
\end{multline*}
We combine the remaining terms, and recall the definition of $\varepsilon(\f d^{j-1})$ to find
\begin{eqnarray*}
&&d_t \Bigl(\varepsilon(\f d ^j) \nabla\Phi^j , \nabla\Phi^j \Bigr) - \varepsilon_a \left ( d_t \f d^j \nabla \Phi^j, \f d^{j-1} \cdot \nabla \Phi^j \right ) \\
&&\qquad +\frac{1}{2k} \left[ \left ( \varepsilon(\f d ^{j-1}) \nabla \Phi^{j-1} , \nabla \Phi^{j-1} \right ) -  \left ( \varepsilon(\f d ^{j-1}) \nabla \Phi^{j} , \nabla \Phi^{j} \right )\right]  
\\
&&\quad = \frac{1}{2}
d_t  \Bigl(\varepsilon(\f d ^j) \nabla\Phi^j , \nabla\Phi^j \Bigr) +\frac{1}{2k} \left ( \varepsilon(\f d ^{j}) - \varepsilon(\f d ^{j-1})- 2 \varepsilon_a  [\f d ^j - \f d^{j-1}]\o \f d ^{j-1} , \nabla \Phi^{j} \o \nabla \Phi^{j} \right )
\\
&&\quad =\frac{1}{2}
d_t  \Bigl(\varepsilon(\f d ^j) \nabla\Phi^j , \nabla\Phi^j \Bigr)  + k \frac{\varepsilon_a }2 \Bigl( d_t \f d^j \o d_t \f d^j , \nabla \Phi^{j} \o \nabla \Phi^{j} \Bigr)\,.
\end{eqnarray*}
This argumentation settles (\ref{starr1}). We may now use it in (\ref{disen:d}) and
(\ref{disen:N}), and combine the result with (\ref{disen:v}), which leads to
\begin{multline}
\frac{1}{2}d_t \| \f v ^j \|^2_{{\mathbb L}^2} + k \frac{1}{2} \| d_t \f v _j \|^ 2_{{\mathbb L}^2}  + \nu\| \nabla \f v^j \|^2_{{\mathbb L}^2} \\
+ \frac{h^\alpha}{2} d_t \| \nabla \f v^j \|^2_{{\mathbb L}^2} + k\frac{h^\alpha}{2} \| \nabla d_t \f v^j \|^2_{{\mathbb L}^2}  +\frac{1}{2}d_t \| \nabla \f d^j \|^2_{{\mathbb L}^2} 
+
\frac{h^\beta}{2}d_t \| \Delta_h \f d^j \|^2_{{\mathbb L}^2}
 \\+ \Bigl(\left (\phi_{\gamma } \bigl(\vert [ n^+]^j\vert \bigr)  [n^+]^j +\phi_{\gamma } \bigl(\vert [ n^-]^j\vert \bigr)  [n^-]^j\right ) \varepsilon(\f d^{j}) \nabla \Phi^j , \nabla\Phi^j \Bigr ) + \| [ n^+]^j-[n^-]^j\|_h^2 
\\+\frac{1}{2} d_t \Bigl( \varepsilon(\f d^j) \nabla \Phi^j , \nabla \Phi^j \Bigr) + k \frac{1}{2} \left (\varepsilon(\f d^{j-1} ) , \left ( d_t \nabla \Phi^j \otimes d_t \nabla \Phi^j  \right ) \right ) + k\frac{\varepsilon_a}{2} \| \nabla \Phi^j \cdot d_t \f d ^j \|^2_{{\mathbb L}^2} =0\,.\label{energyincrement}
\end{multline}
Note that only the first term in the third line of~\eqref{energyincrement} is not necessarily non-negative. Due to the truncation, it may be absorbed into the first term of the fourth line of~\eqref{energyincrement},
if
\begin{align*}
\left | \Bigl(\left (\phi_{\gamma } \bigl(\vert [ n^+]^j\vert \bigr)  [n^+]^j +\phi_{\gamma } \bigl(\vert [ n^-]^j\vert \bigr)  [n^-]^j\right ) \varepsilon(\f d^{j}) \nabla \Phi^j , \nabla\Phi^j \Bigr ) \right | 
\leq{}& 2 \gamma  \Bigl( \varepsilon(\f d^j) \nabla \Phi^j , \nabla \Phi^j \Bigr)\\ \leq{}& \frac{1}{4k}  \Bigl( \varepsilon(\f d^j) \nabla \Phi^j , \nabla \Phi^j \Bigr)\,
\end{align*}
and thus $ k \leq  1 / (8 \gamma )$. 
In the next step, we
test~\eqref{bdis:N} with $[ n^{\pm}]^j$,
\begin{multline}
\frac{1}{2} d_t \| [ n^{\pm}]^j\|_h^2 +k\frac{1}{2}\| d_t [ n^{\pm}]^j\|_h^2  
+\Bigl( \varepsilon(\f d ^j) \nabla [ n^{\pm}]^j , \nabla [n^{\pm}]^j \Bigr) \\
= \Bigl( \f v^j [ n^{\pm}]^j\phi_{\gamma } \bigl(\vert [ n^{\pm}]^j\vert \bigr), \nabla [ n^{\pm}]^j \Bigr) 
\mp \Bigl( \phi_{\gamma } \bigl(\vert [ n^{\pm}]^j\vert \bigr)[ n^{\pm}]^j \varepsilon(\f d^{j}) \nabla \Phi^j,\nabla [ n^{\pm}]^j\Bigr)\,.\label{Npm}
\end{multline}
Note that
\begin{align*}
&\Bigl |\Bigl (  \phi_{\gamma } \bigl(\vert [ n^{\pm}]^j\vert \bigr) [ n^{\pm}]^j \varepsilon(\f d^{j}) \nabla \Phi^j,\nabla [ n^{\pm}]^j\Bigr)\Bigr |  \leq {}  \frac{1}{4} \Bigl ( \varepsilon(\f d^{j}) \nabla [ n^{\pm}]^j,\nabla [ n^{\pm}]^j\Bigr ) + {4 \gamma^2}\Bigl( \varepsilon(\f d^{j}) \nabla \Phi^j,\nabla \Phi^j\Bigr) \,,
\\
& \Bigl( \f v^j  \phi_{\gamma } \bigl(\vert [ n^{\pm}]^j\vert \bigr) [ n^{\pm}]^j, \nabla [ n^{\pm}]^j \Bigr)  \leq{} 2{\gamma} \| \f v^j \|_{\mathbb L^2}  \| \nabla[n^{\pm}]^j \|_{\mathbb L^2}   \leq  \frac{1}{4} \Bigl( \varepsilon(\f d^{j}) \nabla [ n^{\pm}]^j,\nabla [ n^{\pm}]^j\Bigr) + {8\gamma ^2}
  \| \f v^j \|^2_{{\mathbb L}^2
 }
 \,. 
\end{align*} 
Note that $\varepsilon(\f d^j)$ is a positive definite matrix. 
Adding~\eqref{energyincrement} and~\eqref{Npm} multiplied by $1/(8\gamma ^2)$, we find 
\begin{multline}
\frac{1}{4}d_t \| \f v ^j \|^2_{{\mathbb L}^2} + k \frac{1}{2} \| d_t \f v _j \|^ 2_{{\mathbb L}^2}  + \nu\| \nabla \f v^j \|^2_{{\mathbb L}^2} + \frac{h^\alpha}{2} d_t \| \nabla \f v^j \|^2_{{\mathbb L}^2}\\
 + k\frac{h^\alpha}{2} \| \nabla d_t \f v^j \|^2_{{\mathbb L}^2}  +\frac{1}{2}d_t \| \nabla \f d^j \|^2_{{\mathbb L}^2} + {\frac{h^\beta}{2}d_t \| \Delta_h \f d^j \|^2_{{\mathbb L}^2}}
\\+\frac{1}{2} d_t \Bigl( \varepsilon(\f d^j) \nabla \Phi^j , \nabla \Phi^j \Bigr) + k \frac{1}{2} \left (\varepsilon(\f d^{j-1} ) , \left ( d_t \nabla \Phi^j \otimes d_t \nabla \Phi^j  \right ) \right ) + k\frac{\varepsilon_a}{2} \| \nabla \Phi^j \cdot d_t \f d ^j \|^2_{{\mathbb L}^2} 
\\
+ \frac{1}{16\gamma ^2} d_t \| [ n^{\pm}]^j\|_h^2 +k\frac{1}{16\gamma ^2}\| d_t [ n^{\pm}]^j\|_h^2  
+\frac{1}{32\gamma ^2}\Bigl( \varepsilon(\f d ^j) \nabla [ n^{\pm}]^j , \nabla [n^{\pm}]^j \Bigr) + \| [ n^+]^j-[n^-]^j\|_h^2 
\leq 0\,.\label{energylawex}
\end{multline}
This argumentation implies (\ref{coerc1}), where $R_j$ depends on the data
from the previous iteration.
\end{proof}

This auxiliary result will now be used to validate that a solution of Scheme \ref{balgo2} already solves Scheme \ref{algo1}, provided the space-time mesh satisfies certain criteria. 
\begin{lemma}\label{exist-lemma2}
Suppose {\bf (A1)}, {\bf (A2)}
 for Scheme \ref{balgo2}, where
 \begin{equation}\label{constrb}
 0< \alpha < \frac{6}{d}-1\qquad \mbox{and} \qquad 0< \beta < \frac{(4-d)^2}{d}
  \,,
 \end{equation}
$ \gamma = {C_2}/{2}  h^{\frac{d}{2}} $
for some existing constant $C_2$ independent of the discretization parameters, which are assumed 
to be be sufficiently small, \textit{i.e.}, $k \leq k_0(\Omega)$ and $h \leq h_0(\Omega)$. 
Then, the solution $$\bigl(\f v^{j}, {\ov{\f d}}^{j}, {\f q}^{j}, [ n^\pm]^{j}, \Phi^{j} \bigr)\in \f  U_h$$ of Scheme \ref{balgo2} already solves
 Scheme \ref{algo1}. Moreover,
 $$ \vert \f d^j(\f z) \vert = 1 \quad \forall\, \f z \in {\mathcal N}_h\,, \qquad
 0 \leq [n^{\pm}]^j \leq 1
 \qquad (1 \leq j \leq J)\, .$$
\end{lemma}

\begin{proof}
Let again $1 \leq j \leq J$ be fixed; by Lemma \ref{exist-lemma}, there exists
$ {\f u}^j :=\bigl(\f v^j,\ov{\f d}^{j}, \f q^j, [ n^{\pm}]^j, \Phi^j \bigr) \in  {\f U}_h$
which solves Scheme \ref{balgo2}. 

{\em a)} In particular, there exists $\f d^{j-1/2} \in {\f Z}_h$.
  On choosing $\f d^{j-1/2}(\f z) \varphi_{\f z}$ as test function in (\ref{bdis:d}), where $\varphi_{\f z}$ is the nodal basis function attached to $\f z \in {\mathcal N}_h$ with $\f z \in \Omega$ (\textit{i.e.,} $\f z$ is an inner point), we recover $\vert \f d^{j}(\f z)\vert = 1$ for all inner  points $\f z \in {\mathcal N}_h$. For the boundary points $\f z \in{\mathcal N}_h$ , we immediately  observe that $| \f d ^j (\f z) | = | \mathcal{I}_h[\ov{\f d}_1 ](\f z)| = 1 $, since $ \ov{\f d}^j \in \f Y_h$.

 {\em b)} There exists $C_2 \equiv C_2(\Omega,\varepsilon_a)>0$,  {\em s.t.}
\begin{equation}\label{phip}
\Vert \nabla \Phi^j\Vert_{{\mathbb L}^6
}
 \leq 
C h^{1-d/3}\bigl(1+h^{-
\beta d /(8-2d)
}\bigr)\, .
\end{equation}
{\em b$_1$)} Let $\f d^j_h \equiv \f d^j \in \f Z_h$, and $[ n^{\pm}_h]^j \equiv [ n^{\pm}]^j \in Z_h$. For given $\f d_h^j \in \f Z_h$ and $f^j_h \equiv [n^+_h]^j - [n^-_h]^j\in Y_h$ we consider the solution $\widehat{\Phi}^j_h \in {\mathbb H}^1/{\mathbb R}$ to the elliptic PDE
 \begin{equation}\label{tadis:Phi}
 \Bigl(
   \varepsilon (\f d^j_h)
   \nabla \widehat{\Phi} ^j_h , \nabla  g \Bigr)-
 \bigl(f^j_h, g \bigr)=  0  \qquad \forall\, g \in {\mathbb H}^1\,.
 \end{equation}
 Note that $\f x \mapsto \varepsilon \bigl(\f d^j(\f x) \bigr) \in C \bigl( \Omega; {\mathbb R}^{d \times d}_{\rm sym})$, where by {\em a)}
 $$\forall\, \f x \in \Omega: \qquad \vert \pmb{\xi}\vert^2_{{\mathbb R}^d} \leq \bigl\langle \varepsilon\bigl( \f d_h^j(\f x)\bigr) \pmb{\xi}, \pmb{\xi}\bigr\rangle_{{\mathbb R}^d} \leq (1 + \varepsilon_a)\vert \pmb{\xi}\vert^2_{{\mathbb R}^d}
 \quad \forall\, \pmb{\xi} \in {\mathbb R}^d\, .$$
We use elliptic regularity theory to obtain an estimate for the solution
 of (\ref{tadis:Phi}) in the ${\mathbb H}^2$-norm:
on restating (\ref{tadis:Phi}) in non-divergence form and using {\em a)}, we find a constant $C \equiv C(\Omega)>0$, {\em s.t.}
$$\frac{1}{C} \Vert \Delta \widehat{\Phi}^j_h\Vert_{{\mathbb L}^2} \leq 
 \Vert \nabla \varepsilon(\f d^j_h)\Vert_{{\mathbb L}^
4 
 } \Vert \nabla \widehat{\Phi}^j_h\Vert_{{\mathbb L}^
4
 } + \Vert f^j_h\Vert_{{\mathbb L}^2}\, . $$
Since $\Vert \nabla \widehat{\Phi}^j_h\Vert_{{\mathbb L}^
4}
 \leq C \Vert \nabla \widehat{\Phi}^j_h\Vert^{
 {(4-d)}/{4}
 }_{{\mathbb L}^2} \Vert {
 \Delta \widehat{\Phi}^j_h\Vert^{
 d/4
 }_{{\mathbb L}^2}}$,
  by   {\em a)}, Young's inequality, (\ref{energylawex}), and an inverse estimate,
  \begin{eqnarray}\nonumber
\frac{1}{C} \Vert \Delta \widehat{\Phi}^j_h\Vert_{{\mathbb L}^2} 
&\leq&  \Vert \nabla  \f d^j_h \Vert^{
4/(4-d)
}_{{\mathbb L}^
4
} \Vert \nabla \widehat{\Phi}^j_h\Vert_{{\mathbb L}^2}+ \Vert f^j_h\Vert_{{\mathbb L}^2} \\ \label{h2_esti}
&\leq& Ch^{
-\beta d / (8-2d)
}\Vert \nabla \widehat{\Phi}^j_h\Vert_{{\mathbb L}^2}+ \Vert f^j_h\Vert_{{\mathbb L}^2} \leq C\bigl(1+h^{-
\beta d / (8-2d)
}\bigr) \, .
 \end{eqnarray}

{\em b$_2$)} We  consider the following auxiliary problem, which accounts for the effect of mass lumping of the right-hand side in (\ref{bdis:Phi}): given $(\f d_h^j, [n_h^{\pm}]^j)$, find $\widetilde{\Phi}_h^j \in Y_h$, {\em s.t.}
\begin{equation}\label{tbdis:Phi}
 \bigl(
   \varepsilon (\f d_h^j)
   \nabla \widetilde{\Phi} ^j_h , \nabla  g \bigr)-
 \bigl( f_h^j, g \bigr)= {} 0  \qquad \forall\, g \in Y_h\,.
 \end{equation}
Subtraction of (\ref{tbdis:Phi}) from (\ref{bdis:Phi}), choosing $g = \Phi^j - \widetilde{\Phi}^j_h$, and then
using estimates (\ref{dismass}) and (\ref{energylawex}) leads to 
$$\Vert \nabla (\Phi^j - \widetilde{\Phi}^j_h)\Vert_{{\mathbb L}^2} \leq Ch 
\bigl\Vert  f^j_h \bigr\Vert_{{\mathbb L}^2} \leq Ch\, .$$
By an inverse estimate, we then infer
$\Vert \nabla (\Phi^j - \widetilde{\Phi}^j_h)\Vert_{{\mathbb L}^
6
} \leq  C$.

{\em b$_3$)} Since $\widetilde{\Phi}_h^j \in Y_h$ is the Galerkin projection of 
$\widehat{\Phi}_h^j \in {\mathbb H}^1/{\mathbb R}$, a standard estimate, and
(\ref{h2_esti}) yields 
\begin{equation}\label{esti2}\Vert \nabla (\widehat{\Phi}_h^j - \widetilde{\Phi}_h^j)\Vert_{{\mathbb L}^2} 
\leq  
C h \bigl(1+h^{-
\beta d / (8-2d)
}\bigr)
\, . 
\end{equation}
Putting steps {\em b$_1$)}--{\em b$_3$)} together, an inverse estimate then shows (\ref{phip}), since
\begin{eqnarray}\nonumber \Vert \nabla \Phi^j\Vert_{{\mathbb L}^
6
} 
&\leq& \Vert \nabla (\Phi^j-\widetilde{\Phi}_h^j)\Vert_{{\mathbb L}^
6
} + \Vert \nabla (\widetilde{\Phi}_h^j- 
\widehat{\Phi}_h^j)\Vert_{{\mathbb L}^
6
} + C\Vert \nabla  \widehat{\Phi}_h^j \Vert_{{\mathbb L}^
6}
\\ \nonumber
&\leq& 
C h^{1-d/3}\bigl(1+h^{-
\beta d /(8-2d)
}\bigr)
\end{eqnarray}
{\em c)} In (\ref{bdis:N}), and consequently (\ref{bdis:v}), we have $[n^{\pm}]^j \in [0,1]$, provided that 
   {\bf (A1)} and  {\bf (A2)},
 hold. The proof of this assertion adapts a corresponding argument in \cite[Steps 3 \& 4 in Section 4.1]{numap}; we consider (\ref{bdis:N}) as two {\em linear} problems, where $\varepsilon(\f d^j), \phi_{\gamma}\bigl( \vert [n^\pm]^j\vert\bigr), \f v^j, \Phi^j$ are given. Its {\em algebraic}  representation  then leads to a system matrix ${\mathcal B} \in {\mathbb R}^{2L \times 2L}$ which is an $M$-matrix --- which then validates $[n^{\pm}]^j \in [0,1]$; see {\em c$_1$)} and {\em c$_2$)} below.

{\em c$_1$)} It holds $0 \leq [n^\pm]^j$ ($1 \leq j \leq J$), provided that 
 {\bf (A1)}, {\bf (A2)}, and (\ref{constrb}) hold.
This property follows from the 
$M$-matrix property of ${\mathcal B}$ for the two equations (\ref{bdis:N}) for $1 \leq j \leq J$ fixed (but arbitrary), which is assembled via the nodal basis functions $\{\varphi_{\beta} \}_{\beta=1}^L \subset Y_h$, with entries (${\mathcal M}, {\mathcal K} \in {\mathbb R}^{L \times L}$)
\begin{eqnarray*}
m_{\beta \beta'} &\equiv& \bigl( {\mathcal M}\bigr)_{\beta \beta'} 
:= \bigl( \varphi_{\beta}, \varphi_{\beta'}\bigr)_h\,, \\
k_{\beta \beta'} &\equiv& \bigl( {\mathcal K}\bigr)_{\beta \beta'} := \bigl(  \varepsilon(\f d^j) \nabla \varphi_{\beta}, \nabla \varphi_{\beta'} \bigr)\,, \\
c^{1,\pm}_{\beta \beta'} &\equiv& \bigl( {\mathcal C}_1(\f v^j)\bigr)_{\beta \beta'}
:= - \Bigl( { \phi_{\gamma}}\bigl(\vert [n^{\pm}]^{j}\vert
\bigr)\f v^j \varphi_\beta, \nabla \varphi_{\beta'}\Bigr)\,, \\
c^{2,\pm}_{\beta \beta'} &\equiv& \bigl( {\mathcal C}^{\pm}_2(\Phi^j)\bigr)_{\beta \beta'}:= \pm \Bigl( { \phi_{\gamma}}\bigl( \vert [n^{\pm}]^j\vert\bigr) \varepsilon(\f d^j)\varphi_{\beta} \nabla \Phi^j, \nabla \varphi_{\beta'}\Bigr)\, .
\end{eqnarray*}
Hence, the entries $b^{\pm}_{\beta \beta'} \equiv ({\mathcal B})^{\pm}_{\beta \beta'}$ of the system matrix ${\mathcal B} = {\rm diag}[{\mathcal B}^+, {\mathcal B}^-]$ read
\begin{equation}\label{mmatr}
{\mathcal B}^{\pm} := \frac{1}{k} {\mathcal M} + {\mathcal K} + {\mathcal C}_1^{\pm}(\f v^j) + {\mathcal C}_2^{\pm}(\Phi^j)\, ,
\end{equation}
and $[n^{\pm}]^{j} := \sum_{\beta=1}^L x_{\beta}^{\pm} \varphi_{\beta}$
solves
$$\begin{pmatrix}
{\mathcal B}^+ & \f 0 \\
\f 0 & {\mathcal B}^-
\end{pmatrix}
\begin{pmatrix}
\f x^{+} \\
\f x^{-}
\end{pmatrix}
 = 
 \begin{pmatrix}
 [\f f^+]^{j-1} \\
 [\f f^-]^{j-1}
 \end{pmatrix}
\,, $$
where $\f x^{\pm} = (x^{\pm}_1,\ldots, x^{\pm}_L)^\top$, and
$[\f f^\pm]^{j-1}_l := \frac{1}{k} \bigl([n^{\pm}]^{j-1}, \varphi_l\bigr)_h$.
{Note that the matrix ${\mathcal K}$ is an $M$-matrix, since its defining properties (see {\em c$_{11})$--c$_{13})$} below) --- which follow from an element-wise consideration --- are a consequence of 
$$k_{\beta \beta'} = \sum_{\{ K \in {\mathcal T}_h:\,  {{\rm supp} \, \varphi_{\beta} \cap  {\rm supp}\,  \varphi_{\beta'}} \cap K \neq \emptyset\}} \Biggl\langle \Bigl(\int_{K} 
\varepsilon(\f d^j)\, {\rm d}\f x\Bigr) \nabla \varphi_{\beta}, \nabla \varphi_{\beta'} \Biggr\rangle_{{\mathbb R}^{d \times d}}\,,$$
since $\nabla \varphi_{\beta} \in {\mathbb R}^d$ is element-wise constant,
and the minimum resp.~maximum eigenvalue of the symmetric, positive definite-valued function $\f x \mapsto \varepsilon\bigl(\f d^j(\f x)\bigr)$ is bigger resp.~less than $1$ resp.~$1+\varepsilon_a$.} We now guarantee its dominating influence as part of ${\mathcal B}^{\pm}$ via a dimensional argument --- and hence $M$-matrix property of ${\mathcal B}^{\pm}$:
\begin{enumerate}
\item[{\em c$_{11}$)}] Non-positivity of off-diagonal entries of ${\mathcal B}^{\pm}$, {\em i.e.}, 
$\bigl({\mathcal B}^{\pm}\bigr)_{\beta \beta'} \leq 0$ for all $\beta \neq \beta'$. Since
${\mathcal T}_h$ satisfies {\bf (A1)}, there exists $\overline{C}_{\theta_0}>0$, such that
$k_{\beta\beta'} \leq -\overline{C}_{\theta_0} h^{d-2} < 0$ uniformly for $h>0$, for any pair of adjacent nodes. The remaining parts of $\bigl({\mathcal B}^{\pm}\bigr)_{\beta \beta'}$  will be bounded independently, and we start with $c_{\beta \beta'}^{1,\pm}$:
 on using an embedding property, $\Vert \f v^j\Vert_{{\mathbb L}^{6}} \leq C \Vert \nabla \f v^j\Vert_{{\mathbb L}^2}^{d/3} \Vert \f v^j \Vert_{\mathbb L^2}^{(3-d)/3}$ 
  and (\ref{energylawex}), we conclude
\begin{equation}\label{nonpos1}\Bigl\vert \Bigl( { \phi_{\gamma}}\bigl( \vert [n^{\pm}]^j\vert\bigr) \f v^j \varphi_{\beta},
\nabla \varphi_{\beta'} \Bigr)\Bigr\vert 
\leq \Vert \f v^j\Vert_{{\mathbb L}^{
6
}} \Vert \varphi_{\beta} \nabla \varphi_{\beta'}\Vert_{{\mathbb L}^
6/5
} \leq C_2 h^{5d/6-1} C
 h^{
-\alpha d / 6 
}\, .
\end{equation}
We use a dimensional (asymptotic) argument, which ensures that this term may be bounded by $\overline{C}_{\theta_0} h^{d-2}$ --- and thus may be controlled by the corresponding negative term in ${\mathcal K}$: we find that 
\begin{equation}\label{constr1}
d-2 < \frac{5d}{6} - 1 - \frac{\alpha d }{ 6 }\qquad \Longrightarrow \qquad   \alpha  <  \frac{6}{d} -1 
\end{equation}
validates this requirement. 

Below, we use $C \equiv C(\Omega)>0$. We proceed similarly with $c_{\beta \beta'}^{2,\pm}$, utilizing (\ref{phip}) 
\begin{eqnarray}\nonumber
\Bigl\vert \bigl( \varphi_{\beta} \nabla \Phi^j, \nabla \varphi_{\beta'}\bigr)\Bigr\vert 
&\leq& \Vert  \nabla \Phi^j\Vert_{{\mathbb L}^
6
} \Vert \varphi_\beta \nabla \varphi_{\beta'}\Vert_{{\mathbb L}^{
6/5
 }} \leq 
 C h^{1-d/3}(1+h ^{- (\beta d )/(8-2d) } )h^{5d/6-1} 
\label{nonpos2}
\end{eqnarray}
By the same dimensional argument as below (\ref{nonpos1}), we deduce
\begin{equation}\label{constr10}
d-2 < \frac{5d}{6}-1 - \beta \frac{d}{8-2d} + 1- \frac{d}{3}\qquad \Longrightarrow \qquad \beta < \frac{(4-d)^2}{d}
\end{equation}
Such that, we have  $\vert k_{\beta \beta'}\vert > \vert c_{\beta \beta'}^{2,\pm}\vert$, for $h$ small enough. 

Finally, non-positivity of off-diagonal entries of ${\mathcal M}$ holds. Therefore, 
off-diagonal entries of ${\mathcal B}^{\pm}$ are non-positive, if
$h \leq h_0(\Omega)$ is small enough and (\ref{constr1}) holds.
\item[{\em c$_{12}$)}] Strict positivity of the diagonal entries of ${\mathcal B}^{\pm}$: we have to verify that
$$\frac{1}{k} m_{\beta \beta} + k_{\beta \beta} + c^1_{\beta \beta}  + c^{2,\pm}_{\beta \beta} > 0\, .$$
By {\bf (A1)}, we know that there exists $\underline{C}_{\theta_0}>0$, such that
$$\frac{1}{k} m_{\beta \beta} \geq \underline{C}_{\theta_0}h^d\,, \qquad \mbox{and} \qquad
k_{\beta \beta} \geq \underline{C}_{\theta_0}h^{d-2}\, .$$
Moreover, from (\ref{nonpos1}) and (\ref{nonpos2}), we conclude
$$\vert c^{1,\pm}_{\beta \beta}\vert + \vert c^{2,\pm}_{\beta \beta} \vert \leq 
C_2 h^{5d/d-1} C  h^{-\alpha d /6 
} + 
 Ch^{1-d/3}
(1+h ^{- \beta d/(8-2d) } )h^{5d/6-1} 
=: \eta(h)\, .$$
Hence, $\underline{C}_{\theta_0} h^{d-2} - \eta(h)  >0$ is valid by the same dimensional argument as in {\em c$_{11}$)} provided (\ref{constr1}) and  (\ref{constr10})
are valid.
\item[{\em c$_{13}$)}] ${\mathcal B}^{\pm}$ is strictly diagonal dominant, {\em i.e.}, 
$\sum_{\beta' \neq \beta} \vert ({\mathcal B}^{\pm})_{\beta \beta'}\vert < \vert ({\mathcal B}^{\pm})_{\beta \beta}\vert$ for all $1 \leq \beta \leq L$. We use the fact that the number of neighboring nodes $\f x_{\beta'} \in {\mathcal N}_h$ for each $\f x_{\beta}$ is bounded independently of $h>0$, and that this property is inherited from ${\mathcal K}$. Hence, there exists a constant
$\overline{C} := \overline{C} \bigl( \{ \# \beta':\, k_{\beta \beta'} \neq 0\}\bigr)>0$, such that, thanks to $c_{11}), c_{12})$ and for $k \leq k_0(\Omega)$ and
$h \leq h_0(\Omega)$ sufficiently small,
\begin{eqnarray*}
({\mathcal B}^{\pm})_{\beta \beta} &\geq& \frac{1}{k} \underline{C}_{\theta_0} h^d + \underline{C}_{\theta_0} h^{d-2} -
\eta(h)  > \overline{C} \max_{\beta' \neq \beta} \vert ({\mathcal B}^{\pm})_{\beta \beta'} \vert\\
&=& \overline{C} \bigl\vert -\overline{C}_{\theta_0} h^{d-2}-\eta (h)\bigr\vert
\end{eqnarray*}
\end{enumerate}
The properties $c_{11}) -c_{13})$ then guarantee the $M$-matrix property of
${\mathcal B}^{\pm}$ for $k \leq k_0(\Omega)$ and
$h \leq h_0(\Omega)$  small enough, under the given mesh constraints: this
property then implies $0 \leq [n^{\pm}] ^j \in Y_h$  via the discrete maximum principle.

{\em c$_2$)} It holds $ [n^\pm]^j \leq 1$ ($1 \leq j \leq J$), provided that 
{\bf (A1)},  {\bf (A2)}, and (\ref{constrb})  hold. 

First, we identify that $ \phi_{\gamma } ( | [n^{\pm}]^j|) =1$. 
From the bound~\eqref{energylawex} multiplied by $k$, we infer a bound on the $\mathbb{L}^2$-norm of the charges $ [n^{\pm}]^j$ such that an inverse estimate helps to conclude 
\begin{align*}
\| [n^{\pm}]^j\|_{\mathbb L^\infty} \leq C h^{-d/2} \| [n^{\pm}]^j \|_{\mathbb L^2 } \leq C h^{-d/2}  \| [n^{\pm}]^j \|_ h \leq C_2  h^{-d/2} \,.
\end{align*}
Choosing $\gamma^{-1} = \frac{{C_2}}{2} \bigl(h^{-\frac{d}{2}}\bigr)$ yields $
 \phi_{\gamma }
 \bigl(\vert [\widetilde{n}^{\pm}]^j \vert \bigr)=1$.

By induction,  we may
assume that $[n^{\pm}]^{j-1} \leq 1$ for some fixed $1 \leq j \leq J$.  

In the following, we consider the system which is solved by $ ( \mathcal{I}_h {1} - [ n^+]^j, \mathcal{I}_h 1 - [ n^-]^j )^T$. 
First, we observe that 
\begin{align*}
\left ( \td \mathcal{I}_h [1] , e^{\pm}\right )_h = 0 
\,.
\end{align*}
Secondly, for the convection term, we use integration by parts, the definition of $\f V_h$, and  {\bf (A2)}  to conclude
\begin{align*}
 (  \f v^j , 
 \nabla e^{\pm}    )  ={}&  -\bigl( {\rm div}\, \f v^j \, ,
 [{e}^{\pm}]^j - \lambda\bigr) = 0\,,
\qquad \text{where} \qquad \lambda ={}  \frac{1}{\vert \Omega\vert}
\int_{\Omega} e^{\pm} 
\, {\rm d}\f x\, .
\end{align*}
Thirdly, we find 
\begin{align*}
\pm \left ( \varepsilon(\f d^j) \nabla \Phi , \nabla e^{\pm} \right ) = \pm  \left ( [ n^{+}]^j - [n^-]^j, e^{\pm} \right )_h 
\end{align*}
Combining these equations, we observe that 
$ ( \mathcal{I}_h [1] - [ n^+]^j, \mathcal{I}_h [1] - [ n^-]^j )^\top$
 solves
\begin{align*}
\left ( \td ( \mathcal{I}_h  [1] - [ n^+]^j ), e^+ \right ) _h &+{} \left (  \td ( \mathcal{I}_h [1] - [n^-]^j), e^- \right ) _h \\&+ \left ( \varepsilon( \f d^j ) \nabla ( \mathcal{I}_h  [1] - [ n^+]^j ), \nabla e^+ \right ) + \left ( \varepsilon( \f d^j ) \nabla ( \mathcal{I}_h  [1] - [ n^-]^j ), \nabla e^- \right )  
\\
&- \left ( v^j (\mathcal{I}_h [1] - [n^+]^j ) , \nabla e^+\right ) - \left ( v^j ( \mathcal{I}_h [1] - [n^-]^j), \nabla e^- \right ) \\&+ \left ( \varepsilon(\f d^j ) \nabla \Phi ( 1- [n^{+}]^j ) , \nabla e^{+}\right ) - \left (  \varepsilon(\f d^j ) \nabla \Phi ( 1- [n^{-}]^j ) , \nabla e^{-}\right ) 
\\
&+ \left ( \mathcal{I}_h[1] - [n^+]^j - ( \mathcal{I}_h[1] - [n^-]^j ) , e^+\right )_h \\& + \left ( \mathcal{I}_h [1] - [n^-]^j- ( \mathcal{I}_h [1] - [n^+]^j ) , e^- \right )_h =  \left ( \varepsilon(\f d ^j ) \nabla \mathcal{I}_h[ 1]  , \nabla (e^{+}+e^-) \right )
\end{align*}

Going back to~\eqref{mmatr}, we observe that
 $(\mathcal{I}_h[1]-n^{\pm}]^{j}) := \sum_{\beta=1}^L x_{\beta}^{\pm} \varphi_{\beta}$
solves
\begin{align}
\begin{pmatrix}
{\mathcal B}^+  + \mathcal{M}& \f -\mathcal{M} \\
\f -\mathcal{M} & {\mathcal B}^-+ \mathcal{M}
\end{pmatrix}
\begin{pmatrix}
\f x^{+} \\
\f x^{-}
\end{pmatrix}
 = 
 \begin{pmatrix}
 [\f f^+]^{j-1}+ \mathcal{K}\f 1 \\
 [\f f^-]^{j-1}+ \mathcal{K}\f 1
 \end{pmatrix}
\,, \label{matrixsys}
\end{align}
where $\f x^{\pm} = (x^{\pm}_1,\ldots, x^{\pm}_L)^\top$, $ \f 1 =  (1,\ldots, 1)^\top$, and
$[\f f^\pm]^{j-1}_l := \frac{1}{k} \bigl((\mathcal{I}_h [1]- [n^{\pm}]^{j-1}), \varphi_l\bigr)_h$.
We already proved that $\mathcal{B}^{\pm}$ are \textit{M}-matrices and since $\mathcal{M}$ is a diagonal matrix with positive entries due to the mass lumping, we deduce that the system matrix is also an $M$-matrix. The right-hand side remains positive, since $\mathcal{K}$ is an $M$-matrix (this term even vanishes, since constant functions are in the kernel of $\mathcal{K}$) such that $[n^{\pm}]\leq 1$.

\medskip

The arguments in {\em a)}--{\em c)} now show that
$ {\f u}^j \in  {\f U}_h$ solves Scheme \ref{algo1}. By choosing $e^{\pm} =1$ in~\eqref{bdis:N}, we observe that the mass of the iterates $[n^{\pm}]^j$ is conserved. 
\end{proof}
We combine the assumptions of Lemma~\ref{exist-lemma} and Lemma~\ref{exist-lemma2} to:
\begin{itemize}
\item[{\bf (A3)}]  
An admissible step size $(k, h)$  satisfies 
$k \leq C h^{d/2}$.
\end{itemize}
\begin{remark}
The proof of the assertion~\textit{c$_2$)} of the previous Lemma follows a different argument than the associated proof in~\cite{numap}. By considering the system matrix~\eqref{matrixsys} and showing its $M$-matrix property, we can eliminate the previous $k$ and $h$ coupling for this part of the proof. 
The remaining step-size assumption \textbf{(A3)} is only needed to guarantee the existence of solutions. It can probably be improved and was not observed in the numerical computations. 
\end{remark}


\subsection{The structure-inheriting Scheme \ref{algo1}}\label{discr0a}
Scheme \ref{algo1} was designed to inherit key properties of system (\ref{simp}); while such a scheme is of independent relevancy, these properties will be crucial in later sections to construct a dissipative solution of (\ref{simp}) via (proper sequences of)  solutions of Scheme \ref{algo1} in the limit of vanishing discretization parameters.
Below, we use the discrete energy
$$ E (\f v^j,\f d^j , \Phi^j) : = \frac{1}{2}\| \f v^j \|^2_{{\mathbb L}^2} + \frac{1}{2}\| \nabla \fn d  \|^2_{{\mathbb L}^2} + \frac{1}{2} \Bigl( \varepsilon(\f d ^j ) \nabla \Phi^j , \nabla \Phi ^j \Bigr)\, . $$
\begin{theorem}\label{thm:disex}
Let $\Omega \subset \R^d$ for $d= 2,3$ be a bounded convex Lipschitz domain. 
 We additionally assume that there exists a $\bar{\f d}_1\in \mathbb W^{2,2}(\Omega)$ such that $ \tr (\bar{\f d}_1) = \f d _1 $. 
 Suppose {\bf (A1)}, {\bf (A2)}, and \textbf{(A3)}.
 Assume 
$k\leq k_0(\Omega)$ and $h\leq h_0(\Omega)$ to be sufficiently small. 
For every $1 \leq j \leq J$, there exists a
 solution $\bigl\{ (\f v^{j}, \ov{\f d} ^{j}, {\f q}^{j}, [n^\pm]^{j}, \Phi^{j});\, 1 \leq j \leq J\bigr\} \subset {\f U}_h$ of Scheme \ref{algo1}, with the following properties: 
\begin{enumerate}
\item[i)] conservation of properties: For all $1 \leq j \leq J$,
 \begin{align*}
0\leq  [n^{\pm}]^j  \leq 1  \,,
\qquad \mbox{and} \qquad \vert \f d^j(\f z)\vert = 1 \quad \forall\, \f z \in {\mathcal N}_h\, .
\end{align*}

\item[ii)] discrete energy equality: For all $1 \leq j \leq J$, \begin{multline}
E \bigl(\f v^j,\f d^j , \Phi^j\bigr) +  \frac{h^\alpha}{2}\| \nabla \f v^j \|^2_{{\mathbb L}^2}
+ \frac{h^\beta}{2} \| \Delta_h \fii d \|_h^2 
\\ +\frac{ k^2}{2} \sum_{\ell=1}^j \Bigl[ \| d_t \f v^\ell \|^2_{{\mathbb L}^2}+ h^\alpha \| d_t \nabla \f v^\ell \| ^2_{{\mathbb L}^2} 
+ \Bigl(\varepsilon(\f d^{\ell-1})\nabla d_t \Phi^\ell , \nabla d_t \Phi^\ell \Bigr) + \varepsilon_a \| \nabla \Phi^\ell \cdot d_t \f d^\ell \|^2_{{\mathbb L}^2}   \Bigr]
\\
+ k \sum_{\ell=1}^j \Bigl[ \nu \| \nabla \f v^\ell \|^2_{{\mathbb L}^2} + \| \f d ^{\ell-1/2} \times \f q^\ell \|^2_h + \Bigl( ([n^+]^\ell+[n^-]^\ell) \nabla \Phi^\ell , \varepsilon(\f d^{\ell} ) \nabla \Phi ^\ell \Bigr)   \\
+  \|[n^+]^\ell-[n^-]^\ell\|^2_h  \Bigr] = E \bigl(\f v^0,\f d^0 , \Phi^0\bigr) +  \frac{h^\alpha}{2}\| \nabla \f v^0 \|^2_{{\mathbb L}^2}+ \frac{h^\beta}{2} \Vert \Delta_h \f d^0 \Vert _{h}^2 \,, \label{disenin}
\end{multline}
\item[iii)] bounds for discrete charges: For all $1 \leq j \leq J$,\begin{multline*}
 \| [ n^{+}]^j\|_h^2+\| [ n^{-}]^j\|^2_h 
+ k^2 \sum_{\ell=1}^j  \Bigl[\| d_t [ n^{+}]^\ell\|^2_i+ \| d_t [ n^{-}]^\ell\|^2_h\Bigr]
 \\
 +k \sum_{\ell=1}^j \left[
 \left ( \varepsilon(\f d ^\ell) \nabla [ n^{+}]^\ell , \nabla [n^{+}]^\ell \right) +\left ( \varepsilon(\f d ^\ell) \nabla [ n^{-}]^\ell , \nabla [n^{-}]^\ell \right )
  \right]
\\ \leq 
 \| [  n^{+}]^0\|^2_h+\| [ n^{-}]^0\|^2_h + C\left (  E( \f v ^0,\f d^0,\Phi^0) + \frac{h^\alpha}{2}\| \nabla \f v^0\|^2_{{\mathbb L}^2}+ \frac{h^\beta}{2} \Vert \Delta_h \f d^0 \Vert _{h}^2\right ) \,.
\end{multline*}
\item[iv)]  bounds for temporal variation: For all $1 \leq j \leq J$,
\begin{multline}
k \sum_{j=1}^J \left[ \| d_t[ n^{\pm}]^j \| _{({\mathbb H}^{1})^*} ^2 + \| d_t \f d ^j \|_{{\mathbb L}^{4/3}}^2 \right] +  \| \f q^j \|_{  ( \mathbb W^{2,2}\cap \mathbb W^{1,2}_0 )^* } \\
 \leq C \Bigl[E\bigl( \f v^0,\f d^0 , \Phi^0\bigr)+ \frac{h^\alpha}{2}\| \nabla \f v^0\|^2_{{\mathbb L}^2} + \frac{h^\beta}{2} \Vert \Delta_h \fii d \Vert _{h}^2 +1 \Bigr] ^2\,.\label{estdistime}
\end{multline}
\end{enumerate}
\end{theorem}

\medskip

\begin{proof} 
Assertion~$i)$ follows from Lemma~\ref{exist-lemma2} and Assertion~$ii)$ from~\eqref{energyincrement}, Assertion~$i)$, and summation. 
We start from~\eqref{Npm} with $\phi_{\gamma} ( |[n^\pm]^j|)=1$ 
and observe by~\eqref{thm:disex} i) and the coercivity of $\varepsilon(\f d^j)$  that
\begin{align*}
\Bigl |\Bigl ( [ n^{\pm}]^j \varepsilon(\f d^{j}) \nabla \Phi^j,\nabla [ n^{\pm}]^j\Bigr)\Bigr |  \leq {}&  \frac{1}{4} \Bigl ( \varepsilon(\f d^{j}) \nabla [ n^{\pm}]^j,\nabla [ n^{\pm}]^j\Bigr ) +\Bigl( \varepsilon(\f d^{j}) \nabla \Phi^j,\nabla \Phi^j\Bigr) \,,
\\
 \Bigl( \f v^j [ n^{\pm}]^j, \nabla [ n^{\pm}]^j \Bigr)  \leq{}& \| \f v \|_{\mathbb L^2} \| [n^{\pm}]^j \|_{\mathbb L^\infty} \| \nabla[n^{\pm}]^j \|_{\mathbb L^2}   \leq  \frac{1}{4} \Bigl( \varepsilon(\f d^{j}) \nabla [ n^{\pm}]^j,\nabla [ n^{\pm}]^j\Bigr) + 
  \| \f v^j \|^2_{{\mathbb L}^2
 }
 \,, 
\end{align*} 
where the second terms on the right-hand sides are bounded due to~\eqref{disenin} and the first ones may be absorbed into the left hand side of~\eqref{Npm}. 
Note that $\varepsilon(\f d^j)$ is a positive definite matrix. 
Summing up implies the assertion of $iii)$.

To verify assertion $iv)$, we use approximation properties of the 
${\mathbb L}^2$-projection {${\mathcal P}_{{\mathbb L}^2}$}:
 By the ${\mathbb H}^{1}$-stability of ${\mathcal P}_{{\mathbb L}^2}$, see~\cite{stab}, and~\eqref{dismass}, and with the help of (\ref{dis:N}) and the second assertion in i), we find
\begin{align*}
\| d_t [n^{\pm}]^j\|_{ {(\mathbb H}^{1})^*} \leq{}& \sup_{\varphi \in {\mathbb H}^1} \left | \frac{ \left (  d_t [n^{\pm}]^j , {\mathcal P}_{{\mathbb L}^2} \varphi \right )_ h  } { \| \varphi \|_{{\mathbb H}^1} } \right | 
+ \sup_{\varphi \in {\mathbb H}^1} \left | \frac{\left (  d_t [n^{\pm}]^j , {\mathcal P}_{{\mathbb L}^2} \varphi \right )- \left (  d_t [n^{\pm}]^j , {\mathcal P}_{{\mathbb L}^2} \varphi \right )_ h  } { \| \varphi \|_{{\mathbb H}^1} } \right | \\
\leq{}& C\Bigl( \| \f v^j \|_{{\mathbb L}^2}+ \| \nabla [n^{\pm}]^j \| _{{\mathbb L}^2} + \| \nabla \Phi^j \| _{{\mathbb L}^2} + h \| \td [n^{\pm}]^j \|_{{\mathbb L}^2} \Bigr ) \,.
\end{align*}
Note that due to~\eqref{dis:N} and an inverse inequality
\begin{eqnarray*}
\| \td [n^{\pm}]^j \|^2 _h &\leq&  \Bigl ( \| \nabla [n^{\pm}]^j \|_{{\mathbb L}^2} + (\| \nabla \Phi ^j \|_{{\mathbb L}^2} + \| \f v ^j \|_{{\mathbb L}^2})\| [n^{\pm}]^j\|_{{{\mathbb L}^
\infty}}  \Bigr) \| \nabla \td [n^{\pm} ]^j \|_{{\mathbb L}^2} \\
&\leq& \Bigl( \| \nabla [n^{\pm}]^j \|_{{\mathbb L}^2} + \| \nabla \Phi ^j \|_{{\mathbb L}^2} + \| \f v ^j \|_{{\mathbb L}^2}  \Bigr)h^{-1} \|  \td [n^{\pm} ]^j \|_{{\mathbb L}^2} \,.
\end{eqnarray*}
Applying the discrete integral operator implies the desired bound. 

 Using again the stability of the ${\mathbb L}^2$-projection, and~\eqref{dismass}, imply
\begin{align}
\begin{split}
\| d_t \f d^j \|_{{\mathbb L}^{4/3}} ={}& \sup_{\varphi \in {\mathbb L}^4} \left | \frac{\left( d_t \f d^j , \PL( \varphi)\right )_{{ h}} }{\| \varphi \| _{{\mathbb L}^4}} \right |+ \sup_{\varphi \in {\mathbb L}^4} \left | \frac{\left( d_t \f d^j , \PL( \varphi )\right ) _h - \left( d_t \f d^j , \PL( \varphi) \right ) }{\| \varphi \| _{{\mathbb L}^4}} \right | \\ \leq{}&  \sup_{\varphi \in {\mathbb L}^4} \| \varphi \|_{{\mathbb L}^3}\Bigl (  \| \f d^{j-1/2}\|_{{\mathbb L}^\infty}^2\| \f v^j \| _{{\mathbb L}^6} \| \PL(\nabla \f d ^{j-1}) \| _{{\mathbb L}^2}  
\\
&\qquad \qquad +\| \f d^{j-1/2}\|_{{\mathbb L}^\infty} \| \f d^{j-{1/2}} \times \f q^j \|_{{\mathbb L}^2}
\Bigr) + 
 C h ^{d/4}\| \td \f d ^j \| _{{\mathbb L}^{2}}
 \, .
\end{split}
\label{esttdd}
\end{align}
The last term on the right-hand side of the previous inequality stems from  the error due to the mass lumping, it can be seen by a proof similar to the one in~\cite{ciavaldini}.
From~\eqref{dismass}, we find by  Lemma~\ref{lem:inv}, and the stability of the projection that
\begin{subequations}\label{estdtime}
 \begin{align}\begin{split}
%
 \| ( \mathcal{I}_h - I) \bigl[ d_t \f d^j  \PL( \varphi)  \bigr]\|_{{\mathbb L}^1} 
 \leq{}& C
 h^2 \sum_{T \in \mathcal{T}_h} \| 1 \|_{\mathbb L^4(T)} \| \nabla    \td \fn d \|_{{\mathbb L}^{2}(T)} \| \nabla \PL( \varphi) \|_{{\mathbb L}^4(T)}
 \\ \leq{}& C h^{d/4}  \|    \td \fn d  \|_{{\mathbb L}^{2}} \| \varphi \|_{{\mathbb L}^{4}}
   \,.
\end{split}\label{masslumpingest}
\end{align}
Similar, we find for the $L^2$-norm of the time derivative
\begin{align}
\begin{split}
\| \td \f d ^j \| _h ^2 ={}& -\left ( \fh d \times ( \PL [ \nabla \fiii d ] \fii v ), \fh d \times \td \fn d \right )_h  - \left ( \fh d \times \fn q , \fh d \times \td \fn d \right ) _h 
\\
\leq{}&\| \fh d \|_{\mathbb L^\infty(\Omega)}^2 \| \f v^j \| _{\mathbb L^6(\Omega)} \|\PL[ \nabla \f d ^{j-1}] \| _{\mathbb L^2(\Omega)} \| \td \fn d \|_{\mathbb L^3(\Omega)} \\
& +\| \f d^{j-1/2}\|_{\mathbb L^\infty(\Omega)} \| \f d^{j-{1/2}} \times \f q^j \|_{\mathbb L^2(\Omega)}\| \td \fn d \|_{\mathbb L^3(\Omega)} \,,
\end{split}
\intertext{and by the Gagliardo--Nirenberg inequality,  the inverse inequality, and ~\eqref{dismass}, }
\| \td \fn d\|_{\mathbb L^3(\Omega)} \leq{}& C \| \td \fn d  \|_{\mathbb{L}^2}^{(6-d)/6} \| \nabla \td \fn d \|_{\mathbb{L}^2}^{d/6} \leq C h^{-d/6} \| \td \fn d \|_h\,.
\end{align}
Such that the last term on the right-hand side of~\eqref{esttdd} is bounded independently of $h$ and even vanishes for $h\ra 0$.
\end{subequations}
Considering the ${\mathbb L}^2$-norm in time, we find
\begin{align*}
k\sum_{j=1}^J \| d_t \f d^j \|_{{\mathbb L}^{4/3}} ^2 \leq{} c \Big( & \max_{j\in\{1,\ldots ,J\}}\| \f d^{j-1/2}\|_{{\mathbb L}^\infty}^4k\sum_{j=1}^J \| \f v^j \| _{{\mathbb L}^6}^2 \max_{j\in\{1,\ldots ,J\}} \| \nabla \f d ^{J} \| _{{\mathbb L}^2}^2\\&+\max_{j\in\{1,\ldots ,J\}}\| \f d^{j-1/2}\|^2_{{\mathbb L}^\infty}k\sum_{j=1}^J  \| \f d^{j-{1/2}} \times \f q^j \|_{{\mathbb L}^2}^2\Big)\,.
\end{align*}
Due to~\eqref{dis:q}, we may estimate $\f q^j$ via
\begin{subequations}\label{qmass}
\begin{align}
\begin{split}
\| \f q^j \|_{ ( \mathbb W^{2,2}\cap \mathbb W^{1,2}_0 )^* } ={}& \sup _{ \varphi \in \mathbb \mathbb W^{2,2}\cap \mathbb W^{1,2}_0 }\left | \frac{( \f q ^j , \PL  ( \varphi))_h }{\| \varphi \|_{\mathbb W^{2,2}\cap \mathbb W^{1,2}_0}} \right |+\sup _{ \varphi \in \mathbb W^{2,2}\cap \mathbb W^{1,2}_0} \left | \frac{( \f q ^j , \PL  ( \varphi))_h -( \f q ^j , \PL  ( \varphi)) }{\| \varphi \|_{\mathbb W^{2,2}\cap \mathbb W^{1,2}_0}} \right |
\\
\leq{}&  \sup _{ \varphi \in \mathbb W^{2,2}\cap \mathbb W^{1,2}_0 } \| \varphi \|_{\mathbb{ H}^1} \| \nabla \fh d \| _{\mathbb L^2} +  \sup _{ \varphi \in \mathbb W^{2,2}\cap \mathbb W^{1,2}_0 } \| \varphi \|_{ \mathbb L^\infty} \| \nabla \Phi^j \|_{ \mathbb L^2 }^2 \| \fh d \|_{\mathbb L^\infty } \\&+C \sup _{ \varphi \in \mathbb W^{2,2}\cap \mathbb W^{1,2}_0 } h^\beta \| \Delta _h \fh d \|_{\mathbb{L}^2} \| \varphi \| _{\mathbb W^{2,2}}  + C h^{1+d/3}  \| \f q^j \|_{\mathbb{L}^{2}} \,
\end{split}
\end{align}
for all $ \varphi \in \mathbb W^{2,2}\cap \mathbb W^{1,2}_0$,  where the error due to mass lumping is estimated similar to~\eqref{masslumpingest}, by 
\begin{align}\begin{split}
\| ( \mathcal{I}_h - I) \bigl[ \fii q   \PL( \varphi) \bigr]\|_{{\mathbb L}^1} 
 \leq{}& C
 h^2 \sum_{T \in \mathcal{T}_h} \| 1 \|_{\mathbb L^3(T)} \| \nabla  \fn q \|_{{\mathbb L}^{2}(T)} \| \nabla \PL( \varphi) \|_{{\mathbb L}^6(T)}
 \\ \leq{}& C h^{1+d/3}  \|  \fn q  \|_{{\mathbb L}^{2}} \| \varphi \|_{{\mathbb W}^{1,6}}
   \,.
\end{split}
\end{align}
From testing~\eqref{dis:q} by $\f q^j
$ and Lemma~\ref{lem:inv}, we find
\begin{align*}
\| 
\f q^j
\|_h^2 \leq{}& \| \nabla  \fh d \| _{\mathbb L^{2}}  \|
 \nabla\f q^j
\|_{\mathbb{ L}^2} +\| \nabla \Phi^j \|_{ \mathbb L^{2} }^2 \| \fh d \|_{\mathbb L^\infty }  \|
 \f q^j
  \|_{ \mathbb L^\infty }  + h^\beta \| \Delta _h \fh d \|_{\mathbb L^{2} } \| 
 \Delta _h  \fii q
   \|_{\mathbb L^2 } \\ \leq{}& C ( h^{-1}  +h^{-d/2}+ h^{\beta/2-2 } ) \| \f q^j \|_{\mathbb L^{2}} \,,
\end{align*}
which implies the asserted bound on $\f q^j$, (note that we have to choose $ \beta > 2/3$ for $d=2$)
$$ \| \f q^j \|_{  (\mathbb W^{2,2}\cap \mathbb W^{1,2}_0  )^* }  \leq C\,.$$ 
\end{subequations}
\end{proof}
\begin{remark}
Note that in the continuous case the term due to the mass lumping in~\eqref{esttdd} vanishes and we can even deduce the  bound on the time-derivative asserted in~\eqref{reg:weak}. 
\end{remark}

\subsection{Approximate relative energy inequality}\label{approximate}
 In \cite{approx}, an {\em approximate relative energy inequality} has been derived for a (semi-)discretization of the Ericksen-Leslie system (\ref{simp:v})--(\ref{simp:norm}).
The approximate relative energy is a new tool for the construction of
a {\em dissipative solution}, which here is employed for the  space-time discretization (\ref{dis}): instead of showing the convergence of its approximate solutions directly, the result in Proposition \ref{prop:disrel} essentially bounds the distance between 
approximate solutions and a related regular test function in terms of how well the chosen test function solves problem (\ref{simp}). The {\em approximate relative energy inequality} is essential in Section \ref{convergence1} to construct a
{\em dissipative solution} for \eqref{dis} via proper convergent sequences of functions that are generated from the discrete system~\eqref{dis}. 
In the proof of the approximate relative energy inequality, several difficulties arise due to the different discretization steps. To focus on the ideas of the proof of the relative energy inequality, the reader is rather referred to the proof in the continuous case (see Proposition~\ref{prop:cont}).

The next result is a  discrete relative energy inequality
for a solution of Scheme \ref{algo1}, which employs modifications
$\mathcal{K}_1$, $\mathcal{K}_2$ and $\mathcal{K}_d$, and $\mathcal{W}_d$ of related ones used in (\ref{relencont}): the different regularity measures~$\mathcal{K}_1$, $\mathcal{K}_2$ and $\mathcal{K}_d$ are given by
\begin{align*}
\mathcal{K}_1 ( \tu^j , \tu^{j-1}) := {}& {C} \Big(  \| \nabla \ddii \|_{{\mathbb L}^3}^4 + \| \tqj\|_{{\mathbb L}^3}^4 + \| \vvi\|_{{\mathbb L}^\infty}^4 + \| \nabla \tPj \|_{{\mathbb L}^\infty}^8 + \| \nabla \ddj \|_{{\mathbb L}^3}^4 + \| \nabla^2 \tPj \|_{{\mathbb L}^3}^2 \\ & \quad + \| [\nabla \tn^+]^j \|_{{\mathbb L}^3}+ \| [\nabla \tn^-]^j\|_{{\mathbb L}^3}  + \| \t \nabla \tPj \|_{{\mathbb L}^3} + \| [\tn^+]^j + [\tn^-]^j \|_{{\mathbb L}^\infty}^2   \\& \quad + \| \td \ddj \|_{{\mathbb L}^\infty}^{4/3}
+ \bigl\| \ddh\times ( ( \vvi \cdot \nabla ) \ddii + \tqj ) \bigr\|_{{\mathbb L}^\infty} \left ( \| \nabla \tPj\|_{L^\infty(0,T;{\mathbb L}^3)} ^2 + 1\right )  \\& \quad + \bigl\| \ddh\times \bigl( ( \vvi \cdot \nabla \bigr) \ddii + \tqj ) \bigr\|_{{\mathbb W}^{1,3}}  + 1 \Big )
\\
\mathcal{K}_2 ( \tu^j , \tu^{j-1}) : = {}&  {C} \Big(  \| \nabla \ddii \|_{{\mathbb L}^3}^4 + \| \tqj\|_{{\mathbb L}^3}^4 + \| \vvi\|_{{\mathbb L}^\infty}^4 + \| \t \nabla \tPj \|_{{\mathbb L}^3} + \| \nabla \tPj \|_{{\mathbb L}^\infty}^8   \\& \quad 
+ \bigl\| \ddh\times ( ( \vvi \cdot \nabla ) \ddii + \tqj ) \bigr\|_{{\mathbb L}^\infty} \left ( \| \nabla \tPj\|_{L^\infty(0,T; {\mathbb L}^3)} ^2 + 1\right )  \\&  \quad + \bigl \| \ddh\times ( ( \vvi \cdot \nabla ) \ddii + \tqj ) \bigr\|_{{\mathbb W}^{1,3}}   + 1 \Big ) 
\\
\mathcal{K}_d ( \tu^j) := {}& k \varepsilon_a \| \td \ddj \| _{{\mathbb L}^\infty}^2 \,. 
\end{align*}
Note that  $ \mathcal{K}_1 ( \tu|\tu)+ \mathcal{K}_2( \tu, \tu) = \mathcal{K}(\tu)$ and that $\mathcal{K}_d (\tu) \ra 0 $ as $k\ra 0$. Additionally, we define  the discrete dissipation distance by 
\begin{align*}
\mathcal{W}_d( \fn u , \tu^j) ={}& \nu \| \nabla (\fn v - \vvi)\|
^2 + \| \fh d \times \fn q - \dd^{j-1/2} \times \tqj \|_h^2 \\&+ \int_\Omega ( [n^+]^j+ [n^-]^j ) | \nabla \Phi^j-\nabla \tPj |_{\varepsilon(\fn d )}^2 \de \f x  + \bigl\| ([n^+]^j-[n^-]^j)- ( [\tn^+]^j-[\tn^-]^j) \bigr\|_h^2 \,.
\end{align*}
The  relative energy  $\mathcal{R}$ is given by~\eqref{relEnergy}.
 Moreover, we use the abbreviative notation
$\tilde{\f a}^j  = \tilde{\f a}( j\cdot k)$ for a continuous function $\tilde{\f a} \in \C(\Omega\times (0,T))$, where $0\leq j \leq J$ with $J = \lfloor (T/k)\rfloor$.

\begin{proposition}[Relative energy inequality]\label{prop:disrel}
Let $\fn u = (\fn v , \ov{\f d}^j , \fn q, [ n^{\pm}]^j,\Phi^j ) \in \f U_h $ be the solution of the fully discrete system~\eqref{dis} according to Theorem~\ref{thm:disex}. Let $\tu = ( \vv,\dd,\tP,\tn^{\pm}) \in {\mathbb Y}$ be a smooth test function.
Then the discrete relative energy inequality 

\begin{align}
\begin{split}
&\td  \mathcal{R}( \f u^j | \tu^j ) + \frac{h^\alpha }{2}\td \| \nabla 
\f v^j
\|^2_{{\mathbb L}^2}  
+ \frac{h^\beta }{2} \| \Delta _h
 \fii d 
  \|_h^2 
 +\td r_k ^1(h)+
 \frac{1}{2} 
  \mathcal{W}_d( \fn u , \tu^j)\\&\quad
 \leq \Bigl ( \mathcal{K}_1( \tu^j, \tu^{j-1}) + \mathcal{K}_d(\tu^j)\Bigr)   \left (\mathcal{R}( \f u^j | \tu^j ) 
 + \frac{h^\beta }{2} \| \Delta _h
 \fii d 
  \|_h^2 
  \right ) + \mathcal{K}_2(\tu^j, \tu^{j-1})  \mathcal{R}( \f u^{j-1} | \tu^{j-1} )
 \\&\qquad 
  +\frac{1}{2}\left [  \| [n^+]_h^j-[\tn^+]^j \|_{{\mathbb L}^2}^2 + \| [n^-]_h^j-[\tn^-]^j \|_{{\mathbb L}^2}^2  \right ]
 \\&\qquad+ 
 \left \langle \mathcal{A}_d^j(\tu^j) , \begin{pmatrix}
\vvi-\fn  v \\ \tqj - \fn q + 
\nabla \tPj\left ( \nabla \tPj \cdot ( \ddj - \fn d )\right )  \\ \tPj- \Phi_h^j \end{pmatrix}  \right \rangle
 + r_k^2(h)
 \,.\end{split}
\label{relendis}
\end{align}
holds for any $1\leq j \leq J $ ,where $r^1_k(h) \ra 0$ and $ \sumi | r^2_k(h)| \ra 0$ as $h\ra0$. 
Here, we defined the discrete solution operator $\mathcal{A}_d^j $ by 
\begin{align}
\begin{split}
&\left \langle \mathcal{A}^j_d ( \tu^j), \begin{pmatrix}
\f a \\ \f c \\ e^{\pm}
\end{pmatrix}\right \rangle \\ :=& 
\left ( d_t \vv ^j, \f a\right  ) + \left ( \nabla \vv^j , \nabla \f a \right ) 
+ \Bigl( ( \vv^{j-1} \cdot \nabla ) \vv^j , \f a \Bigr) + \frac{1}{2} \left ( \di \vv ^{j-1} \vv^j , \f a \right ) \\
&\quad
+ \Bigl ( ([\tn^+]^j-[\tn^-]^j )\nabla \tP^j , \f a \Bigr) 
+ \Bigl ((\nabla \dd ^{j-1})^\top  \bigl[\dd ^{j-1/2} \times (\dd ^{j-1/2}\times \tq ^{j})\bigr] , \f a \Bigr )_h
\\
&\quad
+\left ( d_t \dd^j , \f c \right )
+ \left ( \dd ^{j-1/2}\times(( \vv ^j \cdot \nabla ) \dd ^{j-1} ) ,  \dd ^{j-1/2} \times \f c \right )_h + \left (\dd ^{j-1/2}\times \tq^j , \dd ^{j-1/2}\times  \f c \right )_h\\
&\quad 
+\left (d_t [ \tn ^{\pm}]^j , e^{\pm} \right )_h + \left ( \varepsilon(\dd^{j})\nabla [ \tn ^{\pm}]^j , \nabla e ^{\pm} \right ) \\
&\quad
\pm \left ([\tn^{\pm}]^j \varepsilon(\dd^{j} ) \nabla \tP^{j} , \nabla e^{\pm} \right )  - \left ( \vv^j [ \tn^{\pm}]^j, \nabla e^{\pm} \right ) \,,
\end{split}\label{Adj}
\end{align}
where 
$ \tq^j = - \Delta \ddh + \varepsilon_a \nabla \tPj \nabla \tPj \cdot \dd^{j-1} $
and   $\tP$ is given as the solution of
$  ( \varepsilon(\dd)\nabla \tP , \nabla g)  = (\tn^+-\tn^-, g)_h $ for all $g \in \C^\infty_c(\Omega \times (0,T))$,
 as well as $| \dd| =1 $  a.e.~in $\Omega\times (0,T)$. 
\end{proposition}

\begin{proof}
We start by decomposing the relative energy into the two energy parts and the mixed parts:
\begin{align}
\begin{split}
\mathcal{R}( \f u^j | \tu) ={}& \frac{1}{2}\left ( \| \fn v \|_{{\mathbb L}^2}^2 + \| \nabla \fn d  \|_{{\mathbb L}^2}^2 + \int_\Omega | \nabla \Phi^j |_{\varepsilon(\fn d)}^2 \, {\rm d}{\f x}  + \| \vvi \|_{{\mathbb L}^2}^2 + \| \nabla \ddi \|_{{\mathbb L}^2}^2 + \int_\Omega | \nabla \tPj |_{\varepsilon(\ddi )}^2 \, {\rm d}{\f x}\right)
\\
& -\bigl( \nabla \fn d, \nabla \ddi\bigr)  -( \fn v , \vvi)  
- \Bigl( \nabla \Phi^j ,\varepsilon(\fn d) \nabla \tPj \Bigr) + \frac{1}{2}\left ( \nabla \tPj , \left ( \varepsilon( \fn d ) - \varepsilon( \ddj)\right ) \nabla \tPj \right ) 
\,.
\end{split}
\label{relencalc}
\end{align}
Similarly, we obtain for the relative dissipation
\begin{align*}
\mathcal{W}_d(\fn u| \tu^j ) ={}&
\mathcal{W}_d(\fn u | \f 0 ) + \mathcal{W}_d(\f 0|\tu^j) - 2 \nu \left ( \nabla \fn v  , \nabla \vvi\right ) - 2 \left ( \fn d \times \fn q , \ddj \times \tqj\right ) _h 
\\&-   2 \Bigl([n^+]^j-[n^-]^j,[\tn^+]^j-[\tn^-]^j \Bigr)_h-
2 \int_\Omega \bigl( [n^+]^j-[n^-]^j \bigr) \nabla \Phi^j \cdot \varepsilon(\fn d )\nabla \tPj \de \f x 
  \, . 
\end{align*}

The energy increment equality for the solution $\f u \in \mathbb X $ is given by (compare to~\eqref{energyincrement}) 
\begin{multline*}
\td \left ( E (\f v^j,\f d^j , \Phi^j) +  \frac{h^\alpha}{2}\| \nabla \f v^j \|^2_{{\mathbb L}^2}
 + \frac{h^\beta}{2} \| \Delta_h \fii d \|_h^2
     \right  ) \\ +
\frac{k}{2}
\left[\| d_t \f v^j \|^2_{{\mathbb L}^2}+ h^\alpha \| d_t \nabla \f v^j \| ^2_{{\mathbb L}^2}
 + \Bigl(\varepsilon(\f d^{j-1})\nabla d_t \Phi^j , \nabla d_t \Phi^j \Bigr) + \varepsilon_a \| \nabla \Phi^j \cdot d_t \f d^j \|^2_{{\mathbb L}^2}   \right]
\\
+
\left ( \nu \| \nabla \f v^j \|^2_{{\mathbb L}^2} + \| \f d ^{j-1/2} \times \f q^j \|^2_h + \Bigl ( ([n^+]^j+[n^-]^j) \nabla \Phi^j , \varepsilon(\f d^{j} ) \nabla \Phi^j   \Bigr) + \|[n^+]^j-[n^-]^j\|^2_h  \right ) 
= 0
\end{multline*}
where $E$ is defined in Theorem~\ref{thm:disex}.
For the test function {$\tilde{u} \in {\mathbb Y}$}, we observe that
\begin{multline*}
\td  E (\vv^j,\dd^j , \tP^j) 
  +
\frac{k}{2}
  \left (\| d_t \vv^j \|^2_{{\mathbb L}^2} 
  + \Bigl(\varepsilon(\dd^{j-1})\nabla d_t \tP^j , \nabla d_t \tP^j \Bigr) + \varepsilon_a \| \nabla \tP^j \cdot d_t \dd^j \|^2_{{\mathbb L}^2}   \right )
\\
+ 
 \left ( \nu \| \nabla \vv^j \|^2_{{\mathbb L}^2} + \| \dd ^{j-1/2} \times \tq^j \|^2_h + \left ( ([\tn^+]^j+[\tn^-]^j) \nabla \tP^j , \varepsilon(\dd^{j} ) \nabla \tP   \right ) + \|[\tn^+]^j-[\tn^-]^j\|^2_h  \right )\\
=
   \left \langle   \mathcal{A}_d(\tu^j) , \begin{pmatrix}
 \vv^j\\ \tq^j\\ \tP^j
\end{pmatrix}\right \rangle 
\end{multline*}
For the mixed terms in the second line of~\eqref{relencalc}, we need several discrete product rules:
\begin{align}
\begin{split}
( \fn v   , 
\vvi
) - ( \fiii v   , \vvii  ) = {}& {k} \left [( \fii v  , d_t \vv^j ) + ( d_t \fii v , \vv^j ) \right  ] -  {k^2}( d_t \fii v  , d_t \vv^j) 
 \, ,\end{split}
\label{intvn}\\
\begin{split}
\left ( \nabla \fn d  , \nabla 
\ddj 
 \right ) - \left (  \nabla \fiii d  ,\nabla\ddii \right )
= {} & 
  \left [  \left ( \nabla \td \fii d , \nabla \ddh    \right) + \left ( \nabla  \fh d , \nabla \td  \ddi    \right)  \right ]
\,,
\end{split}
\label{intd2n}
\end{align}
Indeed by a simple calculation, we deduce the first identity via
\begin{equation*}
( \fn v  , \vvi  ) - ( \fiii v , \vvii  ) =   
  (\fii v- \fiii v ,  \vvi ) + ( \fii v ,\vvi -  \vvii )  - ( \fii v - \fiii v , \vvi -\vvii)  
\end{equation*}
and the second identity follows accordingly.

Testing~\eqref{dis:v} by $
\mathcal{I}_h
[\vv]$ and mimicking the same calculations for the test function tested by $\f v_h$ and using the discrete integration by parts formula~\eqref{intvn} implies 
\begin{align}
\begin{split}
- ( \fii v , \vvi ) &+ ( \fiii v , \vvii  )
\\
 \leq  {}&
  \Bigl( ( \fnn v \cdot \nabla) \fn v ,
  \mathcal{I}_h
  [\vvi] \Bigr) + \frac{1}{2}\left ( (\di \fnn v ) \fn v,
   \mathcal{I}_h
   [\vvi]\right ) + \Bigl( ( \vvii\cdot\nabla) \vvi , \fn v \Bigr)
\\&-
\left (\fh d \times ( \PL[ \nabla \fiii d] 
\mathcal{I}_h
[\vvi]),\fh d \times  \fn q  \right ) _h \\&- \left (\dd^{j-1/2}\times( ( \fn v \cdot \nabla ) \ddii) , \dd^{j-1/2}\times(\tqj) \right ) _h 
\\&+
 \Bigl (([n^+]^j-[n^-]^j) \nabla \Phi^j  ,
  \mathcal{I}_h
  [\vvi] \Bigr ) +2 \nu( \nabla \fn v , \nabla \mathcal{I}_h \vvi ) 
\\&+
\Bigl ( ([\tn^+]^j-[\tn^-]^j) \nabla \tPj , \fn v \Bigr ) - \left \langle \mathcal{A}_d^j(\tu) , \begin{pmatrix}
\fn v \\ {\f 0}\\0
\end{pmatrix}\right \rangle 
\\&+k
 ( d_t \fii v , d_t \vv^j )+h^\alpha
  ( d_t\nabla  \fii v ,  \nabla \mathcal{I}_h \vv^j ) + \Bigl( \td \fn v , ( \mathcal{I}_h - I) \vvi\Bigr) 
  \,.
\end{split}
\label{vtestvtil}
\end{align}
Testing~\eqref{dis:d} by $\mathcal{I}_h[ \tqj]$, adding~\eqref{dis:q} tested by $\mathcal{I}_h[\td \ddi]$ and mimicking the same calculations for the test function tested by $\fn q$ implies 
\begin{multline}
- \left ( \nabla \fii d , \nabla \ddi \right ) + \left ( \nabla \fiii d ,\nabla \ddii \right )  
\\={}
\left ( \fh d \times ( \PL[ \fiii d]\fii v  ),\fh d \times \mathcal{I}_h[ \tq] \right )_h  
+ \left ( \ddi\times((\vvi\cdot\nabla) \ddii) ,\ddi\times \fn q \right )_h 
 \\ 
 +
  \left ( \fh d \times \fn q , \fh d \times \mathcal{I}_h[ \tqj] \right ) _h + \left ( \dd^{j-1/2} \times \tqj , \dd ^{j-1/2} \times \fn q \right ) _h 
\\
- \varepsilon _a\left (
  \left ( \td \fn d \nabla \tPj , \nabla \tPj \cdot \ddii \right ) + \left ( \td \mathcal{I}_h[\ddi] \cdot \nabla \Phi^j , \nabla \Phi^j \cdot \fiii d \right )   \right ) 
  - 
  \left \langle \mathcal{A}_d^j(\tu) ,\begin{pmatrix}
{\f 0}\\\fn q \\0
\end{pmatrix}  \right \rangle 
\\
+h^\beta  \left (  \Delta_h \fh d , \Delta_h \mathcal{I}_h \td \fn d   \right )+ \int_\Omega ( \mathcal{I}_h - I ) \left ( \td \fn d \cdot \tqj \right ) \de \f x 
\end{multline}

For the remaining terms in the relative energy, we use after some manipulations equation~\eqref{simp:Phi} 
\begin{align*}
- \bigl( \nabla \Phi^j , \varepsilon(\f d^j) \nabla \tP ^j\bigr)
={}&-
\left ([n^+]^j -[n^+]^j  ,  \tP^j\right )_h
- \Bigl( \varepsilon(\f d^j) \nabla \Phi^j, \nabla \left ( I- \mathcal{I}_h \right )\tP^j \Bigr )
\,.
 \end{align*}

We observe another integration by parts rule:
\begin{align*}
\td \Bigl( [n^+]^j-[n^-]^j , \mathcal{I}_h [\tP]\Bigr )_h =  \Bigl (\td \left ([n^+]^j-[n^-]^j \right ) , \mathcal{I}_h [\tPj]\Bigr)_h + \Bigl( [n^+]^j-[n^-]^j , \mathcal{I}_h [\td\tPj]\Bigr)_h \\
-k \Bigl(\td\left ( [n^+]^j-[n^-]^j\right ) , \td\mathcal{I}_h [\tPj]\Bigr)_h 
\\
=
 \Bigl(\td \left ([n^+]^j-[n^-]^j\right )  , \mathcal{I}_h [\tP]\Bigr)_h + \Bigl( [n^+]^j-[n^-]^j  , \mathcal{I}_h [\td\tPj]\Bigr)_h \\
 -k \Bigl(\td \left ( \varepsilon(\f d^j) \nabla \Phi^j \right ) ,\nabla \td \mathcal{I}_h [\tPj]\Bigr) 
\,.
\end{align*}

Testing~\eqref{dis:Phi} by $\mathcal{I}_h[\td \tPj]$ and add~\eqref{simp:c} tested by $\mathcal{I}_h[\tPj]$ and mimicking the same for $\tu$, while replacing $ \td ([\tn^+]^j-[\tn^-]^j)\Phi^j $ by $d_t ( \varepsilon(\ddi) \nabla\tPj) \cdot \nabla\Phi^j $, we find
\begin{eqnarray*}
&&- \td \Bigl ( \nabla  \Phi^j , \varepsilon( \fn d ) \nabla \tPj \Bigr ) 
  = - 
  \Bigl (  \fn v ([n^+]^j-[n^-]^j) , \nabla  \mathcal{I}_h[\tPj]  \Bigr )  - \Bigl(\vvi([\tn^+]^j-[\tn^-]^j),\nabla \Phi^j \Bigr)
\\
&&\qquad -
 \Bigl ( (\varepsilon(\fn d ) \nabla \Phi^j  ,   \nabla \mathcal{I}_h[\td\tPj] \Bigr ) + \Bigl( d_t ( \varepsilon(\ddi) \nabla\tPj) , \nabla\Phi^j \Bigr)   \\
&&\qquad  +
  \Bigl ( \varepsilon(\fn d ) \nabla\left ( [n^+]^j-[n^-]^j\right ) ,\nabla  \mathcal{I}_h[\tPj] \Bigr ) + \Bigl (\varepsilon(\ddi) \nabla ([\tn^+]^j-[\tn^-]^j) , \nabla \Phi^j \Bigr ) 
\\&&\qquad  +
 \Bigl( \varepsilon(\fn d) \nabla \Phi^j  , \left ([n^+]^j+[n^-]^j\right )  \nabla \mathcal{I}_h[\tPj] \Bigr ) +\Bigl ( \varepsilon(\ddi) \nabla \tPj , ([\tn^+]^j+[\tn^-]^j) \nabla \Phi^j \Bigr )   
\\
&&\qquad   - 
\left \langle 
 \mathcal{A}_d^j(\tu^j), \begin{pmatrix}
{\f 0}\\{\f 0}\\ \Phi^j 
\end{pmatrix} 
\right \rangle
 + k
  \Bigl(\td \left ( \varepsilon(\f d^j) \nabla \Phi^j \right ), \nabla \td\mathcal{I}_h [\tPj]\Bigr )  
\\&&\qquad- \td \Bigl ( \varepsilon(\f d^j) \nabla \Phi^j, \nabla \left ( I- \mathcal{I}_h \right )\tP^j \Bigr )\,. 
\end{eqnarray*}

For the different terms on the right-hand side, we infer
\begin{align*}
&\left ( \varepsilon(\fn d ) \nabla [n^{\pm}]^j ,  \nabla \mathcal{I}_h[\tPj] \right ) + \left (\varepsilon(\ddi) \nabla [\tn^{\pm}]^j , \nabla \Phi^j \right )  
\\
&\quad = \left ( \varepsilon(\ddi ) \nabla [n^{\pm}]^j ,  \nabla \tPj  \right ) + \left (\varepsilon(\fn d ) \nabla \mathcal{I}_h[\tn^{\pm} ]^j, \nabla \Phi^j \right )\\& \qquad + \left ( \varepsilon(\fn d )\nabla \mathcal{I}_h[\tPj]-\varepsilon(\fn d)\nabla \tPj, \nabla [n^{\pm}]^j \right ) + \left (\varepsilon(\ddi)\nabla [\tn^{\pm}]^j - \varepsilon(\fn d)\nabla \mathcal{I}_h[\tn^{\pm}]^j ,\nabla \Phi^j\right ) 
\\ &\quad = \Bigl ( [\tn^+]^j-[\tn^-]^j, [n^{\pm}]^j\Bigr )_h + \Bigl ( [n^+]^j-[n^-]^j , \mathcal{I}_h[\tn^{\pm} ]^j\Bigr )_h
 \\ &\qquad + \left ( (\varepsilon(\fn d )-\varepsilon(\ddi)) \left (  \nabla [n^{\pm}]^j-\nabla [\tn^{\pm}]^j \right ) ,  \nabla\tPj \right ) + \left ((\varepsilon(\ddi)-\varepsilon(\fn d)) \nabla [\tn^{\pm}]^j , \nabla \Phi^j -\nabla \tPj \right ) 
\\ &\qquad + \Bigl ( \varepsilon(\fn d )\nabla (\mathcal{I}_h[\tPj]- \tPj), \nabla [n^{\pm}]^j \Bigr) - \Bigl( \varepsilon(\fn d)\nabla (\mathcal{I}_h[\tn^{\pm}]^j -[ \tn^{\pm}]^j ) ,\nabla \Phi^j\Bigr) 
\end{align*}
and
\begin{align*}
 &\Bigl( \varepsilon(\fn d) \nabla \Phi^j  , ([n^+]^j+[n^-]^j)  \nabla \mathcal{I}_h[\tPj] \Bigr ) +\Bigl( \varepsilon(\ddi) \nabla \tPj , ([\tn^+]^j+[\tn^-]^j) \nabla \Phi^j \Bigl)
\\
&\quad = 2\Bigr ( \varepsilon(\fn d) \nabla \Phi^j  ,  ([n^+]^j+[n^-]^j)  \nabla \tPj \Bigr ) \\ & \qquad + \Bigl( \varepsilon(\ddi) ([\tn^+]^j+[\tn^-]^j )- \varepsilon(\fn d)  ([n^+]^j+[n^-]^j) , \nabla \tPj \o\nabla \Phi^j \Bigr )
\\& \qquad  + \Bigl( \varepsilon(\fn d) \nabla \Phi^j  , ([n^+]^j+[n^-]^j)  \nabla ( \mathcal{I}_h[\tPj] -\tPj)\Bigr )
\end{align*}
as well as
\begin{align*}
-\Bigl( \varepsilon(\fn d ) \nabla \Phi^j  ,   \nabla \mathcal{I}_h[\td\tPj] \Bigr) +{}& \Bigl ( \td ( \varepsilon(\ddi) \nabla\tPj) , \nabla\Phi^j \Bigr ) \\
={}& - \left ( ( \varepsilon(\fn d)- \varepsilon(\ddi) ) 
\nabla \Phi^j
, \nabla \td \tPj \right ) 
 \\
& + \left ( \td \varepsilon(\ddi)\nabla \tP^{j-1} , \nabla \Phi^j \right ) - \Bigl ( \varepsilon( \fn d ) \nabla \Phi^j , \nabla \td ( \mathcal{I}_h - I ) \tPj \Bigr ) \,.
\end{align*}

From calculating the discrete derivative of $\left ( \nabla \tPj , (\varepsilon(\ddj ) - \varepsilon(\fn d ) ) \nabla \tPj\right  )$, we find
\begin{align*}
&\frac{1}{2}d_t \left ( \varepsilon(\fn d) - \varepsilon ( \ddj) , \nabla \tPj \otimes \nabla \tPj \right ) =  \frac{1}{2} \left ( d_t \left ( \varepsilon(\fn d) - \varepsilon(\ddj) \right ) , \nabla \tPj \otimes \nabla \tPj \right )  \\
&\qquad + \frac{1}{2} \left ( \varepsilon(\fn d) - \varepsilon ( \ddj) , d_t \left ( \nabla \tPj \otimes \nabla \tPj\right ) \right ) - \frac{k}{2} \left ( d_t \left ( \varepsilon(\fn d) - \varepsilon(\ddj) \right ) , d_t \left ( \nabla \tPj \otimes \nabla \tPj\right )\right ) 
\\
&\quad =\varepsilon_a \left ( \fn d \otimes \td \fn d - \ddj \otimes \td \ddj , \nabla \tPj\otimes \nabla  \tPj \right ) + \left ( \varepsilon(\fn d ) - \varepsilon(\ddj) , \nabla \tPj \otimes \td \nabla \tPj \right ) 
\\
&\qquad - \frac{k}{2} \varepsilon_a  \left ( \td \fn d \otimes \td \fn d - \td \ddj \otimes \td \ddj , \nabla \tPj \otimes \nabla \tPj \right )\\&  \qquad - \frac{k}{2} \left ( \left ( \varepsilon(\fiii d) - \varepsilon(\ddii) \right ) , d_t  \nabla \tPj \otimes \td \nabla \tPj\right ) 
- {k} \left ( d_t \left ( \varepsilon(\fn d) - \varepsilon(\ddj) \right ) , d_t  \nabla \tPj \otimes \nabla \tPj\right ) 
\,.
\end{align*}

We find for the terms incorporating $\varepsilon_a$ that
\begin{align*}
\frac{1}{2}d_t& \Bigl ( \varepsilon(\fn d) - \varepsilon ( \ddj) , \nabla \tPj \otimes \nabla \tPj \Bigr )  + k \Bigl ( \td ( \varepsilon ( \fn d) \nabla \Phi^j), \nabla \td \tPj \Bigr ) 
-\varepsilon_a \left ( d_t \fn d \nabla \tPj, \nabla \tPj \ddii \right ) \\
-& \left ( ( \varepsilon(\fn d ) -\varepsilon(\ddi) ) \nabla \Phi^j , \nabla \td \tPj \right )-\varepsilon_a \left ( \td \ddj \nabla \Phi^j, \nabla \Phi^j \cdot \fiii d \right ) + \left ( \td \varepsilon(\ddj) \nabla \tP^{j-1} , \nabla \Phi^j \right ) \\
= {}
&\varepsilon_a \left ( \fn d \otimes \td \fn d - \ddj \otimes \td \ddj , \nabla \tPj\otimes \nabla  \tPj \right ) + \left ( \varepsilon(\fn d ) - \varepsilon(\ddj) , \nabla \tPj \otimes \td \nabla \tPj \right ) 
\\
&- \left ( ( \varepsilon(\fn d ) -\varepsilon(\ddi) ) \nabla \Phi^j , \nabla \td \tPj \right )-\varepsilon_a \left ( d_t \fn d \nabla \tPj, \nabla \tPj \ddi \right )-\varepsilon_a \left ( \td \ddj \nabla \Phi^j, \nabla \Phi^j \cdot \fii d \right ) 
\\
&+\varepsilon_a \left ( \td \ddj \nabla \tPj , \nabla \Phi^j \ddj \right ) +\varepsilon_a \left ( \td \ddj \cdot \nabla \Phi^j , \ddj \cdot \nabla \tPj \right ) 
\\&- \varepsilon_a k \left ( \td \ddj \cdot \nabla \Phi^j , \td \ddj \cdot \nabla \tPj\right ) - k \left ( \td \varepsilon(\ddj) \td \nabla \tPj , \nabla \Phi^j \right )  + k \left ( \td ( \varepsilon ( \fn d) \nabla \Phi^j), \nabla \td \tPj \right ) 
\\
&+\varepsilon_a k \left ( d_t \fn d \nabla \tPj, \nabla \tPj \td \ddi \right )+ \varepsilon_ak \left ( \td \ddj \nabla \Phi^j, \nabla \Phi^j \cdot \td \fii d \right )\\
&- \varepsilon_a\frac{k}{2} \left ( \td \fn d \otimes \td \fn d - \td \ddj \otimes \td \ddj , \nabla \tPj \otimes \nabla \tPj \right ) - \frac{k}{2} \left ( \left ( \varepsilon(\fiii d) - \varepsilon(\ddii) \right ) , d_t  \nabla \tPj \otimes \td \nabla \tPj\right )
\\
&
- {k} \left ( d_t \left ( \varepsilon(\fn d) - \varepsilon(\ddj) \right ) , d_t  \nabla \tPj \otimes \nabla \tPj\right ) \\
={}& 
\varepsilon_a \left ( \td \fn d - \td \ddj , \nabla \tPj \left (\nabla \tPj \cdot ( \fn d - \ddj )\right )\right ) +\varepsilon_a \left ( \td \ddj \cdot ( \nabla \Phi^j - \nabla \tPj ) , \nabla \tPj \ddj-\nabla \Phi^j \fn d \right )\\&  +\varepsilon_a \left ( \td \ddj \cdot \nabla \tPj , ( \fn d - \ddj ) \cdot \nabla( \tPj-\Phi^j ) \right ) +  \left (  \left ( \varepsilon( \fiii d )- \varepsilon( \ddii) \right ) ( \nabla \Phi^j - \nabla \tPj ) , \td \nabla \tPj \right )  
\\&- \frac{k}{2} \left ( \left ( \varepsilon(\fiii d) - \varepsilon(\ddii) \right ) , d_t  \nabla \tPj \otimes \td \nabla \tPj\right ) + k \left ( \varepsilon(\fiii d ) \td\nabla  \Phi^j , \td \nabla \tPj \right )  
\\& + \varepsilon_a \frac{k}{2} \| \td \fn d  \cdot \nabla \Phi^j \|^2_{{\mathbb L}^2} +  \varepsilon_a \frac{k}{2}  \| \td \ddj \cdot \nabla \tPj \|^2_{{\mathbb L}^2} -  \varepsilon_a \frac{k}{2}  \| ( \td \fn d - \td \ddj ) \cdot \nabla \Phi^j \|^2  _{{\mathbb L}^2}
\\& 
- \varepsilon_a \frac{k}{2} \| ( \td \fn d - \td \ddj ) \cdot \nabla \tPj \|^2_{{\mathbb L}^2}
+ \varepsilon_a \frac{k}{2} \| \td \ddj ( \nabla \tPj - \nabla \Phi^j) \|^2_{{\mathbb L}^2} \, .
\end{align*}
The first two lines on the right-hand side are similar to the contributions in the continuous case. 
The terms in the third and fourth line on the right-hand side contribute to the positive terms on the left-hand side of the following inequality. 
The term in the last line has to be estimated later on. 

Putting the pieces together, we observe the inequality
\begin{align*}
 \td   \mathcal{R} ( \fn u | \tu^j  ) &+\frac{h^\alpha }{2} \td  \| \nabla 
 \f v^j
 \|^2_{{\mathbb L}^2} 
 + \frac{h^\beta}{2} \| \Delta_h
  \fii d
    \|_h^2 
+\frac{ k}{2}  \left [\| \td \fn v -  d_t \vv^j \|^2_{{\mathbb L}^2} + h^\alpha \|  \td \nabla 
\fn v 
 \| ^2_{{\mathbb L}^2} 
\right ]
\\
&+\frac{ k}{2}  \Bigl [ \Bigl(\varepsilon(\dd^{j-1})\nabla d_t (\Phi^j- \tP^j) , \nabla d_t (\Phi^j-\tP^j) \Bigr)) 
+  \varepsilon_a   \| ( \td \fn d - \td \ddj ) \cdot \nabla \Phi^j \|^2_{{\mathbb L}^2} 
\\& + \varepsilon_a \| ( \td \fn d - \td \ddj ) \cdot \nabla \tPj \|^2_{{\mathbb L}^2}
 \Bigr ] +\frac{1}{2}   \mathcal{W}( \fn u , \tu^j)
\\
 \leq&
 \Bigl( ( \fnn v \cdot \nabla) \fn v ,
\vvi
\Bigr) + \frac{1}{2}\Bigl( (\di \fnn v ) \fn v, 
\vvi
\Bigr) + \Bigl( ( \vvii\cdot\nabla) \vvi , \fn v \Bigr)+ \frac{1}{2}\Bigl( (\di \vvii ) \vvi, \fii v \Bigr)   
\\&-
\left (\fh d \times (  \PL [ \nabla \fiii d]\vvi ),\fh d \times  \fn q  \right )_h + \left (\dd^{j-1/2}\times( ( \fn v \cdot \nabla ) \ddii ) , \dd^{j-1/2}\times\tqj s\right ) _h
\\&+ 
 \left ( \fh d \times ( \PL[ \nabla\fiii d]\fn v  ),\fh d \times 
\tqj
 \right )_h + \left ( \ddi\times((\vvi\cdot\nabla) \ddii) ,\ddi\times \fn q \right ) _h
\\&+ \Bigl(([n^+]^j-[n^-]^j) \nabla \Phi^j  , 
\vvi
 \Bigr) + \Bigl ( ([\tn^+]^j-[\tn^-]^j) \nabla \tPj , \fn v \Bigr)  
\\ 
&- 
 \Bigl( (\fn v [n^+-n^-]^j , \nabla
  \tPj
   ) \Bigr)  - \Bigl(\vvi([\tn^+]^j-[\tn^-]^j),\nabla \Phi^j \Bigr)+
  \left ( \fh d \times \fn q , \fh d \times 
  \tqj
  \right )_h \\
 &+ \left ( \dd^{j-1/2} \times \tqj , \dd ^{j-1/2} \times \fn q \right ) _h- 2 \left ( \dd^{j-1/2} \times \tqj , \f d ^{j-1/2} \times \fn q \right ) _h
 \\&+
 \Bigl ( (\varepsilon(\fn d )-\varepsilon(\ddi)) \Bigl (  \nabla [n^{+}]^j-\nabla [\tn^{+}]^j \Bigr) ,  \nabla\tPj \Bigr ) + \Bigl ((\varepsilon(\ddi)-\varepsilon(\fn d)) \nabla [\tn^{+}]^j , \nabla [\Phi^j - \tPj] \Bigr) 
\\&-
 \left ( (\varepsilon(\fn d )-\varepsilon(\ddi)) \left (  \nabla [n^{-}]^j-\nabla [\tn^{-}]^j \right ) ,  \nabla\tPj \right ) - \Bigl ((\varepsilon(\ddi)-\varepsilon(\fn d)) \nabla [\tn^{-}]^j , \nabla [\Phi^j - \tPj] \Bigr )
 \\& +
  \left ( \varepsilon(\ddi) ([\tn^+]^j+[\tn^-]^j)- \varepsilon(\fn d)  ([n^+]^j+[n^-]^j) , \nabla \tPj \o(\nabla [\Phi^j -  \tPj] ) \right )
\\& +
  \left ( ( \varepsilon(\fiii d)- \varepsilon(\ddii) ) ( \nabla \Phi^j-\nabla \tPj ), \nabla \td \tPj \right  )  - \left ( \varepsilon( \fn d ) \nabla \Phi^j , \nabla \td ( \mathcal{I}_h - I ) \tPj \right )
\\
&+ k \left ( \td ( \varepsilon ( \fn d) \nabla \Phi^j), \nabla \td ( \mathcal{I}h-I)\tPj \right ) 
\\
&+ \varepsilon_a  
\left ( 
 \left ( \td \fn d - \td \ddj , \nabla \tPj \nabla \tPj \cdot ( \fn d - \ddj )\right ) + \left ( \td \ddj \cdot ( \nabla \Phi^j - \nabla \tPj ) , \nabla \tPj \ddj-\nabla \Phi^j \fn d \right ) \right ) 
\\& 
+
\varepsilon_a \left ( \td \ddj \cdot \nabla \tPj , ( \fn d - \ddj ) \cdot \nabla( \tPj-\Phi^j ) \right ) + \int_\Omega ( \mathcal{I}_h - I ) \left ( \td \fn d \cdot \tqj \right ) \de \f x 
 \\&+ 
 \left \langle \mathcal{A}_d^j(\tu^j) , \begin{pmatrix}
\vvi-\fn  v \\ \tqj - \fn q \\ \tPj- \Phi^j \end{pmatrix}  \right \rangle 
+ 
 \left \langle \mathcal{A}_d^j(\f u^j) , \begin{pmatrix}
( 
\mathcal{I}_h
 - I ) \vvi \\
( \mathcal{I}_h-I) \td \ddi 
\\ ( \mathcal{I}_h - I) \tqj \\ ( \mathcal{I}_h - I ) \tPj \\ (\mathcal{I}_h -I) ( [\tn^+]^j-[\tn^-]^j -\td \tPj ) 
\end{pmatrix}  \right \rangle 
\\&+ \varepsilon_a \frac{k }{2}  \| \td \ddj ( \nabla \tPj - \nabla \Phi^j) \|^2- \td \Bigl ( \varepsilon(\f d^j) \nabla \Phi^j, \nabla \left ( I- \mathcal{I}_h \right )\tP^j \Bigr )
\\
&+ h^\alpha \left ( \nabla \td \fn v , \nabla \vvi \right )  + h^ \beta \left ( \Delta _h \fh d , \Delta _h \td \ddi \right ) 
\\=:{}&I+ I_h+I_k+
 \left \langle \mathcal{A}_d^j(\tu^j) , \begin{pmatrix}
\vvi-\fn  v \\ \tqj - \fn q \\ \tPj- \Phi^j \end{pmatrix}  \right \rangle 
 \,,
\end{align*}
where we used the definition of the discrete solution operator for the continuous solution~\eqref{Adj} and similarly the definition of the discrete solution operator, which is given analogously to~\eqref{Adj}, but according to the discrete system~\eqref{dis},\textit{i.e.},
$ \langle \mathcal{A}_d^j(\f u^j ) , ( \f a , \f b , \f c , e^{\pm}, g)^T \rangle $ is given by the sum of the left-hand sides of the five equations in~\eqref{dis}. 

Above, $I_h$ abbreviates all terms that vanish for vanishing spacial discretization parameter $h$, \textit{i.e.},
\begin{align*}
I_h={}& 
 \left \langle \mathcal{A}_d^j(\f u^j) , \begin{pmatrix}
( \mathcal{I}_h - I ) 
\vvi \\
( \mathcal{I}_h-I) \td \ddi 
\\ ( \mathcal{I}_h - I) \tqj \\ ( \mathcal{I}_h - I ) \tPj \\ (\mathcal{I}_h -I) ( [\tn^+]^j-[\tn^-]^j -\td \tPj ) 
\end{pmatrix}  \right \rangle 
+ k \Bigl ( \td ( \varepsilon ( \fn d) \nabla \Phi^j), \nabla \td ( \mathcal{I}_h-I)\tPj \Bigr )  \\
&+ \int_\Omega ( \mathcal{I}_h - I ) \bigl ( \td \fn d \cdot \tqj \bigr ) \de \f x - \td \Bigl ( \varepsilon(\f d^j) \nabla \Phi^j, \nabla \left ( I- \mathcal{I}_h \right )\tP^j \Bigr )
\\
&+ h^\alpha \left ( \nabla \td \fn v , \nabla \vvi \right )  + h^ \beta \left ( \Delta _h \fh d , \Delta _h \td \ddi \right ) 
\\={}&
\td \left ( \fn v , ( \mathcal{I}_h-I) \vvi \right  ) + h^\alpha \left (  \nabla \td  \fii v , \nabla \mathcal{I}_h \vvi \right )   + h^\beta \left (  \Delta _h \fh d , \Delta _h \mathcal{I}_h  \td \ddi \right ) -  ( \fiii v , ( \mathcal{I}_h-I)\td \vvi )
\\&
+ \td \left ( \varepsilon(\f d^j) \nabla \Phi^j, \nabla \left (  \mathcal{I}_h-I  \right )\tP^j \right ) - \left ( ( \varepsilon ( \fiii d) \nabla \Phi^{j-1}), \nabla \td ( \mathcal{I}_h-I)\tPj \right )
\\& + \left ( (\fii v \cdot \nabla ) \fii v , ( \mathcal{I}_h-I)  \vvi\right ) + \frac{1}{2}\left ( ( \di \fiii v) \fii v , ( \mathcal{I}_h-I)  \vvi\right ) + \left ( ( [ n^+]^j-[n^-]^j)\nabla \Phi^j ,( \mathcal{I}_h-I)  \vvi\right ) 
\\& +\left ( \nabla \fh d , \nabla ( \mathcal{I}_h-I)  \td \ddj ) \right ) - \varepsilon _a \left ( \nabla \Phi^j \left ( \fiii d \cdot \nabla \Phi^j \right ) , ( \mathcal{I}_h-I)  \td \ddj \right ) \\
&+ \Bigl ( \varepsilon(\fn d ) \nabla [n^{\pm}]^j , \nabla ( \mathcal{I}_h-I)  \tPj \Bigr ) \pm \Bigl ( [ n^{\pm}]^j \varepsilon(\ddj) \nabla \Phi^j , \nabla ( \mathcal{I}_h-I)  \tPj \Bigr ) 
\\&- \Bigl ( \fii v [n^{\pm} ]^j , \nabla ( \mathcal{I}_h-I) \tPj \Bigr )  \pm \Bigl( \varepsilon(\fii d ) \nabla \Phi^j , \nabla [\tn^{\pm}]^j \Bigr) 
\,.
\end{align*}
Note that all terms including mass-lumping vanish immediately since 
\begin{align*}
\left ( \f a ,  \mathcal{I}_h\f b\right ) _h = \left ( \f a , \f b \right ) _h \,.
\end{align*}
The term $I_k$ abbreviates the term vanishing for vanishing temporal discretization parameter, \textit{i.e.},
\begin{align*}
I_k =  \varepsilon_a \frac{k }{2}  \| \td \ddj ( \nabla \tPj - \nabla \Phi_h^j) \|^2_{{\mathbb L}^2} \leq \varepsilon_a k \| \td \ddj \|_{{\mathbb L}^\infty}^2 \| \nabla (\Phi^j - \tPj) \|_{{\mathbb L}^2}^2 \,,
\end{align*}
and $I$ incorporates the terms, which are similar to the continuous case, such that the associated manipulations are also similar, \textit{i.e.},
\begin{align*}
I={}&  ( ( \fnn v \cdot \nabla) \fn v ,
\vvi
) + \frac{1}{2}\bigl ( (\di \fnn v ) \fn v, 
\vvi
\bigr ) + \bigl( ( \vvii\cdot\nabla) \vvi , \fn v \bigr)+ \frac{1}{2}\bigl ( (\di \vvii ) \vvi, \fii v \bigr )   
\\&-
\left (\fh d \times ( \PL [ \nabla \fiii d]\vvi  ),\fh d \times  \fn q  \right )_h + \left (\dd^{j-1/2}\times( ( \fn v \cdot \nabla ) \ddii ) , \dd^{j-1/2}\times\tqj \right ) _h
\\&+ 
 \left ( \fh d \times (  \PL [ \nabla \fiii d]\fn v ),\fh d \times 
\tqj
 \right )_h + \Bigl ( \ddh\times((\vvi\cdot\nabla) \ddii) ,\ddh\times \fn q \Bigr ) _h
\\&+ \Bigl (([n^+]^j-[n^-]^j) \nabla \Phi^j  , 
\vvi
 \Bigr ) + \Bigl ( ([\tn^+]^j-[\tn^-]^j) \nabla \tPj , \fn v \Bigr )  
\\ 
&- 
 \Bigl ( (\fn v [n^+-n^-]^j , \nabla
  \tPj
   ) \Bigr)  - \Bigl(\vvi([\tn^+]^j-[\tn^-]^j),\nabla \Phi^j \Bigr)\\
 &+
  \Bigl( \fh d \times \fn q , \fh d \times 
  \tqj
  \Bigr) _h+ \Bigl ( \dd^{j-1/2} \times \tqj , \dd ^{j-1/2} \times \fn q \Bigr) _h- 2 \Bigl( \dd^{j-1/2} \times \tqj , \f d ^{j-1/2} \times \fn q \Bigr) _h
 \\&+
 \left ( (\varepsilon(\fn d )-\varepsilon(\ddi)) \left (  \nabla [n^{+}]^j-\nabla [\tn^{+}]^j \right ) ,  \nabla\tPj \right ) + \left ((\varepsilon(\ddi)-\varepsilon(\fn d)) \nabla [\tn^{+}]^j , \nabla \Phi^j -\nabla \tPj \right ) 
\\&-
\left ( 
 \left ( (\varepsilon(\fn d )-\varepsilon(\ddi)) \left (  \nabla [n^{-}]^j-\nabla [\tn^{-}]^j \right ) ,  \nabla\tPj \right ) + \left ((\varepsilon(\ddi)-\varepsilon(\fn d)) \nabla [\tn^{-}]^j , \nabla \Phi^j -\nabla \tPj \right )
\right) 
 \\& +
  \left ( \varepsilon(\ddi) ([\tn^+]^j+[\tn^-]^j)- \varepsilon(\fn d)  ([n^+]^j+[n^-]^j) , \nabla \tPj \o(\nabla \Phi^j - \nabla \tPj ) \right )
\\& +
  \left ( ( \varepsilon(\fiii d)- \varepsilon(\ddii) ) ( \nabla \Phi^j-\nabla \tPj ), \nabla \td \tPj \right  ) 
\\
&+ \varepsilon_a  
\left ( 
 \left ( \td \fn d - \td \ddj , \nabla \tPj \nabla \tPj \cdot ( \fn d - \ddj )\right ) + \left ( \td \ddj \cdot ( \nabla \Phi^j - \nabla \tPj ) , \nabla \tPj \ddj-\nabla \Phi^j \fn d \right ) \right ) 
\\& 
+
\varepsilon_a \left ( \td \ddj \cdot \nabla \tPj , ( \fn d - \ddj ) \cdot \nabla( \tPj-\Phi^j ) \right ) 
\\={} &\left ( (( \fiii v - \vvii)\cdot \nabla) ( \fn v - \vvi ) , \vvi \right ) + \frac{1}{2} \left ( ( \di \fiii v ) \vvi , \fii v - \vvi \right ) 
\\
 &+ \left ( \fh d\times \left (  ( \PL[\nabla \fiii d] - \nabla \ddii ) (\fii v - \vvi) \right ) , \fh d \times \tqj \right ) _h \\&+ \left ( ( \fh d -\ddh ) \times\left (  (( \fii v -\vvi) \cdot \nabla )\ddii \right ), \fh d \times \tqj\right ) _h
 \\ 
& + \left ( \ddh \times \left ( ( \fii v -\vvi) \cdot\nabla ) \ddii \right ), ( \fh d -\ddh) \times \tqj\right ) _h
\\&  +
  \left ( \fh d \times \left (  (\PL[\nabla\fiii d] - \nabla \ddii )\vvi \right ), \ddh \times \tqj-\fh d \times \fn q\right ) _h
 \\
 &+ \left ( ( \fh d - \ddh ) \times \left ( (\vvi \cdot \nabla) \ddii \right ) , \ddh \times \tqj - \fh d \times \fn q \right )_h
\\&  + \left ( \fh d \times \left ((\mathcal{P}[\nabla \fiii d] - \nabla\ddii ) \vvi \right) , ( \fh d - \ddh ) \times \tq\right ) _h
 \\&+   \left ((\fh d - \ddh ) \times \left ( ( \vvi\cdot\nabla)\ddii\right ) , ( \fh d - \ddh ) \times \tqj\right )_h\\& +  \left ( \fh d \times \fn q -\ddh\times \tqj, (\fh d-\ddh) \times \tqj\right ) _h
 \\&+
  \left ( \nabla \mathcal{I}_h \left (( \fh d -\ddh ) \times\tilde{\f a}^j\right ) ,   \nabla \fh d - \nabla \ddh  \right )
  \\
& 
+h^\beta 
  \left ( \Delta _h  \mathcal{I}_h \left (( \fh d -\ddh ) \times \tilde{\f a}^j\right ) ,   \Delta_h \fh d \right ) 
\\{}&-\varepsilon_a  \left (\mathcal{I}_h \left ( (  \fh d- \ddh  )\times  \tilde{\f a}^j\right ) , \left ( \nabla \Phi^j (\nabla \Phi^j \cdot \fiii d) - \nabla \tPj ( \nabla \tPj \cdot \ddii)\right ) \right ) 
\\
 &+ \Bigl( ([n^+]^j-[n^-]^j) -([\tn^+]^j-[\tn^-]^j) , \vvi \cdot  (\nabla \Phi^j-\nabla\tPj) \Bigr ) \\
  &  + \Bigl ( ([n^+]^j-[n^-]^j) -([\tn^+]^j-[\tn^-]^j) , (\vvi -\fn v)\cdot \nabla\tPj  \Bigr ) 
\\&-    \Bigl (  ([n^+]^j-[n^-]^j) -([\tn^+]^j-[\tn^-]^j)  , \di (\varepsilon(\fn d )-\varepsilon(\ddi)) \cdot \nabla  \tPj + (\varepsilon(\fn d )-\varepsilon(\ddj)) : \nabla^2 \tPj \Bigr ) 
\\& + \Bigl ((\varepsilon(\ddj)-\varepsilon(\fn d))\left (  \nabla ([\tn^+]^j-[\tn^-]^j)\right ) , \nabla (\Phi^j -  \tPj) \Bigr )
\\&
 +\Bigl ( (\varepsilon(\ddj) -\varepsilon(\fn d))([\tn^+]^j-[\tn^+]^j)
, \nabla \tPj \o(\nabla (\Phi^j - \tPj)) \Bigr )\\&
 +\Bigl ( \varepsilon(\f d)\left (([n^+]^j-[n^-]^j) -([\tn^+]^j-[\tn^-]^j) \right ), \nabla \tPj \o(\nabla (\Phi^j - \tPj)) \Bigr )\\
&+ \varepsilon_a  
 \Bigl( \td \fn d - \td \ddj , \nabla \tPj \nabla \tPj \cdot ( \fn d - \ddj )\Bigr ) + \Bigl ( \td \ddj \cdot \nabla (\Phi^j -  \tPj ) , \nabla (\tPj \ddj- \Phi^j \fn) d \Bigr )  
\\& 
+  \Bigl( ( \varepsilon(\fiii d)- \varepsilon(\ddii) ) ( \nabla \Phi^j-\nabla \tPj ), \nabla \td \tPj \Bigl  ) 
+ 
\varepsilon_a \Bigl ( \td \ddj \cdot \nabla \tPj , ( \fn d - \ddj ) \cdot \nabla( \tPj-\Phi^j ) \Bigr) 
\\& - \left ( \nabla \ddh , \nabla \left ( \left ( \mathcal{I}_h-I\right ) \left ( (\ddh - \fh d ) \times \tilde{\f a}^j\right ) \right )\right )
 \\&+ \varepsilon_a \left ( \nabla \tPj ( \nabla \tPj \cdot \ddii ) , \left ( \mathcal{I}_h-I\right ) \left ( (\ddh - \fh d ) \times  \tilde{\f a}^j \right )\right )
\\& + \int_\Omega  \left ( I - \mathcal{I}_h \right ) \left ( \left (( \fh d - \ddh ) \times  \tilde{\f a}^j \right ) \cdot \tqj \right ) \de \f x 
 \,,
 \end{align*}
where we employed~\eqref{dis:q} and the definition of $\tqj$ (see Proposition~\ref{prop:disrel}).
For convenience, we introduced the abbreviation $ \tilde{\f a}^j := \ddh \times ( ( \vvi \cdot \nabla ) \ddii + \tqj) $. 

The term incorporating the difference in the discrete time derivative may be handled as follows: for any function $ \f a : \ov \Omega \times [0,T] \ra \R^d$, we find
\begin{align}
\begin{split}
&\left ( \td \fn d - \td \ddj , \f a \right ) =  \left \langle \mathcal{A}_d^j( \f u^j), \begin{pmatrix}
\f 0\\(I - \mathcal{I}_h)\f a \\0
\end{pmatrix}  \right \rangle-\left \langle \mathcal{A}_d^j( \tu^j), \begin{pmatrix}
\f 0\\\f a\\ 0 
\end{pmatrix} \right \rangle + \left ( \td \fn d , \f a \right ) - \left ( \td \fn d , \f a \right )_h 
\\
&\qquad + \left ( \dd^{j-1/2} \times (( \vvi\cdot \nabla) \ddii), \dd^{j-1/2} \times \f a \right )_h - \left ( \fh d \times( \PL[ \nabla  \fiii d] \fii v , \fh d \times \f a   \right )_h
\\
&\qquad + \left ( \dd^{j-1/2} \times \tqj , \dd^{j-1/2} \times \f a \right ) _h- \left ( \fh d \times \fn q , \fh d \times \f a   \right ) 
\\&\quad = \left \langle \mathcal{A}_d^j( \f u^j), \begin{pmatrix}
0\\(I - \mathcal{I}_h)\f a \\0
\end{pmatrix}  \right \rangle-\left \langle \mathcal{A}_d^j( \tu^j), \begin{pmatrix}
0\\\f a\\ 0 
\end{pmatrix} \right \rangle +\int_\Omega ( I - \mathcal{I}_h)( \td \fn d , \f a ) \de \f x 
\\
&\qquad + \left ( \dd^{j-1/2} \times (( \vvi\cdot \nabla) \ddii), ( \dd^{j-1/2}- \fh d)  \times \f a \right ) _h
\\&\qquad + \left ( ( \dd^{j-1/2} - \fh d ) \times (( \vvi\cdot \nabla) \ddii), \fh d \times \f a \right )_h
\\&\qquad + \left ( \fh d \times \left (
 \PL[\nabla \fiii d]( \vvi-\fii v ) 
 \right ) , \fh d \times \f a \right ) _h\\&\qquad + \left ( \fh d \times ( (\PL[\nabla \fiii d] -\nabla  \ddii )\vvi ) , \fh d \times \f a \right ) _h
\\
&\qquad + \left ( \dd^{j-1/2} \times \tqj , (\dd^{j-1/2} - \fh d )  \times \f a \right )_h + \left ( \dd^{j-1/2} \times \tqj -  \fh d \times \fn q , \fh d \times \f a   \right )_h \,.
\end{split}\label{timedis}
\end{align}
Concerning the ${\mathbb L}^2$-projection, we may estimate
\begin{align*}
\| \PL(\nabla \fiii d) -\nabla  \ddii \|_{{\mathbb L}^2} \leq{}& \| \PL(\nabla \fiii d) -\PL(\nabla  \ddii) \|_{{\mathbb L}^2} + \| \PL(\nabla \ddii ) - \nabla \ddii \|_{{\mathbb L}^2}
\\
\leq{}&  \| \nabla (\fiii d -   \ddii) \|_{{\mathbb L}^2} + c h  \| \nabla^2  \ddii \|_{{\mathbb L}^2} \,. 
\end{align*}
Inserting this, we may start to estimate the right-hand side of the relative energy inequality.
Note that the interpolation operator is stable with respect to the  ${\mathbb H}^1$- and ${\mathbb L}^\infty $-norm, \textit{i.e.},
\begin{align*}
\|\nabla  \mathcal{I}_h [ f ] \|_{{\mathbb L}^2} \leq C \|\nabla  f \|_{{\mathbb L}^2} \text \quad \forall\, f \in {\mathbb H}^1\,,  \qquad\text{and}\qquad \| \mathcal{I}_h[f]\|_{{\mathbb L}^\infty} \leq \| f \|_{{\mathbb L}^\infty}  \quad \forall\, f \in \C (\ov\Omega)\,.
 \end{align*}

First, we may estimate the terms also occurring in the continuous setting as in the proof of Proposition~\ref{prop:cont}; we will not repeat the details here. Keeping only the additional terms stemming from the discretization, 
we end up with
\begin{align*}
&\td \mathcal{R}( \f u^j | \tu^j ) + \td \frac{h^\alpha }{2}\| \nabla 
\f v^j
 \|^2_{{\mathbb L}^2} 
 + \frac{h^\beta}{2} \| \Delta _h 
  \fn d 
   \|_h^2 
   +\td  r_k^1(h) \\
&\quad +\frac{ k}{2}  \left [\| \td \fn v -  d_t \vv^j \|^2_{{\mathbb L}^2} + h^\alpha \|  \td \nabla \fn v
 \| ^2_{{\mathbb L}^2}
 + \| \nabla \td \fn d - \nabla d_t \dd^j \|^2_{{\mathbb L}^2}
    \right ]
\\
&\quad +\frac{ k}{2}  \Bigl [ \Bigl(\varepsilon(\dd^{j-1})\nabla d_t (\Phi_h^j- \tP^j) , \nabla d_t (\Phi_h^j-\tP^j) \Bigr) \\
&\quad +  \varepsilon_a   \| ( \td \fn d - \td \ddj ) \cdot \nabla \Phi_h^j \|^2_{{\mathbb L}^2} + \varepsilon_a \| ( \td \fn d - \td \ddj ) \cdot \nabla \tPj \|^2_{{\mathbb L}^2} 
 \Bigr ]
 +\frac{1}{2} \mathcal{W}( \fn u , \tu^j)
\\
 \leq&
 \Bigl (\mathcal{K}_1(\tu ^j,\tu^{j-1})+\varepsilon_a \frac{k}{2} \| \td \ddj\|_{{\mathbb L}^\infty}^2 \Bigr ) \mathcal{R}( \fn u | \tu^j ) \\
 & +
\mathcal{K}_2(\tu ^j,\tu^{j-1}) \left ( \mathcal{R}( \fiii u | \tu^{j-1} ) + \| \PL ( \nabla \ddii ) - \nabla \ddii \|_{{\mathbb L}^2}^2  \right )   \\& +\frac{1}{2} \left [
  \|  [n^+]^j-[\tn^+]^j \|_{{\mathbb L}^2}^2 + \|  [n^-]^j-[\tn^-]^j \|_{{\mathbb L}^2}^2 \right  ]
 \\&+ 
 \left \langle \mathcal{A}_d^j(\tu^j) , \begin{pmatrix}
\vvi-\fn  v \\ \tqj - \fn q + \nabla \tPj \left (\nabla \tPj \cdot ( \ddj - \fn d ) \right ) \\ \tPj-  \Phi^j\end{pmatrix}  \right \rangle 
\\&+  \left \langle \mathcal{A}_d^j( \f u^j), \begin{pmatrix}
0\\(I - \mathcal{I}_h)\left (\nabla \tPj \left (\nabla \tPj \cdot ( \fn d- \ddj  \right  )\right ) \\0
\end{pmatrix}  \right \rangle
\\&- \int_\Omega ( I - \mathcal{I}_h)( \td \fn d \cdot(
  \nabla \tPj \nabla \tPj \cdot ( \ddj - \fn d ) ) ) \de \f x   + I_h
\\& 
- \left ( \nabla \ddh , \nabla \left ( \left ( \mathcal{I}_h-I\right ) \left ( (\ddh - \fh d ) \times \tilde{\f a}^j\right ) \right )\right )
\\&
 + h^\beta \left ( \Delta \fh d , \Delta \mathcal{I}_h \left ( \left ( (\ddh - \fh d ) \times \tilde{\f a}^j\right ) \right )\right )
 \\&+ \varepsilon_a \left ( \nabla \tPj ( \nabla \tPj \cdot \ddii ) , \left ( \mathcal{I}_h-I\right ) \left ( (\ddh - \fh d ) \times \tilde{\f a}^j\right )\right )
\\& + \int_\Omega  \left ( I - \mathcal{I}_h \right ) \left ( \left (( \fh d - \ddh ) \times  \tilde{\f a}^j \right ) \cdot \tqj \right ) \de \f x 
\\
 \leq&
 \left (\mathcal{K}_1(\tu ^j,\tu^{j-1}) +\varepsilon_a \frac{k}{2} \| \td \ddj\|_{L^\infty(\Omega)}^2 \right ) \mathcal{R}( \fn u | \tu^j )+
\mathcal{K}_2(\tu^j,\tu ^{j-1}) \mathcal{R}( \fiii u | \tu^{j-1} )+ r^2_k(h) 
\\
&
 + h^\beta \left \| \Delta \fh d \right \|_{\mathbb L^2} ^2    \left \| \ddh \times \left ( ( \vvi \cdot \nabla) \ddii + \tqj \right ) \right \|_{\mathbb W^{2,\infty}} 
 \\& +\frac{1}{2} \left [
  \| [n^+]^j-[\tn^+]^j \|_{{\mathbb L}^2}^2 + \| [n^-]^j-[\tn^-]^j \|_{{\mathbb L}^2}^2 \right  ]
\\& + 
 \left \langle \mathcal{A}_d^j(\tu^j) , \begin{pmatrix}
\vvi-\fn  v \\ \tqj - \fn q + \nabla \tPj\left ( \nabla \tPj \cdot ( \ddj - \fn d )\right )   \\ \tPj- \Phi^j \end{pmatrix}  \right \rangle 
 \,.
\end{align*}
Above, we defined 
\begin{align*}
r^1_k(h) :=  \Bigl ( \fn v , ( \mathcal{I}_h-I) \vvi \Bigr  ) 
+  \Bigl( \varepsilon(\f d^j) \nabla \Phi^j, \nabla \left (  \mathcal{I}_h-I  \right )\tP^j \Bigr) 
\end{align*}
and correspondingly 
\begin{align*}
r^2_k(h):={}& I_h - r^1_k(h) - \int_\Omega ( I - \mathcal{I}_h)\Bigl( \td \fn d \cdot \bigl( \tqj  + \nabla \tPj \nabla \tPj \cdot [ \ddj - \fn d ] \bigr) \Bigr) \de \f x   
\\
&+ \Biggl ( \tilde{\f a}^j, ( \ddh - \fh d ) \times \left ( I - \mathcal{I}_h\right ) \fn q  \Biggr)  
\\& 
- \Biggl ( \nabla \ddh , \nabla \Biggl [ \left ( \mathcal{I}_h-I\right ) \left ( (\ddh - \fh d ) \times\tilde{\f a}^j\right ) \Biggr ]\Biggr)
 \\&+ \varepsilon_a \Biggl( \nabla \tPj ( \nabla \tPj \cdot \ddii ) , \left ( \mathcal{I}_h-I\right ) \left ( (\ddh - \fh d ) \times \tilde{\f a}^j\right )\Biggr )
\\& + \int_\Omega  \left ( I - \mathcal{I}_h \right ) \Bigl ( ( \fh d - \ddh ) \times \tilde{\f a}^j \cdot \tqj \Bigr ) \de \f x 
 \\
& 
+  \mathcal{K}_2(\tu^j,\tu^{j-1})  \| \mathcal{P}_{{\mathbb L}^2} (\nabla \ddii) - \nabla \ddii \| ^2_{{\mathbb L}^2} 
\\& 
+ h^\beta \left ( \Delta_h \fh d , \Delta _h \mathcal{I}_h \left ( \left ( (\ddh- \fh d) \times \tilde{\f a}^j\right ) \right )\right )
\\& 
- h^\beta \left ( \Delta_h \fh d ,  \mathcal{I}_h \left ( \left ( (\ddh- \Delta _h\fh d) \times \tilde{\f a}^j\right ) \right )\right )
\,.
\end{align*}

Thus, it remains to show that $r^1_k(h)\ra 0$ and  $\sumi | r^2_k(h)| \ra 0 $ as $h\ra 0$. 
With regard to $r^1_k$, we may estimate by Lemma~\ref{lem:interpolation}
\begin{align*}
| r^1_k(h)| \leq{}&  \| \fn v \|_{{\mathbb L}^2} \| ( \mathcal{I}_h-I) \vvi \|_{{\mathbb L}^2} 
+ \| \varepsilon(\f d^j) \nabla \Phi^j\|_{{\mathbb L}^2} \| \nabla \left (  \mathcal{I}_h-I  \right )\tP^j \|_{{\mathbb L}^2} \\
 \leq {}& C h \Bigl (  \| \fn v \|_{{\mathbb L}^2} \| \nabla  \vvi \|_{{\mathbb L}^2} 
 + \| \varepsilon(\f d^j) \nabla \Phi^j\|_{{\mathbb L}^2} \| \nabla ^2 \tP^j \|_{{\mathbb L}^2}\Bigr )\,.
\end{align*}
For $r^2_k$, we may conclude
\begin{align*}
&\sumi |r^2_k(h)|-  \mathcal{K}_2(\tu^j,\tu^{j-1})  \| \mathcal{P}_{{\mathbb L}^2} (\nabla \ddii) - \nabla \ddii \| ^2_{{\mathbb L}^2} \\
&\leq 
 k \sum_{j=1}^J \Bigl[ \| \fiii v \|_{{\mathbb L}^2}  \| ( \mathcal{I}_h-I)\td \vvi \|_{{\mathbb L}^2} 
+ h^\alpha  \| \nabla \fiii v\|_{{\mathbb L}^2} \|  \nabla \td \vvi \|_{{\mathbb L}^2} 
\\
&\qquad 
+   \| \varepsilon(\fiii d) \nabla \Phi^{j-1} \|_{{\mathbb L}^2} 
\|  \td \nabla \left ( \mathcal{I}_h-I\right ) \tPj\|_{{\mathbb L}^2}   
 + \| \nabla \fn v \|_{{\mathbb L}^2}  \| \nabla ( \mathcal{I}_h - I ) \vvi \|_{{\mathbb L}^2} \Bigr]\\
&\quad + k \sum_{j=1}^J \Bigl[\left (\| \fiii v \|_{{\mathbb L}^{10/3}} \| \nabla \fii v \|_{{\mathbb L}^2} + \| \di \fiii v \|_{{\mathbb L}^2} \| \fii v \|_{{\mathbb L}^{10/3}}\right ) \| (\mathcal{I}_h - I ) \vvi \|_{{\mathbb L}^{5}}  \Bigr]  \\
&\quad +  k\sum_{j=1}^J \Bigl[ \| [ n^{+}]^j - [n^-]^j \|_{{\mathbb L}^\infty} \| \nabla \Phi^j \|_{{\mathbb L}^2} \| (\mathcal{I}_h - I ) \vvi \| _{{\mathbb L}^2}\Bigr]
\\
&\quad +\sumii{  \| \nabla \fh d \|
 \| \nabla (\mathcal{I}_h - I ) \td \ddj \|
 + \varepsilon_a \| \nabla \Phi^j \|_{{\mathbb L}^2}^2 \| \fh d \|_{{\mathbb L}^\infty} \| ( \mathcal{I}_h - I) \td \ddj \|_{{\mathbb L}^\infty} } \\
&\quad +k\sum_{j=1}^J \Bigl[ \Bigl(\| \varepsilon(\fn d ) \nabla [n^{\pm}]^j \|_{{\mathbb L}^2}
+ \| [n^{\pm}]^j \varepsilon(\fn d) \nabla \Phi^j \|_{{\mathbb L}^2}
+ \| \fn v [n^{\pm}]^j \|_{{\mathbb L}^2}
 \Bigr ) \| \nabla ( \mathcal{I}_h- I) \tPj \|_{{\mathbb L}^2}
 \Bigr]
\\
&\quad +k\sum_{j=1}^J \Bigl[\|  \varepsilon(\fn d) \nabla \Phi^j \|_{{\mathbb L}^2}
 \left \|
 (\mathcal{I}_h -I) ( [\tn^+]^j-[\tn^-]^j ) 
 \right \|_{{\mathbb L}^2}
 \Bigr]
\\& \quad
  + hk \sum_{j=1}^J \Biggl[ \| \td \fn d\|_{{\mathbb L}^{3/2}}  \left[ \| \nabla  \tqj \|
  _{{\mathbb L}^3} 
  + \| \nabla \left ( \nabla \tPj \bigl[ \nabla \tPj \cdot ( \ddj - \fn d ) \bigr] \right) \|
  _{{\mathbb L}^3}
  \right] \Biggr]
\\
&\quad +k\sum_{j=1}^J \| \nabla \ddh  \|_{{\mathbb L}^\infty} \left \| \nabla \Biggl[ \left ( \mathcal{I}_h-I\right ) \Biggl ( (\ddh - \fh d ) \times \tilde{\f a}^j\Biggr ) \Biggr]\right \|_{{\mathbb L}^1} 
\\
&\quad + \sumii{\| \nabla \tPj \|_{{\mathbb L}^4}^2 \|  \ddii  \|_{{\mathbb L}^\infty} \left \|  \left ( \mathcal{I}_h-I\right ) \Biggl ( (\ddh - \fh d ) \times \tilde{\f a}^j\Biggr)\right \|_{\mathbb L^2 }
}
\\&\quad +  h \sumii{\| \fh d - \ddh \|_{{\mathbb H}^1} \| \tilde{\f a}^j\cdot \tqj \|_{{\mathbb H}^1} }
\\& \quad +\sumii{ h^\beta \left \| \Delta \fh d \right \| _{\mathbb L^2} \left \| \Delta \mathcal{I}_h \left ( \left ( \ddh  \times \tilde{\f a}^j\right ) \right )\right\|_{\mathbb L^2}}
\\
& \quad +\sumii{ h^\beta \left \| \Delta_h  \fh d \right \| _{\mathbb L^2} \left \| \fh d \right \|_{\mathbb{W}^{1,2}} \left \|  \tilde{\f a}^j\right\|_{\mathbb W^{1,\infty}}}
\\
&+\quad  \sumii{  h^\alpha  \left \|  \nabla   \fii v \right \|_{\mathbb L^2} \left \| \nabla  \td  \vvi \right \|_{\mathbb L^2} + h^\beta \left \|  \Delta _h \fh d \right \| _{\mathbb L^2} \left \| \Delta _h   \td \ddi \right \|_{\mathbb L^2} } 
\\
\leq{}& 
h  \sumii{ \| \fiii v \|_{{\mathbb L}^2}  \| \nabla \td \vvi \|_{{\mathbb L}^2}
+ h^\alpha  \| \nabla \fiii v\|_{{\mathbb L}^2} \|  \nabla^2 \td \vvi \|_{{\mathbb L}^2}
+   \| \varepsilon(\fiii d) \nabla \Phi^{j-1} \|_{{\mathbb L}^2} \|  \td \nabla ^2 \tPj\|_{{\mathbb L}^2}}
\\&+h
\sumii{ \| \nabla \fn v \|_{{\mathbb L}^2}  \| \nabla ^2 \vvi \|_{{\mathbb L}^2}}
\\& +h \sumii{ \left (\| \fiii v \|_{{\mathbb L}^{10/3}} \| \nabla \fii v \|_{{\mathbb L}^2} + \| \di \fiii v \|_{{\mathbb L}^2} \| \fii v \|_{{\mathbb L}^{10/3}}\right ) \| \nabla  \vvi \|_{{\mathbb L}^{5}} }  \\&+  \sumii{\| [ n^{+}]^j - [n^-]^j \|_{{\mathbb L}^\infty} \| \nabla \Phi^j \|_{{\mathbb L}^2} \| \nabla  \vvi \| _{{\mathbb L}^2} }
\\
&+h \sumii{  \| \nabla \fh d \|_{{\mathbb L}^2}
 \| \nabla ^2 \td \ddj \|_{{\mathbb L}^2}
 + \varepsilon_a \| \nabla \Phi^j \|_{{\mathbb L}^2}^2 \| \fh d \|_{{\mathbb L}^\infty} \| \nabla  \td \ddj \|_{{\mathbb L}^\infty} } \\
&+ hk\sum_{j=1}^J \Bigl[\| \varepsilon(\fn d ) \nabla [n^{\pm}]^j \|_{{\mathbb L}^2}
+ \| [n^{\pm}]^j \varepsilon(\fn d) \nabla \Phi^j \|_{{\mathbb L}^2}
+ \| \fn v [n^{\pm}]^j \|_{{\mathbb L}^2}
 \Bigr] \| \nabla ^2 \tPj \|_{{\mathbb L}^2}
\\
&+hk\sum_{j=1}^J \|  \varepsilon(\fn d) \nabla \Phi^j \|_{{\mathbb L}^2}
 \left \|
 \nabla  ( [\tn^+]^j-[\tn^-]^j ) 
 \right \|_{{\mathbb L}^2}
\\& 
  +hk
   \sum_{j=1}^J  \| \td \ddj \|_{{\mathbb L}^{3/2}}  \Bigl[ \| \nabla  \tqj \|_{{\mathbb L}^3}
  + \| \nabla \bigl[ \nabla \tPj \bigl( \nabla \tPj \cdot ( \ddj - \fn d ) \bigr) \bigr] \|_{{\mathbb L}^3}
  \Bigr]
\\
& +h \sumii{\| \nabla \ddh  \|_{{\mathbb L}^\infty}  \left \| \nabla   \fh d  \right \|_{{\mathbb L}^2}  \left \| \ddh \times \bigl[ ( \vvi \cdot \nabla) \ddii + \tqj \bigr]\right \|_{{\mathbb W}^{2,6/5}} } 
\\
& +h \sumii{\| \nabla \ddh  \|_{{\mathbb L}^\infty}   \left \| \ddh \times \left (\ddh \times \bigl[ ( \vvi \cdot \nabla) \ddii + \tqj \bigr] \right ) \right \|_{{\mathbb W}^{2,1}} } 
\\
&+h  \sumii{\| \nabla \tPj \|_{{\mathbb L}^4}^2 \|  \ddii  \|_{{\mathbb L}^\infty} \left \|  \fh d \right \|_{{\mathbb W}^{1,2}}   \left \|  \ddh \times \bigl[ ( \vvi \cdot \nabla) \ddii + \tqj \bigr] \right \|_{{\mathbb W}^{2,3}\cap \mathbb L^\infty} }
\\
&+h  \sumii{\| \nabla \tPj \|_{{\mathbb L}^4}^2 \|  \ddii  \|_{{\mathbb L}^\infty}  \left \|\ddh \times \left (  \ddh \times \bigl[ ( \vvi \cdot \nabla) \ddii + \tqj \bigr]\right )  \right \|_{{\mathbb W}^{2,2}} 
}
\\&+  h \sumii{\| \fh d - \ddh \|_{{\mathbb H}^1} \| \left ( \ddh \times \bigl[ ( \vvi \cdot \nabla ) \ddii + \tqj) \bigr]\right ) \cdot \tqj \|_{{\mathbb H}^1} }
\\&  +h^{\beta/2} \sumii{ h^{\beta/2} \left \| \Delta \fh d \right \| _{\mathbb L^2} \left \| \Delta \mathcal{I}_h \left ( \left ( \ddh  \times \left ( \ddh \times \left ( ( \vvi \cdot \nabla) \ddii + \tqj \right ) \right )\right ) \right )\right\|_{\mathbb L^2}}
\\
&  +h^{\beta/2}\sumii{ h^{\beta/2} \left \| \Delta_h  \fh d \right \| _{\mathbb L^2} \left \| \fh d \right \|_{\mathbb{W}^{1,2}} \left \|  \left ( \ddh \times \left ( ( \vvi \cdot \nabla) \ddii + \tqj \right ) \right )\right\|_{\mathbb W^{1,\infty}}}
\\
&+  \sumii{  h^{\alpha/2}   \left (h^{\alpha/2} \left \|  \nabla   \fii v \right \|_{\mathbb L^2} \right )\left \| \nabla  \td  \vvi \right \|_{\mathbb L^2} + h^{\beta/2}\left (h^{\beta/2} \left \|  \Delta _h \fh d \right \| _{\mathbb L^2}\right ) \left \| \Delta _h   \td \ddi \right \|_{\mathbb L^2} } 
\\
\,.
\end{align*}
Above, we used Lemma~\ref{lem:interpolation}, Lemma~\ref{lem:masslumping}, and~\eqref{estdtime}.

\end{proof}
\begin{corollary}
Let the assumptions of Proposition~\ref{prop:disrel} be fulfilled.
Then it holds that
\begin{multline}
-k \sum_{j=0}^J \td \phi^{j+1}\left ( \mathcal{R}( \fn u | \tu^j) + \frac{h^\alpha}{2}\| \nabla
 \fn v
 \|^2_{{\mathbb L}^2}  + \frac{h^\beta}{2}\| \Delta_h \fn d \|_{\mathbb L^2}^2 + r^1_k(h)  \right )   \prod_{l=1}^j \frac{1}{\omega^l } \\+ \sum_{j=1}^J \phi^j \frac{\mathcal{W}_d(\fn u | \tu^j ) }{1- k \mathcal{K}_1(\tu^j) }\prod_{l=1}^j\frac{1}{\omega^l }  
\\
\leq \phi^0\left (  \mathcal{R}( \f u ^0| \tu^0)  + \frac{h^\alpha}{2}\| \nabla \f v^0 
\|^2_{{\mathbb L}^2}  + \frac{h^\beta }{2 } \| \Delta _h \f d^0 \| _{\mathbb{ L}^2} \right ) 
\\+ \sum_{j=1}^J \phi^j \frac{1 }{1- k \mathcal{K}_1(\tu^j) }
\frac{1}{2} \Bigl[
  \| [n^+]^j-[\tn^+]^j \|_{{\mathbb L}^2}^2 + \| [n^-]^j-[\tn^-]^j \|_{{\mathbb L}^2}^2 \Bigr]
\prod_{l=1}^j\frac{1}{\omega^l } \\+ \sum_{j=1}^J \phi^j \frac{1 }{1- k \mathcal{K}_1(\tu^j) }
\left ( 
 \left \langle \mathcal{A}_d^j(\tu^j) , \begin{pmatrix}
\vvi-\fn  v \\ \tqj - \fn q + \nabla \tPj \left (\nabla \tPj \cdot ( \ddj - \fn d ) \right )  \\ \tPj- \Phi^j \end{pmatrix}  \right \rangle
 + r^2_k(h)
 \right ) 
\prod_{l=1}^j\frac{1}{\omega^l }   \label{reldisint}
\end{multline}
for all  $ \phi\in\C^\infty_c([0,T)) ${ with } $\phi \geq 0 $, and $ \phi'\leq 0$ on $[0,T]$, where $$ \omega^j := \frac{1+k\mathcal{K}_2(\tu^j)}{1-k\mathcal{K}_1(\tu^j)}\,.$$

\end{corollary}
\begin{proof}
This result follows from applying Lemma~\ref{lem:disgron} to~\eqref{relendis}. 
\end{proof}

\subsection{Convergence to a dissipative solution}\label{convergence1}
The \textit{a priori} estimates achieved in the Theorem~\ref{thm:disex} allow to apply well established standard results to conclude convergence of a subsequence. For $k$, $h \ra 0$ as given above, there exists a subsequence such that
\begin{subequations}
\label{conv}
\begin{align}
\fno {v}\,,\,\fnf v \,,\, \fnu v \stackrel{*}{\rightharpoonup}& \f v  \quad \text{in } L^\infty(0,T;{\mathbb L}^2 ) \cap L^2(0,T;{\mathbb V})\,,\\
\fno d\,,\,\fnf  d\,, \, \fnu d  \stackrel{*}{\rightharpoonup} & \f d \quad \text{in }L^\infty(0,T;{\mathbb L}^{4/3} ) \cap L^\infty(0,T;{\mathbb L}^\infty )\cap W^{1,2}(0,T; {\mathbb L}^{3/2})  \,,\\
[\ov{n}^{\pm}]^k_h\,,\,[n^{\pm}]^k_h \stackrel{*}{\rightharpoonup} &  n^{\pm} \quad \text{in }L^\infty(0,T;{\mathbb L}^2) \cap  L^2(0,T;{\mathbb H}^1)\cap W^{1,2}\bigl(0,T;
  ({\mathbb H}^1)^*\bigr)
     \,,\\
\overline{\Phi} ^k_h \,,\,\Phi ^k_h \stackrel{*}{\rightharpoonup} & \Phi \quad \text{in } L^\infty (0,T; {\mathbb H}^1/_{\R})\label{w:Phi}\,,
\\
\fno q \stackrel{*}{\rightharpoonup} & \f q  \quad \text{in } L^\infty (0,T; ({\mathbb W}^{2,2}\cap \mathbb W^{1,2}_0 )^*)\label{w:q}\,,
\intertext{Due to the Lions-Aubin lemma, we infer the strong convergences}
[\ov{n}^{\pm}]^k_h\,,\,[n^{\pm}]^k_h \ra & n^{\pm} \quad \text{in }   L ^2(0,T; {\mathbb L}^2)  \,,\label{strong:N}\\
\fno d\,,\,\fnf  d\,, \, \fnu d  \ra & \f d \quad \text{in }  L ^p (0,T;{\mathbb L}^p) \quad \bigl(\text {for any } p\in [1,\infty)\bigl)\,,\label{strong:d}
\end{align}
\end{subequations}
where we employed the standard notations
\begin{align*}
\fno u (t) := \tu(jk)\,, \quad  \fnu u (t):= \tu((j-1)k) \,, \quad \fnf u  (t) = \frac{j k-t}{k} \tu((j-1)k) + \frac{t-(j-1)k}{k} \tu( jk )   \,, 
\end{align*}
for $(j-1)k < t \leq  j  k$.
Additionally, we use the abbreviation $ {\fnuo d} := \frac{1}{2}(\fno d + \fnu d )$. 
Using these convergences, going to the limit in~\eqref{dis:Phi} gives immediately the weak formulation of~\eqref{simp:Phi} and thus~\eqref{weak:Phi}. 
The convergence of the Nernst--Planck--Poisson system may be verified as in~\cite{numap} and~\cite{andreas} due to the strong convergence~\eqref{strong:d}.
Passing to the limit in the formulation~\eqref{dis:N}, we find the weak formulation~\eqref{weak:c},

Testing~\eqref{dis:d} by $ \mathcal{I}_h \f h  $ for $\f h\in \C^\infty(\Omega \times (0,T))$ implies 
\begin{align*}
\int_0^T \left ( \t \fnf d ,\mathcal{I}_h \f h \right )_h + \left (
  \fnuo d
   \times  (  \PL[  \nabla \fnu d] \fno v)   ,
\fnuo d    
     \times \mathcal{I}_h \f h \right )_h  + \left (
\fnuo d       
       \times  \fno q ,
\fnuo d         
         \times \mathcal{I}_h \f h \right )_h  \de t =0 \,.
\end{align*}
 In the limit as $k$, $h\ra 0$, we find that~\eqref{weak:d} is fulfilled. 
 Indeed, the only non-obvious point may be the change from mass-lumping to ${\mathbb L}^2$-inner products. With respect to this point, we observe 
 by~\eqref{estdtime} as well as  
 Lemma~\eqref{lem:masslumping} for $p=2$
 that
 \begin{align*}
\left | \int_0^T \left ( \t \fnf d ,
\f h \right )_h -  \left ( \t \fnf d ,
 \f h \right ) \de s\right |  \leq \int_0^T h \| \t \fnf d \| _{{\mathbb L}^2} \|  
  \f h \|_{{\mathbb W}^{1,2}} \, {\rm d}s\leq c h^{(6-d)/6} \|   \f h \|_{{\mathbb W}^{1,2}} \,
 \end{align*}
and
\begin{multline}
 \int_0^T \left (
\fnuo d   
  \times  \fno q ,
\fnuo d  
   \times 
    \f h \right )_h - \left ( 
\fnuo d    
    \times  \fno q , 
\fnuo d     
    \times 
    \f h \right ) \de s
\\  \leq C h  \left \| 
  \fnuo d \times \left ( \fnuo d \times 
   \fno q\right )  \right \|_{L^2(\Omega\times (0,T))}  \| \f h \|_{L^2(0,T;{\mathbb W}^{1,2})} \, .
 \end{multline}
By Lemma~\ref{lem:masslumping}, we find for the remaining  term that
\begin{multline*}
\left|\int_0^T \left (
\fnuo d  
  \times  (  \PL[  \nabla \fnu d] \fno v)   ,
\fnuo d    
    \times 
    \f h \right )_h-\left ( 
\fnuo d      
      \times  (  \PL[  \nabla \fnu d] \fno v)   , 
\fnuo d       
      \times 
      \f h \right )\de s \right|   \\
\leq  C  h \int_0^T\left \|
  \fnuo d \times \left ( \fnuo d \times 
  (  \PL[  \nabla \fnu d] \fno v) \right )  \right \| _{{\mathbb L}^{3/2}} \| \f h \|_{W^{1,3} (\Omega)}\de s  \\
\leq C  h \| \nabla \fnu d \|_{L^\infty(0,T;{\mathbb L}^2)} \| \fno v \|_{L^2(0,T;{\mathbb L}^6}\| \f h \|_{L^2(0,T;{\mathbb W}^{1,3})}  \,.
\end{multline*}

In order to pass to the limit in equation~\eqref{dis:q}, we first establish strong convergence of $\nabla \ov\Phi^k_h$. 
Therefore, we use a standard trick for strongly elliptic problems together with the additional regularity of the limit $\nabla \Phi $ (see~\ref{itemi}). 
\begin{align}
\begin{split}
&\| \nabla (\Phi - \ov\Phi^k_h) \|_{{\mathbb L}^2}^2 + \varepsilon _a \|\f d \cdot  \nabla \Phi -\fno d  \cdot  \nabla \ov\Phi^k_h  \| _{{\mathbb L}^2}^2 \\
&\quad = \left ( \varepsilon( \f d)\nabla \Phi , \nabla \Phi \right ) + \left ( \varepsilon(\fno d) \nabla \ov\Phi^k_hj , \nabla \ov\Phi^k_h \right ) - 2 \left ( \nabla \ov\Phi^k_h , \varepsilon(\fno d ) \mathcal{I}_h \nabla \Phi \right ) \\&\qquad - 2 \left ( \nabla \ov\Phi^k_h , \varepsilon(\fno d ) \left ( I - \mathcal{I}_h \right ) \nabla \Phi \right ) - 2\varepsilon_a \left ( ( \fno d - \f d ) \nabla \Phi , \fno d \cdot \nabla \ov\Phi^k_h \right ) \\
& \quad \leq  \left ( n^+-n^-, \Phi \right ) +  \left ( [n^+]^j - [ n^-]^j ,  \ov\Phi^k_h - 2 \mathcal{I}_h  \Phi \right )_h\\& \qquad +2 \| \nabla \ov\Phi^k_h  \| _{{\mathbb L}^2} \left ( 1+ \varepsilon_a\| \fno d\|_{{\mathbb L}^\infty}^2 \right ) \| (I - \mathcal{I}_h ) \nabla \Phi \|_{{\mathbb L}^2} 
\\
&\qquad + 2 \varepsilon_a \| \f d - \fno d \|_{{\mathbb L}^{2p/(p-2)}} \| \nabla \Phi \|_{{\mathbb L}^p} \| \fno d \|_{{\mathbb L}^\infty} \| \ov\Phi^k_h \|_{{\mathbb L}^2} 
\end{split}\label{StrongPhi}
\end{align}
The strong convergences~\eqref{strong:N} and~\eqref{strong:d} as well as the weak convergence~\eqref{w:Phi} allows us together with the additional regularity of the limit (see~\eqref{addregPhi})
to pass to zero on the right-hand side as $k$, $h\ra 0$. 

Testing~\eqref{dis:q} by $ \mathcal{I}_h 
 \f h
   $ for $\f h\in \C^\infty(\Omega \times (0,T); {\mathbb R}^3)$, we may observe 
\begin{align}
\int_0^T
h^\beta \left ( \Delta_h \fnuo d , \Delta _h \mathcal{I}_h \f h \right ) + 
 \left ( \nabla  \fnuo d ,   \nabla \mathcal{I}_h
   \f h
    \right )  -\varepsilon_a  \left ( \nabla \overline{\Phi}^k_h (\nabla \ov{\Phi}_h^k \cdot \fnu d ) , \mathcal{I}_h 
   \f h
   \right ) \de t =
\int_0^T \left (  \fno q , \mathcal{I}_h 
  \f h
   \right )_h \de t \,.\label{disqeq}
\end{align}
For the last term, we observe by~\eqref{qmass} that
\begin{align*}
 \left (  \fno q , 
   \f h
    \right )_h 
 = {}&  \left (
    \fno q ,    \f h \right ) -\left (  \left (
     \fno q ,    \f h \right ) _h- \left (
       \fno q ,    \f h \right ) \right )
      \\ \leq {}& \left (    \fno q ,    \f h \right )  + C h^{1+d/3} \left ( h^{-d/2} + h^{- d/(12)}  + h^{-(2-\beta/2)}\right ) \| \f h \|_{\mathbb W^{1,6}} 
   \,
\end{align*}
such that its convergence is inferred from~\eqref{w:q}. 

For the second term on the left-hand side of~\eqref{disqeq}, we first estimate the influence of the Interpolation operator by Lemma~\ref{lem:masslumping}
\begin{align*}
\left ( \nabla  \fnuo d ,   \nabla (\mathcal{I}_h-I) 
  \f h
 \right ) 
\leq 
{}& \| \nabla \fnuo d \|_{{\mathbb L}^2} \|  (\mathcal{I}_h-I) 
 \f h
 \|_{{\mathbb W}^{1,2}} 
 \leq{} C h  \| \nabla \fnuo d \|_{{\mathbb L}^2} 
\| \f h \| _{\mathbb W^{2,2}}   
  \,.
\end{align*}

such that this term actually converges to the first term on the right-hand side of~\eqref{weak:q}. 

Using the additional regularity of the limit $\Phi$, we may observe for the second term in~\eqref{disqeq}
\begin{align*}
& \left | \Bigl ( \nabla \Phi (\nabla \Phi \cdot \f d) , 
\f h \Bigr) - \left ( \nabla \overline{\Phi}^k_h (\nabla \ov{\Phi}_h^k \cdot \fnu d ) , \mathcal{I}_h (
 \f h ) \right ) \right | 
 \\
 ={}&
 \Big |\int_0^T  \left ( \nabla \Phi \left ((\nabla \Phi \cdot \f d)- ( \nabla \ov{\Phi}_h^k \cdot \fno d )\right ) , 
 \f h \right ) +  \left ( \nabla \Phi ( \nabla \ov{\Phi}_h^k ( \fno d - \fnu d ) ) ,
  \f h \right ) 
 \\
 &
 + \left ( \nabla \Phi ( \nabla \ov{\Phi}_h^k \cdot \fnu d) , \left ( I -\mathcal{I}_h\right ) ( 
 \f h ) \right ) 
 + \left ( ( \nabla \Phi - \nabla \ov{\Phi}_h^k ) ( \nabla \ov{\Phi}_h^k \cdot \fnu d ) , \mathcal{I}_h ( 
 \f h ) \right ) \de t 
 \Big|
 \\
\leq {}&
  \| \nabla \Phi \|_{L^2(\Omega\times(0,T))} \| \nabla \Phi\cdot \f d - \nabla \ov\Phi_h^k \cdot \fno d \|_{L^2(\Omega \times (0,T))}\| \f h \|_{L^\infty(\Omega \times (0,T))} 
  \\& +
  \| \nabla \Phi\|_{L^p(\Omega\times (0,T))} \| \nabla \ov\Phi_h^k \|_{L^2(\Omega\times(0,T))} \| \fno d - \fnu d \|_{L^{2p/(p-2)}(\Omega\times(0,T))}\| \f h \|_{L^\infty(\Omega \times (0,T))}
  \\&
+ \| \fnu d\|_{L^\infty(\Omega\times(0,T))}
   \| \nabla \Phi\|_{L^p(\Omega\times (0,T))} \| \nabla \ov\Phi_h^k \|_{L^2(\Omega\times(0,T))} \left \|  \left ( I -\mathcal{I}_h\right ) ( 
   \f h )\right \| _{L^{2p/(p-2)}(\Omega\times(0,T))} 
  \\& 
+ 
 \| \nabla \Phi - \nabla \ov \Phi_h^k \|_{L^2(\Omega\times(0,T))} \| \nabla \ov\Phi_h^k \|_{L^2(\Omega\times(0,T))} \| \fnu d \|_{L^\infty(\Omega\times(0,T))}\| \f h \|_{L^\infty(\Omega \times (0,T))} \,.
\end{align*}
The right-hand side vanishes as $k$, $h\ra 0$ due to the strong convergences~\eqref{strong:d} and \eqref{StrongPhi}.
For the first term in~\eqref{disqeq}, we observe that
\begin{align*}
h^\beta \left ( \Delta_h \fnuo d , \Delta_h \mathcal{I}_h \f h \right ) \leq h^{\beta /2} \left ( h^{\beta/2} \| \Delta_h \fnuo d \|_h \right ) \| \f h \|_{\mathbb W^{2,2}} \ra 0 \quad \text{as } h \ra 0\,.
\end{align*}
We may conclude 
  that~\eqref{disqeq} converges to the limit equation~\eqref{weak:q} as $k$, $h\ra 0$. 

It remains to pass to the limit in the relative energy inequality. 
Therefore, we define the linear and constant interpolates also for the test function $\tu\in \mathcal{C}([0,T]; \mathbb{Y})$ and $ \tu \in \mathcal{C}([0,T])$, \textit{i.e.},
\begin{align*}
\ov{\tu} (t) := \tu(jk)\,, \quad  \underline{\tu}(t):= \tu \bigl((j-1)k\bigr) \,, \quad \hat{\tu} (t) = \frac{(j+1)k-t}{k} \tu(jk) + \frac{t-jk}{k} \tu((j+1)k )   \,, 
\end{align*}
for $(j-1)k < t \leq  j k$.
The inequality~\eqref{reldisint} may be interpreted as
\begin{align*}
- \int_0^T &\t \hat{\phi} \left ( \mathcal{R}( \fno{\f u}|\ov{\tu}^k ) + \frac{h^\alpha}{2}\| \nabla \fno v 
\|_{{\mathbb L}^2}^2 + \frac{h^\beta}{2}\| \Delta_h \fnuo d
\|_{{\mathbb L}^2}^2 +r^1_k(h) \right ) \zeta_k(t)   \de t
\\+& \int_0^T \ov\phi \left ( 1 + \gamma_k(t)  \right ) \mathcal{W}_k ( \fno{u} | \ov{\tu}^k ) \zeta_k (t) \de s 
\\
\leq {}&\phi(0) \left ( \mathcal{R}( \f u^0_h | \tu(0)) + \frac{h^\alpha}{2} \| \nabla \f v^0_h
 \|_{{\mathbb L}^2}^2+ \frac{h^\beta}{2}\| \Delta_h \f d^0
\|_{{\mathbb L}^2}^2 \right )  \\
&+ \int_0 ^T\ov\phi  \left ( 1 + \gamma_k(t)  \right )\left \langle \mathcal{A}_d ( \fno{u} ), \begin{pmatrix}
\ov{\vv}^k - \fno{u} \\ \ov{\tq}^j - \fno{q} + \nabla \ov{\tP}^k ( \nabla \ov{\tP}^k \cdot ( \ov{\dd }^k - \fno{d})) \\\ov{\tP}^k - \ov{\Phi}_h^k
\end{pmatrix}\right \rangle  \zeta_k(t)  \de t 
\\&+ \int_0 ^T\ov\phi  \left ( 1 + \gamma_k(t)  \right )\frac{1}{2}\left [ \| [\ov{n}^{+}]^k_h - [\ov{\tilde n}^+]^k   \|_{{\mathbb L}^2}^2+ \|  [\ov{n}^{-}]^k_h - [\ov{\tilde n}^-]^k  \|_{{\mathbb L}^2}^2 + r^2_k(h) \right ]  \zeta_k(t)  \de t 
\,
\end{align*}
 for all  $ \phi\in\C^\infty_c([0,T)) ${ with } $\phi \geq 0 $, and $ \phi'\leq 0$ on $[0,T]$, where we defined  
\begin{align*}
\gamma_k(t) ={}& k \frac{\mathcal{K}_1 (\ov{\tu}(t) , \underline{\tu}(t))}{1- k \mathcal{K}_1 (\ov{\tu}(t) , \underline{\tu}(t)) }  \intertext{and}  
\zeta_k(t) = {}&
\prod_{l=1}^{t_k}  \frac{1}{ 1 + k( 1+ \gamma_k(lk)) \left ( \mathcal{K}_1 (\ov{\tu}(lk) , \underline{\tu}(lk))  + \mathcal{K}_2 (\ov{\tu}(lk) , \underline{\tu}(lk) )\right ) }
 \,,
\end{align*}
where $t_k:= jk\,, \text{ for } (j-1)k < t \leq  j  k $ and the error terms $r^1_k$ and $r^2_k$ are interpreted accordingly. 
For the above terms, we observe  that
\begin{align*}
\gamma_k(t) \ra 0 \quad 
\text{and}\quad  \zeta_k(t) \ra e^{- \int_0^t \mathcal{K}(\tu) \, {\rm d}s} \quad \text{as } k \ra 0 \,,
\end{align*}
where we used $ \mathcal{K}_1(\tu,\tu)+ \mathcal{K}_2(\tu,\tu)=  \mathcal{K}(\tu)$. 
Note that the regularizing terms may be estimated form below by zero on the left hand-side of the above inequality. 
The regularizing terms of the initial values vanish in the limit $h\ra0$ due to the additional regularity of the initial values (compare to the assumptions of Theorem~\eqref{thm:cont}).

Passing to the limit with $h$ and $k$, we find the inequality
\begin{multline*}
- \int_0^T \t \phi \mathcal{R}(\f u | \tu) e^{- \int_0^t \mathcal{K}(\tu) \de \tau} \de s + \int_0^T \phi \mathcal{W}(\f u| \tu) e^{- \int_0^t \mathcal{K}(\tu) \de \tau} \de t \\
\leq 
\phi(0) \mathcal{R}(\f u_0 | \tu(0)) + \int_0^T \phi \left \langle \mathcal{A}(\tu ) , \begin{pmatrix}
 \vv - \f v \\ \tq - \f q - \nabla \tP \nabla\tP \cdot ( \dd - \f d) \\ \tP-\Phi 
\end{pmatrix}\right \rangle e^{- \int_0^t \mathcal{K}(\tu) \de \tau} \de t
\\ + \int_0^T \phi \frac{1}{2}\left [ \| n^+-\tn^+ \|_{{\mathbb L}^2}^2+\| n^--\tn^- \|_{{\mathbb L}^2}^2  \right ] e^{- \int_0^t \mathcal{K}(\tu) \de \tau} \de t
 \,.
\end{multline*}
for all  $ \phi\in\C^\infty_c([0,T)) ${ with } $\phi \geq 0 $, and $ \phi'\leq 0$ on $[0,T]$. 
A variation of the fundamental lemma of variational calculus (see~\cite[Lemma~2.2]{maxdiss}) and multiplying by $ e^{ \int_0^t \mathcal{K}(\tu) \de \tau}$ implies~\eqref{relencont}.

\section{Computational studies\label{sec:comp}}

We set $\Omega=(-0.5,0.5)^d$, $d=2,3$ and consider a slight modification of the numerical scheme (\ref{dis}):
{\small \begin{subequations}
\begin{align}
\begin{split}
\left ( d_t \f v ^j, \f a\right  ) + {\nu}\left ( \nabla \f v^j , \nabla \f a \right ) 
+ \left ( ( \f v^{j-1} \cdot \nabla ) \f v^j , \f a \right ) + \frac{1}{2} \left ( \di \f v ^{j-1} \f v^j , \f a \right )
\qquad\qquad\qquad\qquad \qquad\quad
\\
+ \lambda_{npp} \left ( ([N^+]^j-[N^-]^j )\nabla \Phi^j , \f a \right ) 
+ \nu_{el} \left ( (\PL[\nabla \f d ^{j-1}])^T  (\f d ^{j-1/2} \times (\f d ^{j-1/2}\times \f q ^{j})) , \f a \right )_h ={}& 0\,,
\end{split} \label{dis:v_s5}
\\
  A\left( \nabla \fh d, \nabla \f b \right ) 
- {\mu_{\Phi}}  \left ( \varepsilon_a \nabla \Phi^{j} ( \f d^{j-1} \cdot \nabla \Phi^j), \f b \right )
-\left (  \f q^j ,\f b \right ) _h  ={}&0\,,\label{dis:q_s5}
\\ 
\begin{split}
\left ( d_t \f d^j , \f c \right ) _h
+ \nu_{el} \left ( \f d ^{j-1/2}\times( \PL [\nabla \f d ^{j-1}] \fn v  ) ,  \f d ^{j-1/2} \times \f c \right )_h 
\qquad \qquad
\\
+ \left (\f d ^{j-1/2}\times \f q^j , \f d ^{j-1/2}\times  \f c \right )_h ={}&0\,, 
\end{split}\label{dis:d_s5}
\\
\begin{split}
\left (d_t [ N ^{\pm}]^j , e^{\pm} \right )_h  + \mu_{\Phi} \left ( \varepsilon(\f d^{j})\nabla [ N ^{\pm}]^j , \nabla e ^{\pm} \right )
\qquad\qquad\qquad\qquad\qquad
\\ 
\pm \left ([N^{\pm}]^j \varepsilon(\f d^{j} ) \nabla \Phi^{j} , \nabla e^{\pm} \right )  - \lambda_{npp} \left ( \f v^j [ N^{\pm}]^j, \nabla e^{\pm} \right ) ={}&0\,,
\end{split}
\label{dis:N_s5}
\\
\mu_{\Phi} \left ( \varepsilon(\f d ^{j})\nabla \Phi ^j , \nabla  g \right )
-\left ( [ N ^+]^j-  [N^-]^j, g \right)_h = {}& 0  \, , \label{dis:Phi_s5}
\end{align}
\end{subequations}}
where $\varepsilon(\f d ) = \varepsilon_\perp I + \varepsilon_a \f d \otimes \f d$;
we introduced additional constants $\varepsilon_\perp$, $A$, $\lambda_{npp}$, $\mu_{\Phi}$, $\nu_{el}$
in order to control the strength of interactions between the different physical variables in the system.
If not mentioned otherwise, we set $\f v_0 = \f{0}$, $\varepsilon_\perp=0.1$, $\varepsilon_a=10$,
$A=0.01$, $\mu_{\Phi}=0.25$, $\nu_{el}=1$, {$\nu =1$}.
In \eqref{dis:v_s5} we use  homogeneous Dirichlet boundary conditions for the velocity,
and in (\ref{dis:q_s5}), \eqref{dis:N_s5}-\eqref{dis:Phi_s5}
we employ homogeneous Neumann-type boundary conditions;
{\em i.e.}, we use the same boundary conditions as given in System~\eqref{simp}, except for homogeneous Neumann boundary conditions for the director in~\eqref{dis:d_s5} and~\eqref{dis:q_s5}.

In (\ref{dis:v_s5}), we neglect the stabilization terms $h^{\alpha} \left ( \nabla d_t \f v^j , \nabla \f a \right )$ from (\ref{dis:v}) and $h^{\beta} (\Delta_h \f d^{j-1/2}, \Delta_h \f b)$ from (\ref{dis:q}), which was not required to preserve the
discrete maximum principle for $ [n^\pm]^j$
in the presented experiments --- as opposed to part {\em d)} in the proof of
Lemma \ref{exist-lemma2}.
In addition, we note that a suitable choice of the nonlinear solver guarantees that
the discrete constraint $|\f d^j| = 1$, $j=0,\dots,J$ is always satisfied at the nodes of the finite element
mesh up to machine accuracy independently of $\tau$, $h$; cf.~\cite{sllg} and below.

The velocity in the equation \eqref{dis:v_s5} is approximated using the $P2$-$P1$ Taylor-Hood element, see \textit{e.g.}~\cite{mhd},
the remaining unknowns are approximated using standard continuous piecewise linear finite elements.
To solve the nonlinear algebraic system related to the coupled equations 
\eqref{dis:v_s5}-\eqref{dis:Phi_s5}, we use a simple 
fixed-point iterative scheme analogous to \cite[Algorithm A$_1$]{andreas} (cf.~also \cite[Algorithm~A]{mhd}).
The stopping criterion for the iterative solvers was the $\ell_\infty$-norm of the subsequent iterates
with respective tolerance $tol=10^{-9}$
in the fixed algorithm,
and tolerance $tol=10^{-14}$ in the arising linear and nonlinear systems in each fixed-point iteration 
to eliminate a possible effect of the algebraic solvers on the numerical approximation; 
we note that more efficient implementations of the algorithm are possible.
In each iteration of the fixed-point algorithm the equations \eqref{dis:v_s5}-\eqref{dis:Phi_s5} are linearized, cf.~\cite{mhd, andreas},
in a way that the resulting respective equations are decoupled and can be solved separately. 
All {resulting} equations, except for \eqref{dis:d_s5},
are linear; the nonlinear algebraic system that corresponds to (the linearized version of)
\eqref{dis:d_s5}-\eqref{dis:q_s5}
is solved using a Newton-multigrid algorithm; cf.~\cite{sllg}.
Linear systems arising from \eqref{dis:v_s5} in $d=3$ are solved using the Vanka-multigrid method, 
cf.~\cite{mhd}.

A simple modification of the proof of Theorem~\ref{thm:disex}~ii) implies that the above numerical scheme satisfies the following discrete energy law
\begin{multline}
E (\f v^J,\f d^J , \Phi^J)   
+\frac{ k^2}{2} \sum_{j=1}^J \left (\| d_t \f v^j \|^2_{{\mathbb L}^2}
+ \mu_\Phi \bigl(\varepsilon(\f d^{j-1})\nabla d_t \Phi^j , \nabla d_t \Phi^j \bigr) + \mu_\Phi\varepsilon_a \| \nabla \Phi^j \cdot d_t \f d^j \|^2_{{\mathbb L}^2}   \right )
\\
+ k \sum_{j=1}^J \Bigg[ \nu \| \nabla \f v^j \|^2_{{\mathbb L}^2} + \| \f d ^{j-1/2} \times \f q^j \|^2_h + \Bigl ( ([N^+]^j+[N^-]^j) \nabla \Phi^j , \varepsilon(\f d^{j} ) \nabla \Phi ^j  \Bigr ) 
+ \|[N^+]^j-[N^-]^j\|^2_h  \Bigg] 
\\
= E (\f v^0,\f d^0 , \Phi^0)
\,, \label{disenin_s5}
\end{multline}
where 
$E(\f v^j,\f d^j , \Phi^j) : = \frac{1}{2}\| \f v^j \|^2_{{\mathbb L}^2} + \frac{A}{2}\| \nabla \fn d  \|^2_{{\mathbb L}^2}
 + \frac{\mu_{\Phi}}{2} \bigl  ( \varepsilon(\f d ^j ) \nabla \Phi^j , \nabla \Phi ^j \bigr )$.
In the experiments below (except for the ones with applied field) we verified the decrease of the physically relevant component in the above energy law,
{\em i.e.}, we neglected the numerical damping term scaled by $\frac{k^2}{2}$ in \eqref{disenin_s5}.

\subsection{Ericksen--Leslie interactions}

In the next two experiments we illustrate the Ericksen--Leslie interactions in the model.
We set $n^{\pm}_0=0$, $\lambda_{npp}=\varepsilon_a=\varepsilon_\perp=0$ which implies that $\Phi^{j}\equiv 0$, $j=1,\dots, J$, and the interactions in the system (\ref{dis:v_s5})-(\ref{dis:Phi_s5})
reduce to the coupling between (\ref{dis:v_s5})-(\ref{dis:d_s5}).

\subsubsection{Defect driven flow}
We choose the parameters analogically to \cite[Example 5.2]{prohl}:
we choose $\f v_0 = 0$,  and $A=1$, $\nu=1$, $\nu_{el}=0.25$, 
and the discretization parameters were $k=0.0005$, $h=2^{-4}$, $T=0.1$. 
To construct the initial condition for the director we set $\f{\hat{d}}(x, y) = (4x^2 + 4y^2 - 0.25, 2y, 0)^T$ and define
$$
\displaystyle
\f{d}_0 = \left\{
\begin{array}{llll}
\displaystyle \f{\hat{d}}/|\f{\hat{d}}| \qquad & \mathrm{if}\quad |\f{\hat{d}}| > 0\,,
\\
(0,0,1)^T \qquad & \mathrm{if}\quad |\f{\hat{d}}| = 0\,.
\\
\end{array}
\right.
$$
We observe that the initial condition above contains two defects, see Figure~\ref{fig1_el1} (left).

The computed results are displayed in Figures~\ref{fig1_el1} and \ref{fig2_el1}.
The observed evolution is similar to the results in \cite[Example 5.2]{prohl} for $d=2$:
the velocity drives the defect towards each other and the director field gradually becomes uniform in space.
\begin{figure}[!htp]
\begin{center}
\includegraphics[width=0.245\textwidth]{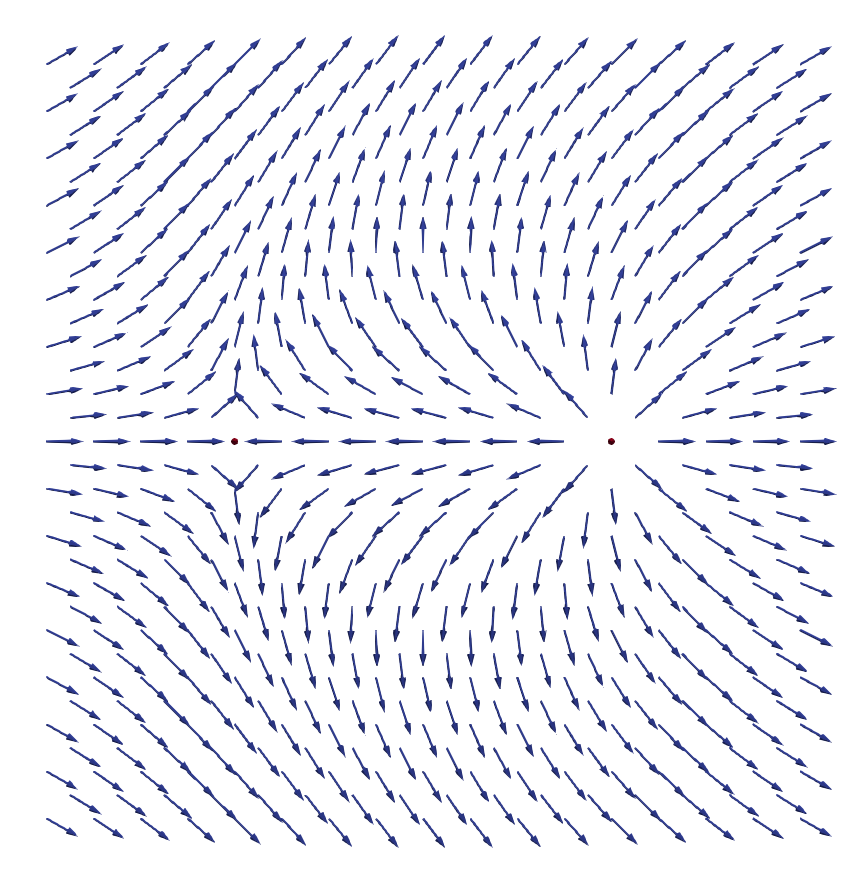}\qquad
\includegraphics[width=0.32\textwidth]{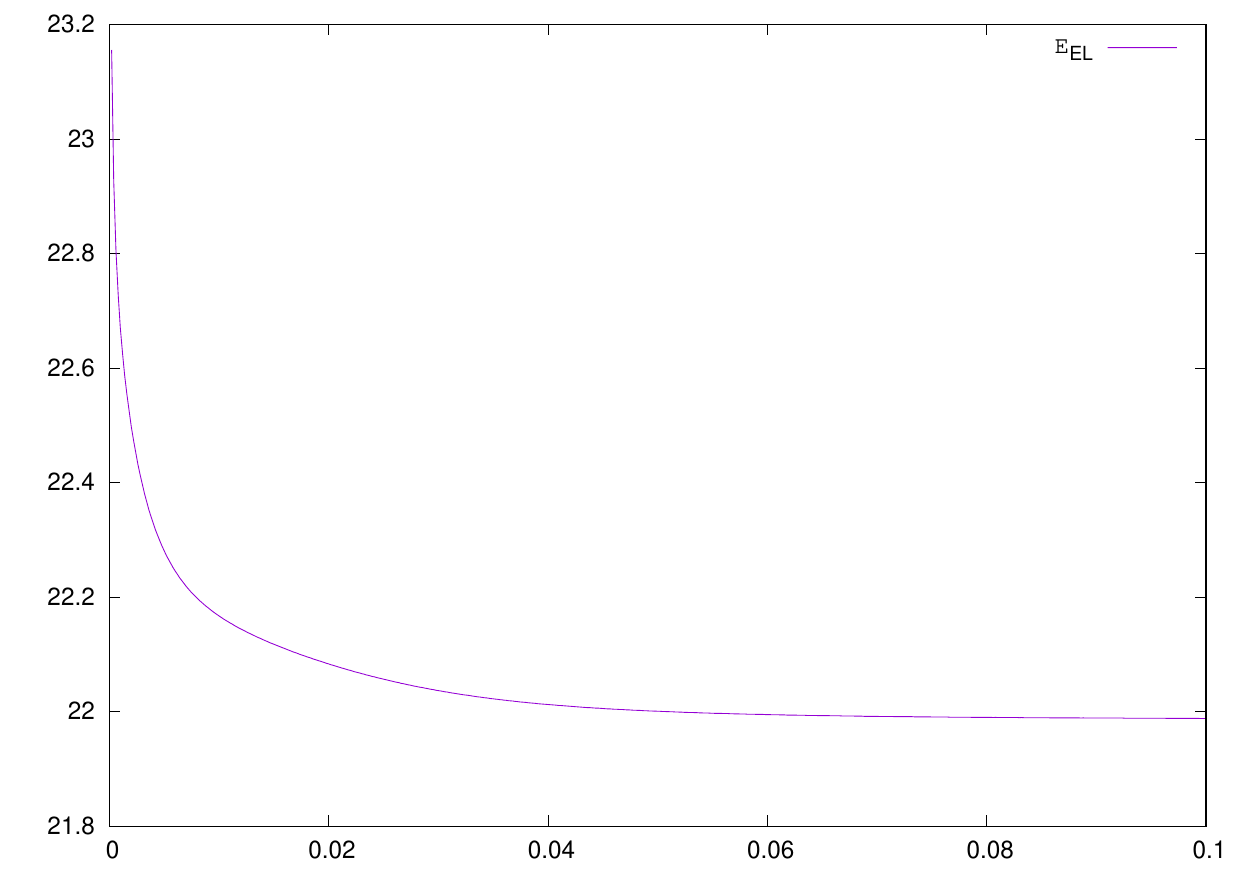}\qquad
\includegraphics[width=0.245\textwidth]{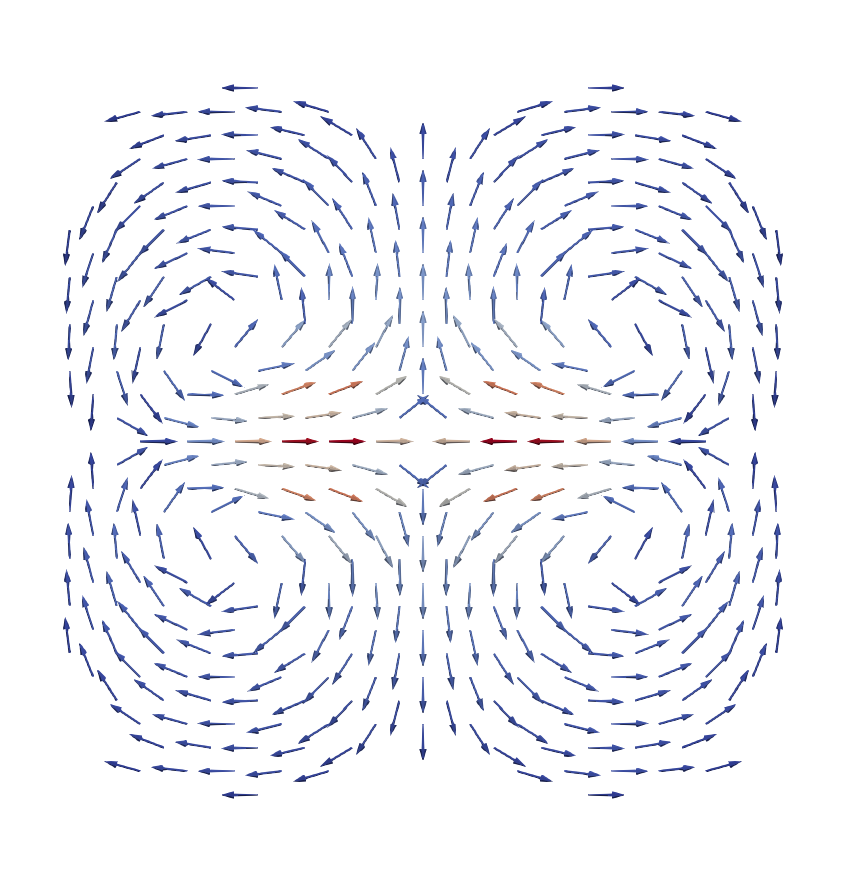}
\end{center}
\caption{From let to right: initial condition $\f{d}_0$ (colored by the $z$-component), evolution of the discrete energy, velocity at $t=0.02$ (colored by the magnitude).}
\label{fig1_el1}
\end{figure}
\begin{figure}[!htp]
\begin{center}
\includegraphics[width=0.245\textwidth]{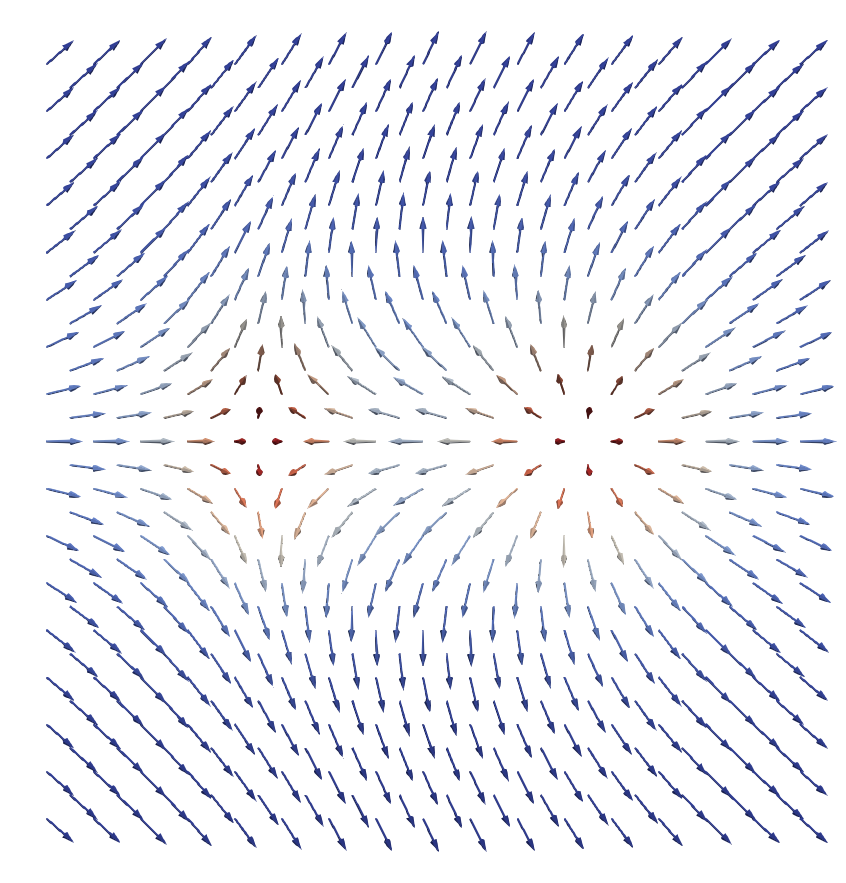}\qquad
\includegraphics[width=0.245\textwidth]{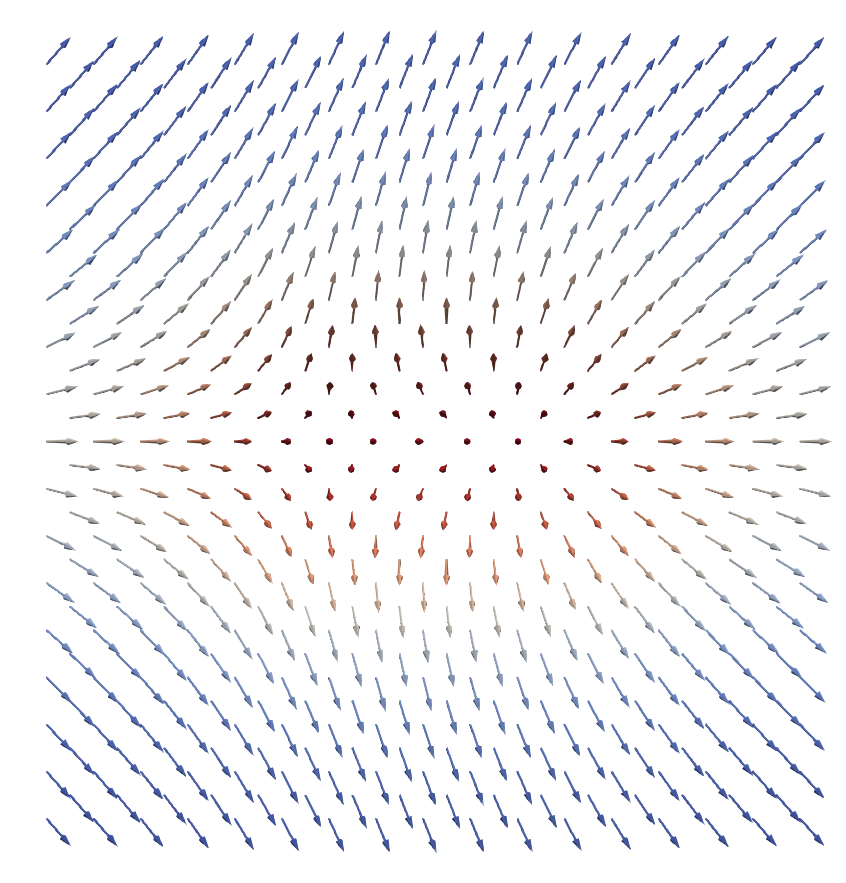}\qquad
\includegraphics[width=0.245\textwidth]{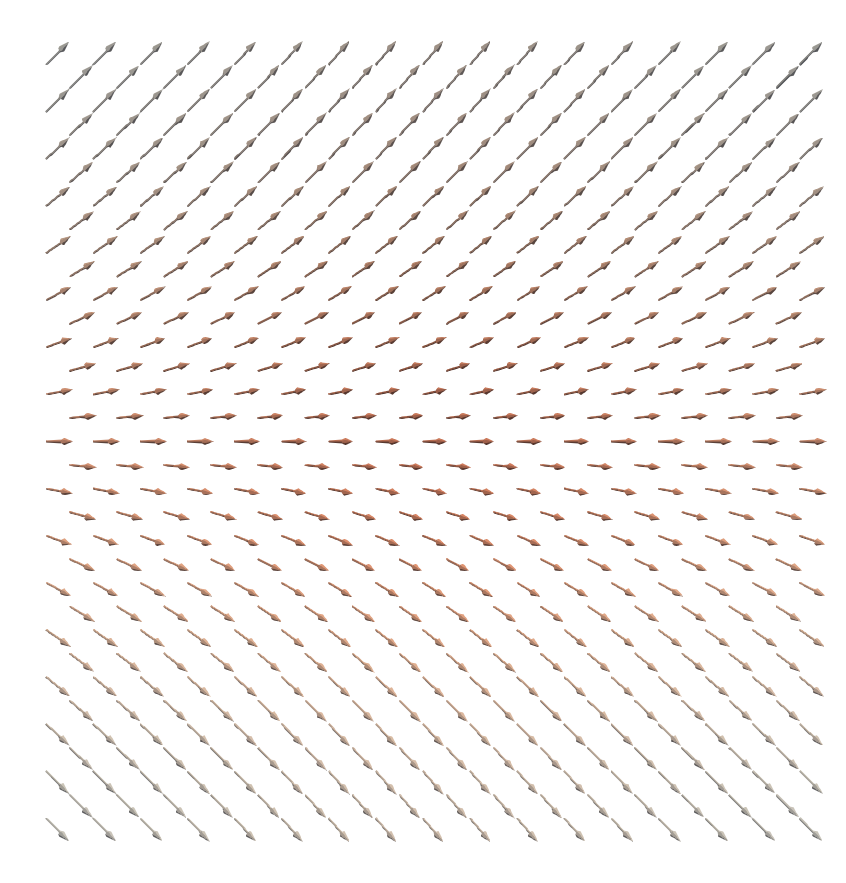}
\end{center}
\caption{Director at time $t=0.02, 0.04, 0.1$, colored by the magnitude of the $z$-component.}
\label{fig2_el1}
\end{figure}

\subsubsection{Velocity driven flow}
Next, we examine the effect of the velocity on the evolution of the director field.
We choose $\f v_0 = 10(-y, x)^T$, $A=0.1$, $\nu=1$, $\nu_{el}=1$, and 
the remaining parameters, as well as the initial condition for the director are the same as in the previous experiment.
In Figure~\ref{fig1_el2} we observe that the defects in the director field rotate
around the center of the domain
due to the advection effect of the velocity field (Figure~\ref{fig1_el2} (right)). 
We note that since the energy decreases over time, the velocity field becomes weaker 
and eventually vanishes.
\begin{figure}[!htp]
\begin{center}
\includegraphics[width=0.245\textwidth]{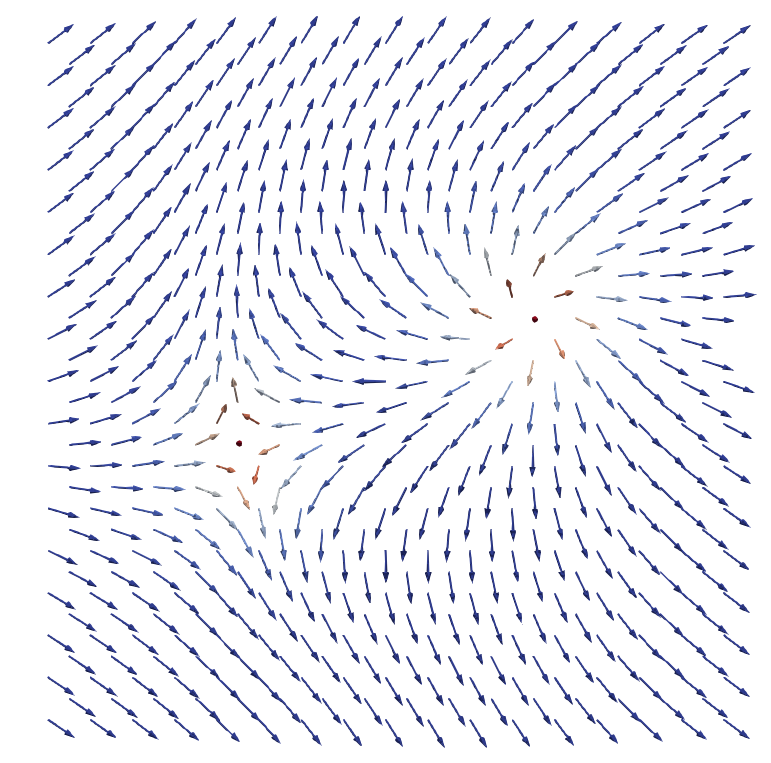}
\includegraphics[width=0.245\textwidth]{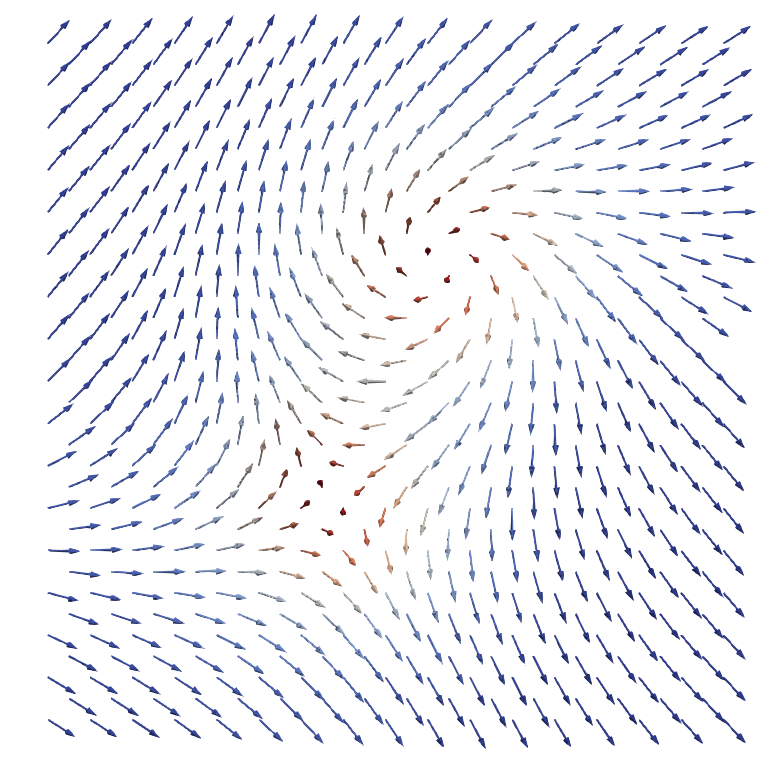}
\includegraphics[width=0.245\textwidth]{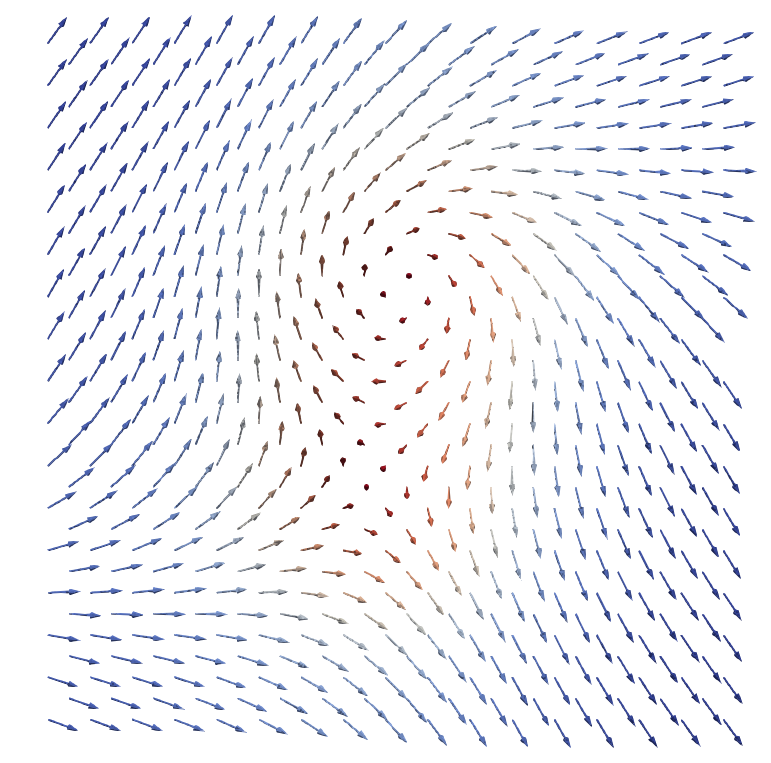}
\includegraphics[width=0.245\textwidth]{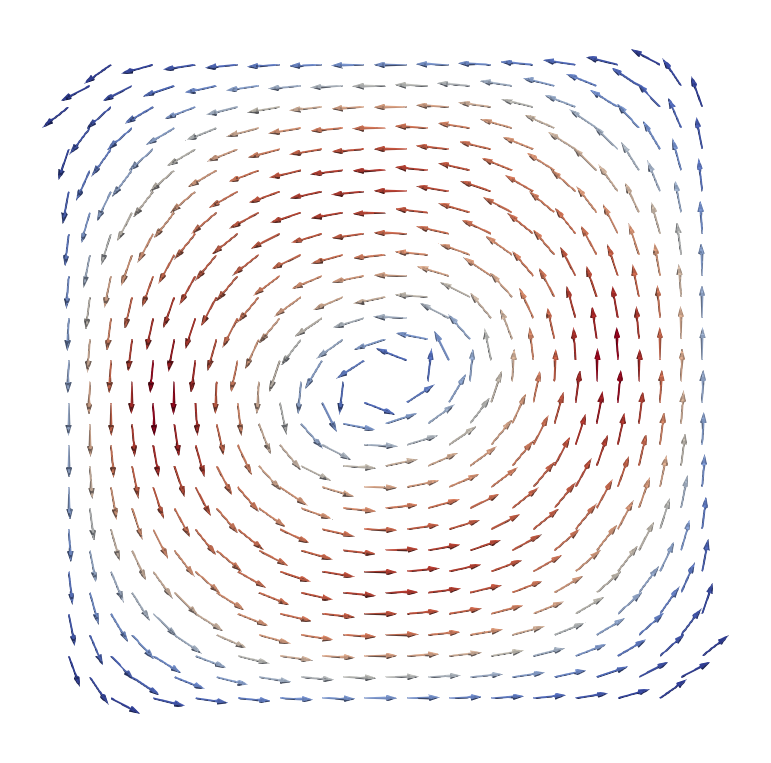}
\end{center}
\caption{Director at time $t=0.05, 0.15, 0.25$ (colored by the magnitude of the $z$-component) and the velocity field at time $t=0.15$ (colored by the magnitude).}
\label{fig1_el2}
\end{figure}

\subsubsection{Effect of the director on the electric field}\label{sec_static}

We demonstrate the anisotropy effect which is created as a consequence of interaction between the electric field
and the director.
We consider $d=3$ and study a stationary situation where 
we only solve (\ref{dis:v_s5}), (\ref{dis:Phi_s5})
with prescribed director field and charge densities that do not evolve over time.
The configuration of the charges is the so-called dipole:
we consider two spherically distributed (stationary) charges with opposite polarity centered at $\f{x}_{0}^\pm=(\pm 0.2,0,0)^T$
given as $n^{\pm}(\f{x}) = \exp(-50|\f{x}-\f{x}_0^\pm|^2)$,
and a constant director field in the $z$-direction  $\f{d} \equiv (0,0,1)^T$.
The remaining parameters were $\nu=1$, $A=0.1$, $\lambda_{npp}=100$,  $\lambda_{el}=0$, $\nu_{el}=1$, $\mu_{\Phi}=0.25$, $\varepsilon_\perp=0.1$, $\varepsilon_a=100$,
and the results were computed with $k=5\times 10^{-4}$, $h=2^{-5}$.

In general the induced (negative) electric field $-\f{E} = \nabla \Phi$ points from the negatively charged region towards
the positively charged one.
Without the director effect the electric field induced by the dipole with $\varepsilon_a=0$ (\textit{i.e.}, no effect of the director)
is radially symmetric along the $x$-axis; see Figure~\ref{fig_stat_dip_phi} (left).
When the director field $\f{d} \equiv (0,0,1)^T$ is included in the system it introduces an anisotropy effect 
in the $z$-direction, \textit{i.e.}, the field is approximately constant in the $z$-direction;
see Figure~\ref{fig_stat_dip_phi} (right). 
For illustration in Figure~\ref{fig_stat_dip_vel} we also display the velocity field induced by the electric field at time $t=0.0005$; 
the velocity is qualitatively similar for both cases. 

\begin{figure}[!htp]
\begin{center}
\includegraphics[width=0.45\textwidth]{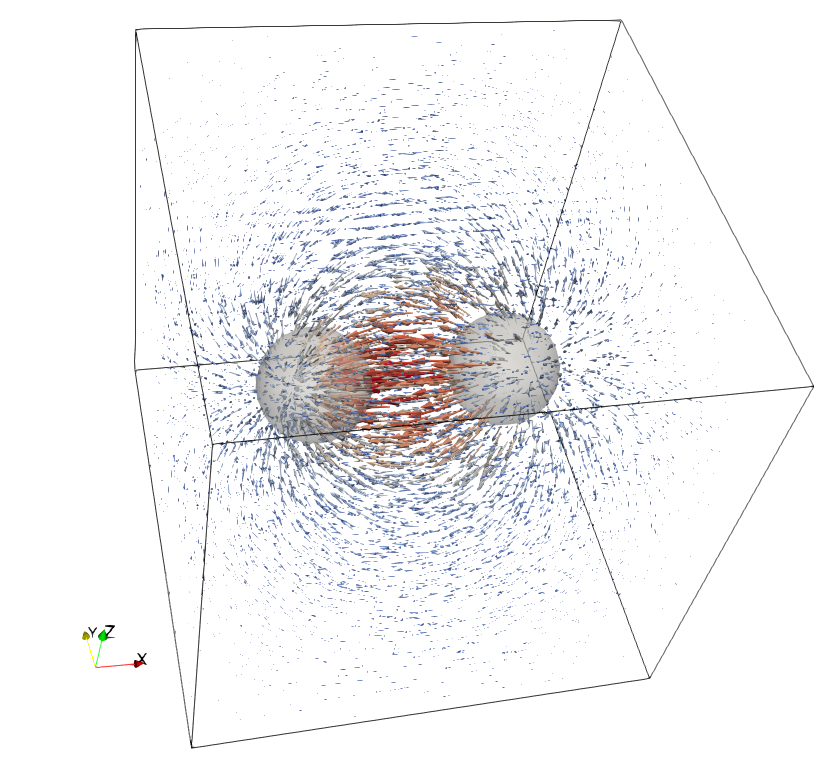}\qquad
\includegraphics[width=0.45\textwidth]{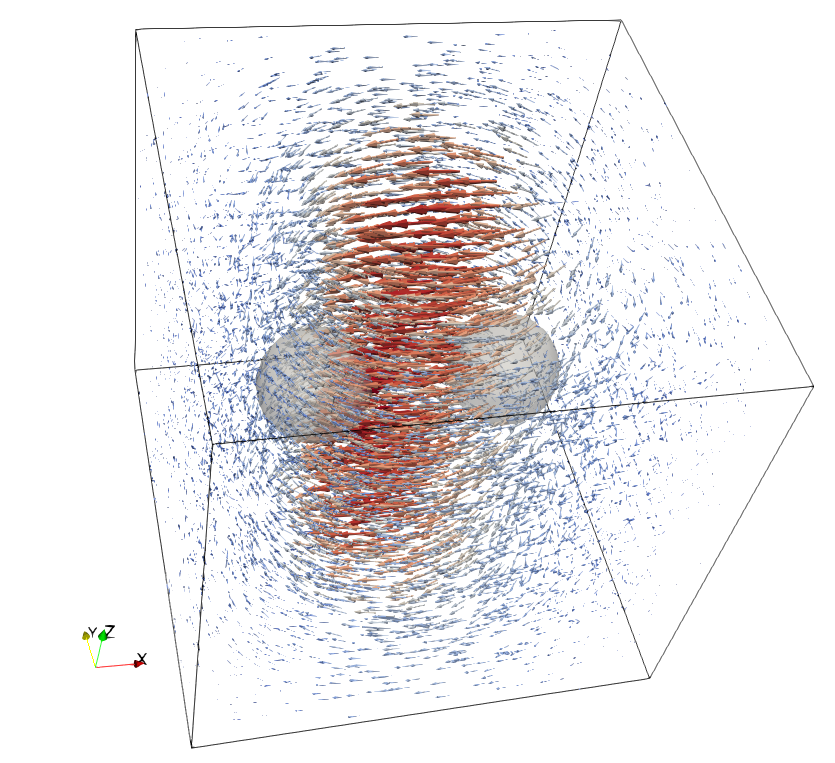}
\end{center}
\caption{Negative electric field $-\f{E} = \nabla \Phi$
and the $\pm 0.5$-level set of $n^{+}-n^{-1}$ for $\f{d} = \f{0}$ (left) and $\f{d} = (0,0,1)^T$ (right).}
\label{fig_stat_dip_phi}
\end{figure}

\begin{figure}[!htp]
\begin{center}
\includegraphics[width=0.45\textwidth]{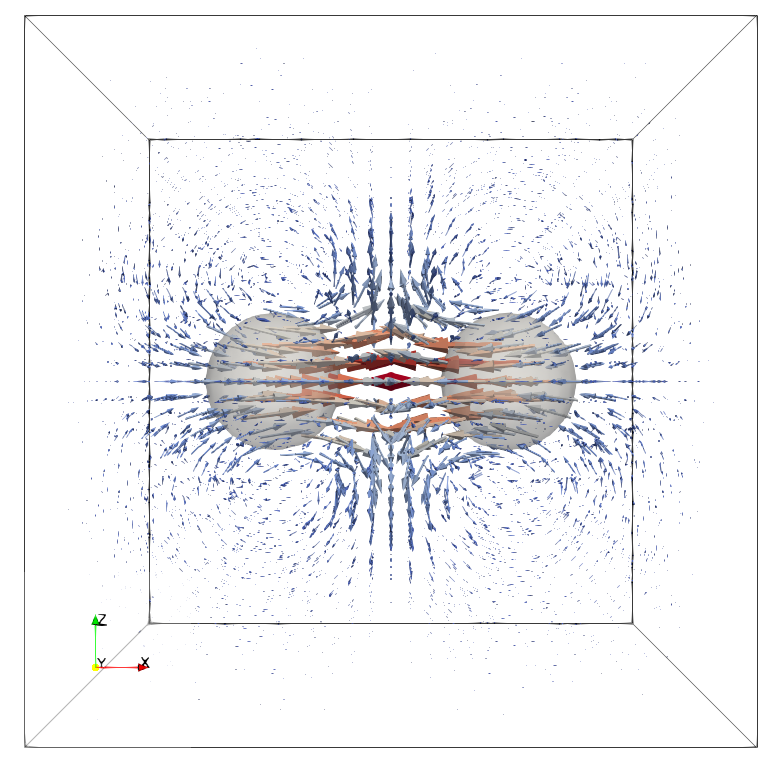}\qquad
\includegraphics[width=0.45\textwidth]{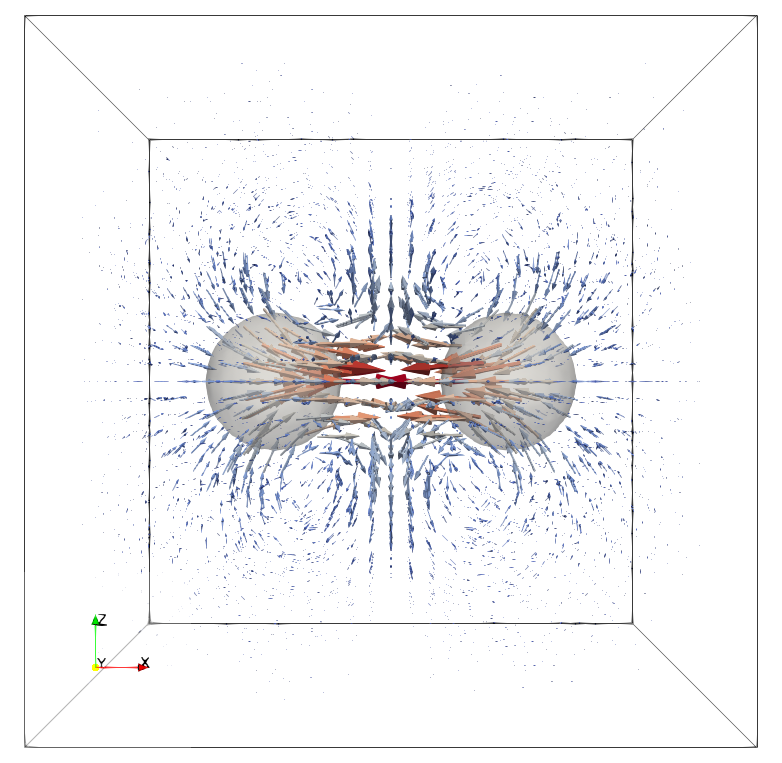}
\end{center}
\caption{Velocity field at $t=5\times 10^{-4}$
and the $\pm 0.5$-level set of $n^{+}-n^{-1}$ for $\f{d} = \f{0}$ (left) and $\f{d} = (0,0,1)^T$ (right).}
\label{fig_stat_dip_vel}
\end{figure}

\subsection{The full system in 3D}

In the subsequent experiments we examine the evolution of the full system (\ref{dis:v_s5})-(\ref{dis:Phi_s5}) in $d=3$
for different configurations of the model parameters. 
The main observations from the presented simulations can be summarized as follows: 
\begin{itemize}
\item 
In the simulations below we want to illustrate physically relevant features which are predominantly due to the effect of the electric field.
Consequently, the parameters are chosen such that the velocity field (which is induced by the interactions of the flow with the director field and the electric field, \textit{i.e.},
the $\lambda_{npp}$, $\nu_{el}$-terms in (\ref{dis:v_s5}))
has a comparably weaker effect on the overall evolution.
Furthermore, except for the last experiments, the director field did not significantly evolve over time.
\item The orientation of the director induces an anisotropy into the system, {\em i.e.},
the charges are transported by the electric field along the director field.
In addition, the orientation of the director determines the direction of the induced electric field, as well as of the velocity field.
In particular, due to the anisotropy effects of the diffusion tensor $\varepsilon(\f{d})$ in \eqref{dis:Phi_s5},
the electric field ($\f{E} = -\nabla \Phi$), which is induced by the difference between the positive and negative charges, 
remains predominantly perpendicular to the director field.
\end{itemize}

\subsubsection{Effect of the director on the diffusion of the charges}

The next experiment is to demonstrate the anisotropy effect due to the orientation of the director.
The initial condition for the director is $\f{d}_0\equiv 2^{-1/2}(0,1,1)^T$,
and the initial charges are taken  as $n^{\pm}_0 = \exp(-25|\f{x}-\f{x}_0^\pm|^2)$ with $\f{x}_{0}^\pm=(\pm 0.2,0,0)^T$.
The remaining parameters were $\nu=1$, $A=0.1$, $\lambda_{npp}=100$, $\nu_{el}=1$, $\mu_{\Phi}=0.125$, $\varepsilon_\perp=0.1$, $\varepsilon_a=100$,
and the results were computed with $k=2.5\times 10^{-4}$, $h=2^{-4}$.

The director field remains approximately constant during the whole evolution,
and the induced velocity field (which exhibits symmetry properties along in plane
perpendicular to $(0,-1,1)^T$; cf.~Figure~\ref{fig_dip3d_phi} (right)) is small.
Consequently, the velocity has a negligible effect, and the evolution 
is driven mainly by the diffusion of the charges and the advective effects of the electric field.

In Figure~\ref{fig_dip3d_phi} we display a typical configuration of the gradient of the electric potential,
the $0.2$-level set of the charges, and the magnitude of the velocity field
along the direction $(0,1,1)^T$ (the direction of the director).
As in Section~\ref{sec_static}, we observe 
anisotropy in the displayed electric field along the $(0,1,1)$-direction,
which is due to the interactions with the director.

\begin{figure}[!htp]
\begin{center}
\includegraphics[width=0.3\textwidth]{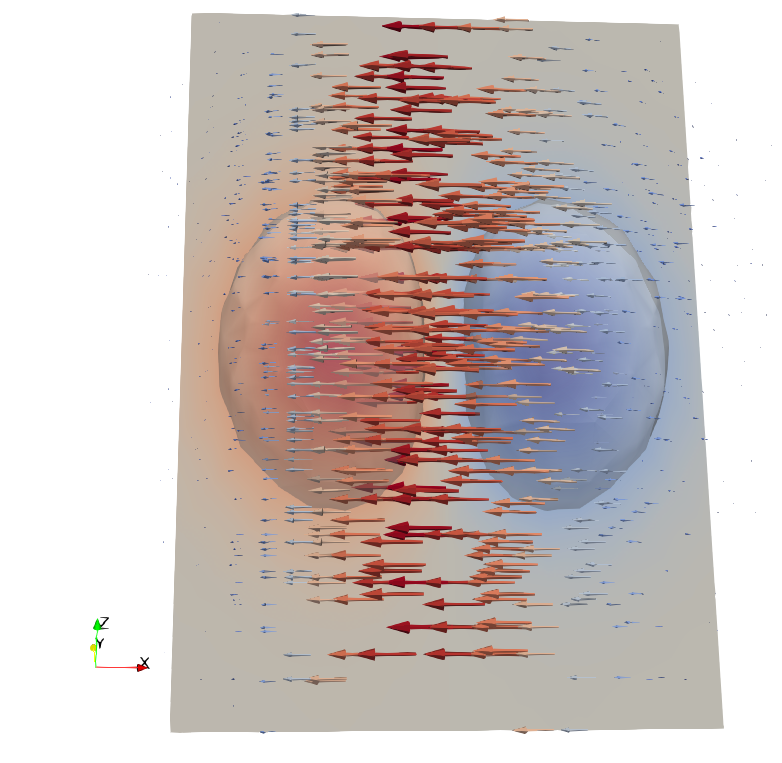}\quad
\includegraphics[width=0.3\textwidth]{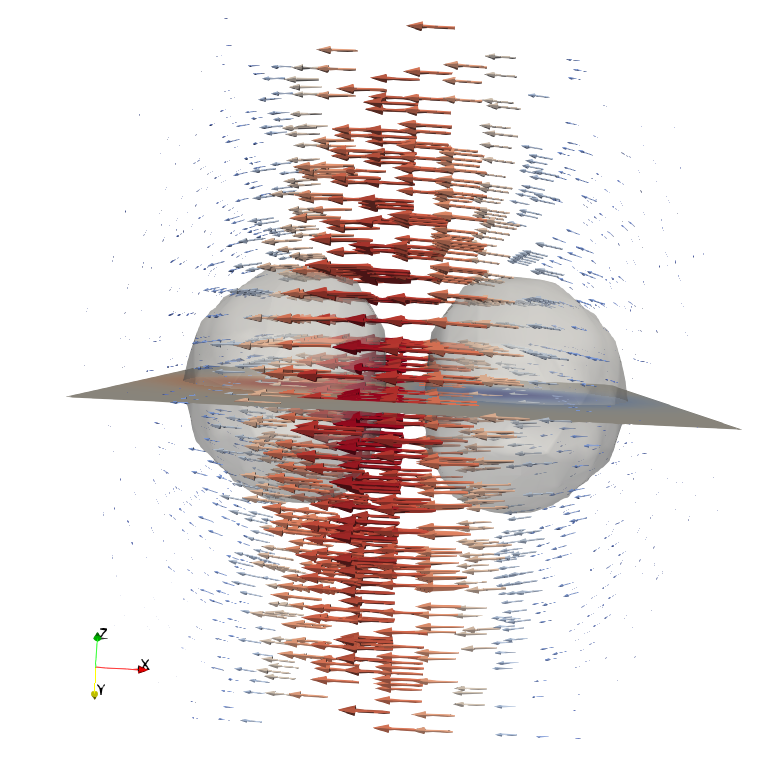}\quad
\includegraphics[width=0.3\textwidth]{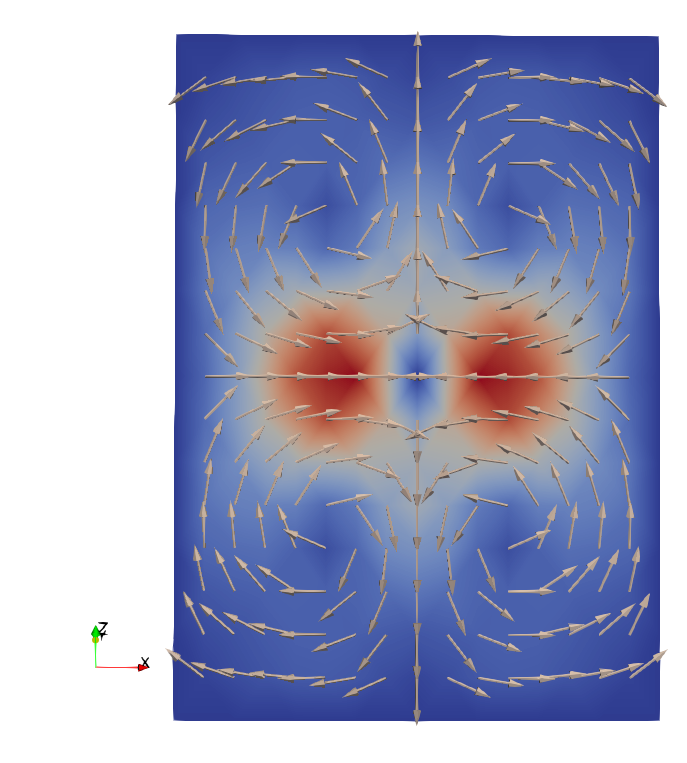}
\end{center}
\caption{Snapshots of the negative electric field from two different angles (left, middle),
and the velocity field in the plane normal to $(0,-1,1)^T$ at time $t=0.015$.}
\label{fig_dip3d_phi}
\end{figure}

In Figure~\ref{fig_dip3d_npm}, we display the evolution of the $0.2$-level set of $n^{\pm}$,
as well as of the value of $n^{+}-n^-$ in the plane normal to $(0,-1,1)^{T}$;
we observe that the charges evolve along the direction of the director, which is
$\approx (0,1,1)^T$.
\begin{figure}[!htp]
\begin{center}
\includegraphics[width=0.3\textwidth]{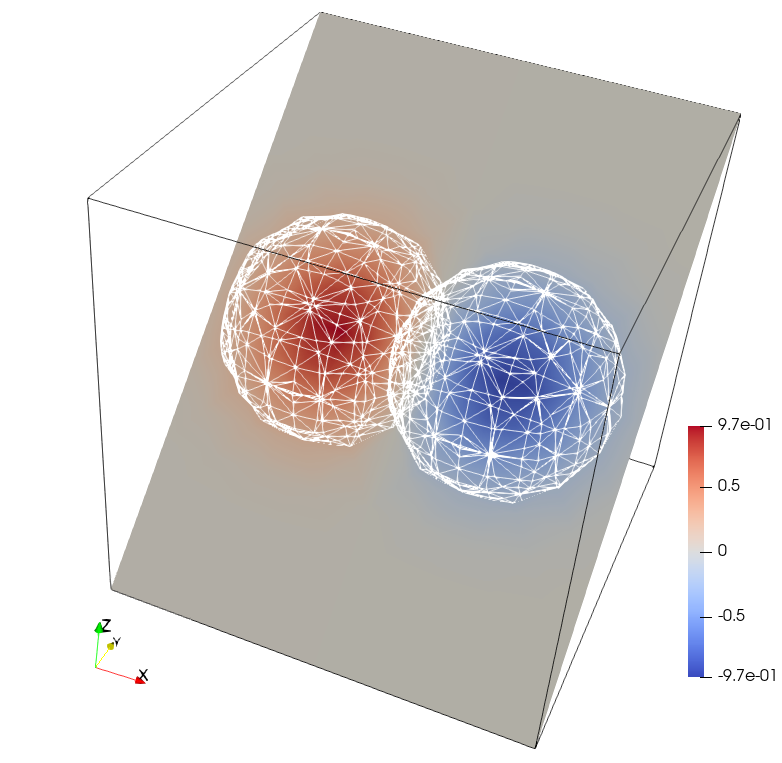}\quad
\includegraphics[width=0.3\textwidth]{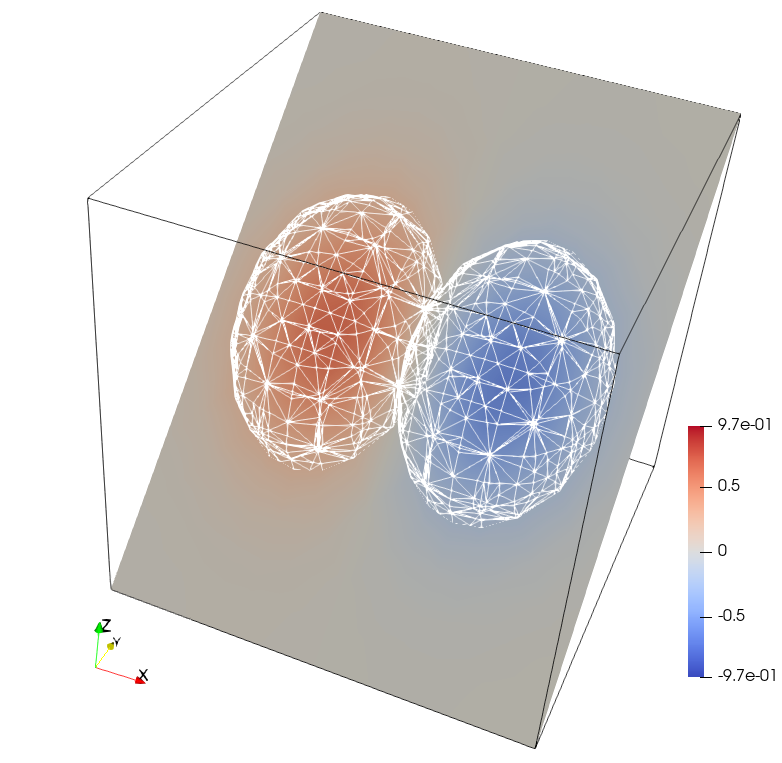}\quad
\includegraphics[width=0.3\textwidth]{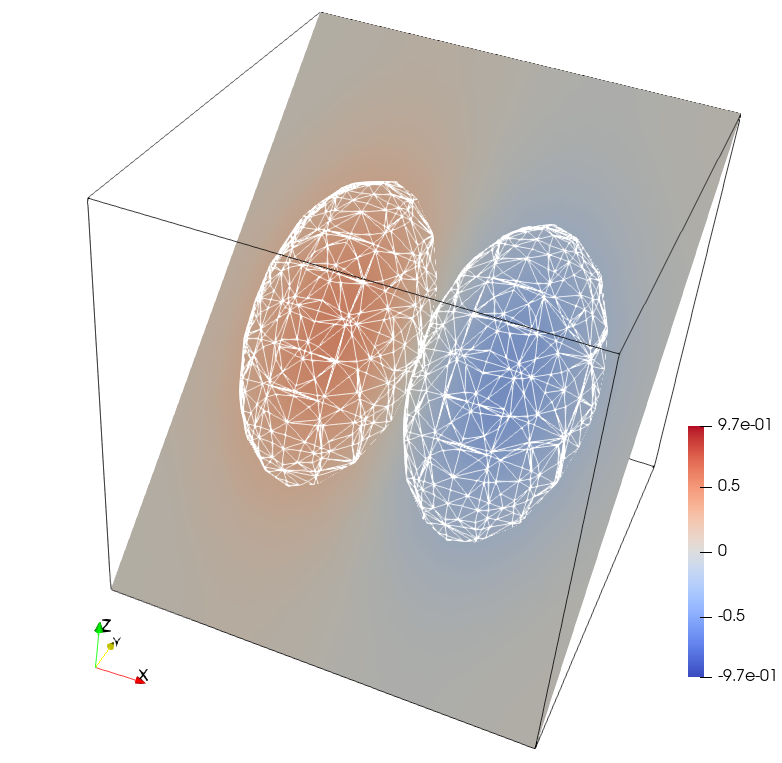}
\end{center}
\caption{Snapshots of the $0.2$-level set of $n^{+}$, $n^{-}$, 
and a cut through $n^{+}-n^{-}$ at $t=0,0.02,0.045$.}
\label{fig_dip3d_npm}
\end{figure}

\subsubsection{Effect of an applied electric field}

Without external influence,  the gradient of the electric potential 
({\em i.e.}, the negative electric field) is generated solely by the difference between the negative and positive charges.
From the previous experiments we deduce that  for $\varepsilon_a \gg \varepsilon_\perp$ the electric field
is induced predominantly in the direction that is perpendicular to the director field,  
and thus only  has a little influence on the director field.
In order to demonstrate the effects of the electric field on the evolution of the director
we apply a uniform external electric field $\f{E}_0=(0.4,0,0)^T$ along the $x$-direction,
\textit{i.e.}, we replace $\nabla \Phi$ by $\nabla \widetilde \Phi = \nabla \Phi - \f{E}_0$ in (\ref{dis:v_s5})-(\ref{dis:N_s5}).
The remaining parameters in the simulation were $\nu=1$, $A=0.1$, $\lambda_{npp}=1000$, $\nu_{el}=1$, $\mu_{\Phi}=1.0$, $\varepsilon_\perp=0.1$, $\varepsilon_a=10$,
and the discretization was performed for $k=10^{-3}$, $h=2^{-4}$.}
The initial distribution of the charges and the initial orientation of 
the director field are chosen to be uniform, {\em i.e.}, $n^{\pm}_0=0.5$ and $\f{d}_0=3^{-1/2}(1,1,1)$.

The applied electric field forces the positive and negative charges to accumulate
according to their polarity in the opposing parts of the spatial domain along the direction of the director field.
Initially the charges accumulate in the opposing corners of the domain along $\f{d}_0$, {\em i.e.}, the $(1,1,1)$-direction; see Figure~\ref{fig_appl3d}.
Due to the effect of the external field, the director 
rotates from its initial orientation towards the direction of the applied field ({\em i.e.}, the direction
parallel to the $x$-axis); see Figure~\ref{fig_appl3d_mh}.
As the system approaches a stationary state,
the charges accumulate along the $x$-direction. 
The induced perpendicular component of the electric field $\nabla \Phi= \nabla \widetilde{\Phi} - \f{E}_0$, and 
the induced velocity at $t=0.03$ 
are displayed in Figure~\ref{fig_appl3d_phi_vel}.

\begin{figure}[!htp]
\begin{center}
\includegraphics[width=0.24\textwidth]{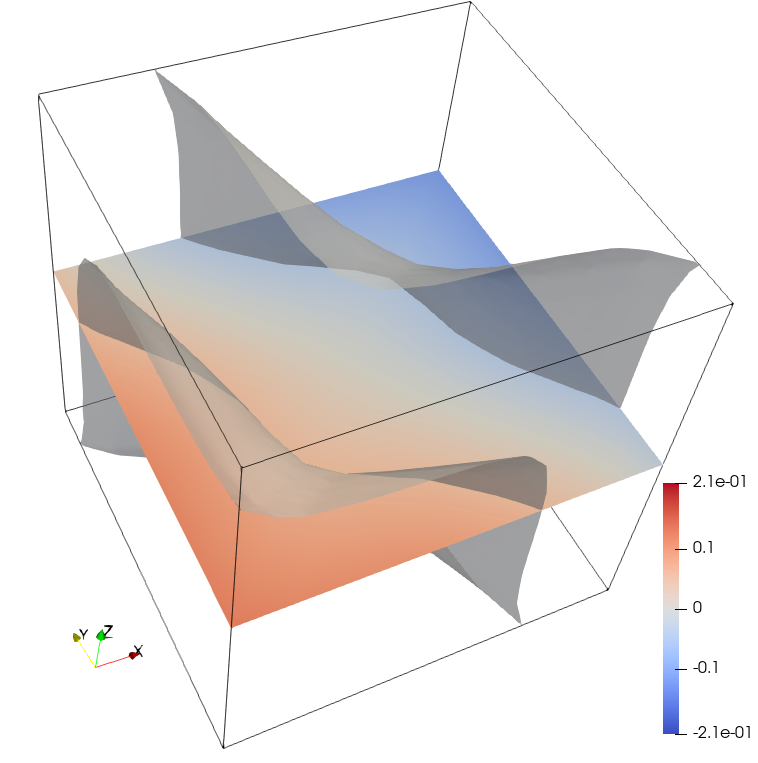}
\includegraphics[width=0.24\textwidth]{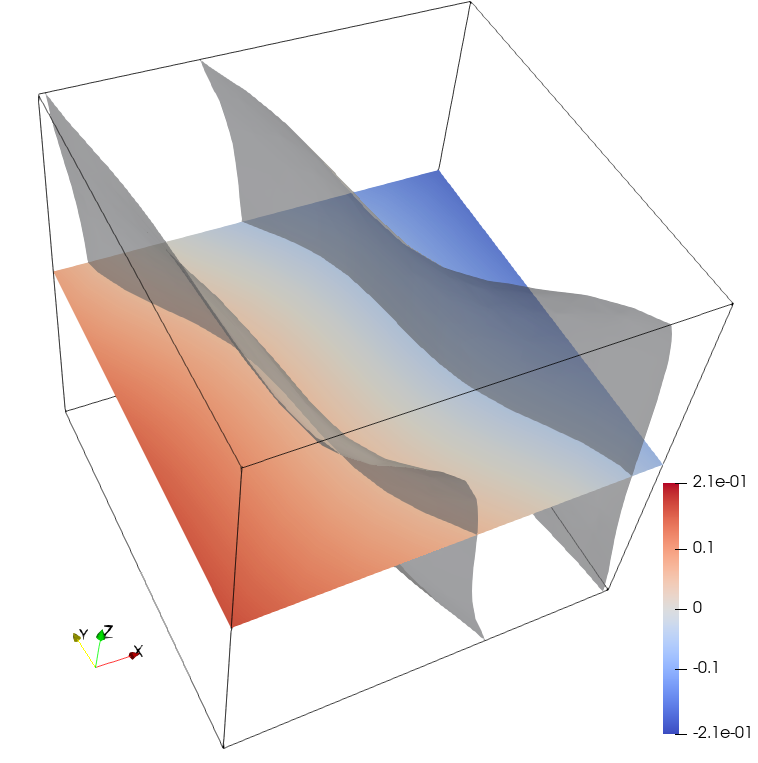}
\includegraphics[width=0.24\textwidth]{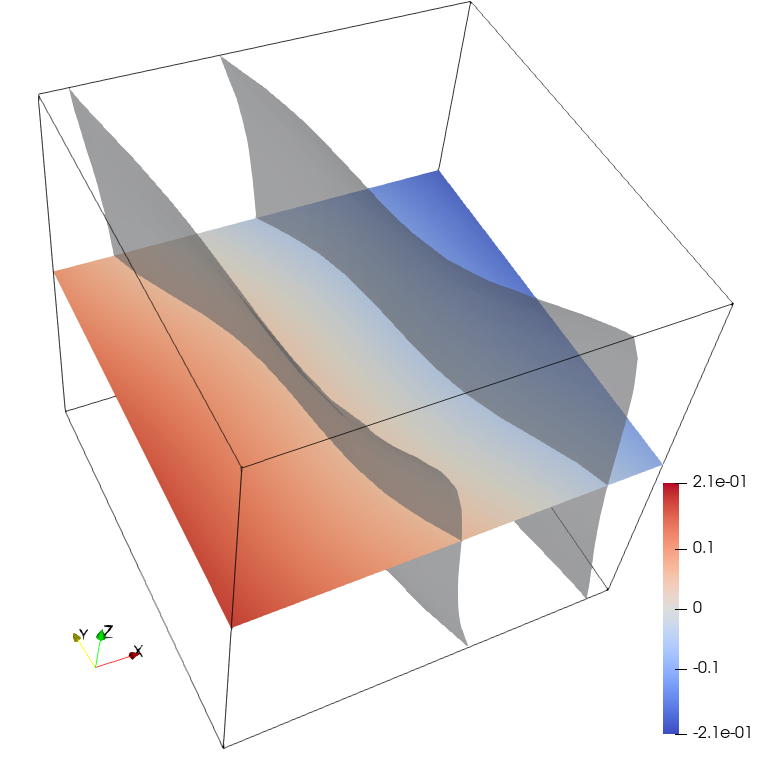}
\includegraphics[width=0.24\textwidth]{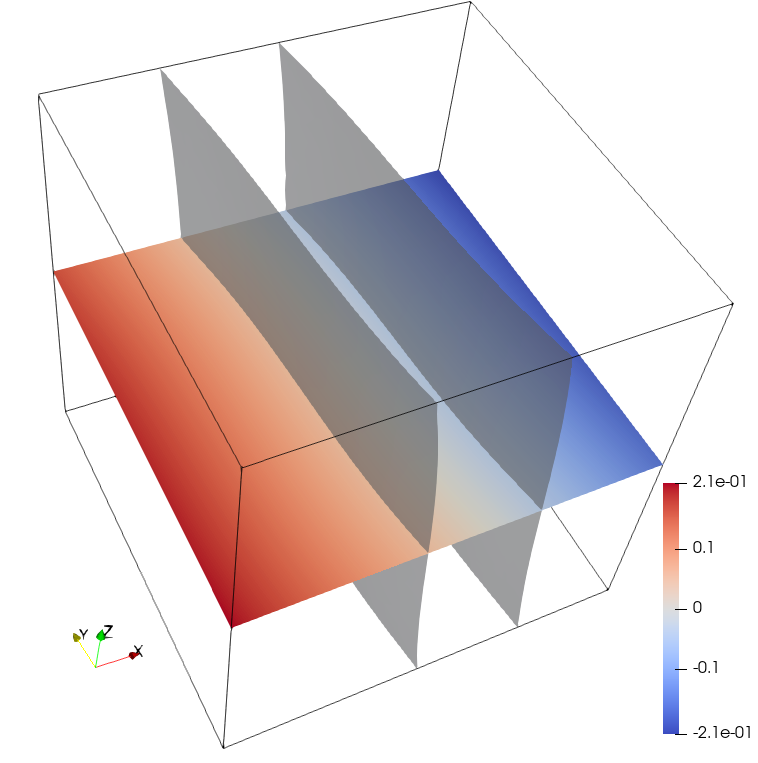}
\end{center}
\caption{Snapshots of the $\pm0.05$-level set of $n^{+}-n^{-}$ at $t=0.001,0.004, 0.006,0.03$, along with a cut through $n^{+}-n^{-}$ at $z=0$.}
\label{fig_appl3d}
\end{figure}

\begin{figure}[!htp]
\begin{center}
\includegraphics[width=0.3\textwidth]{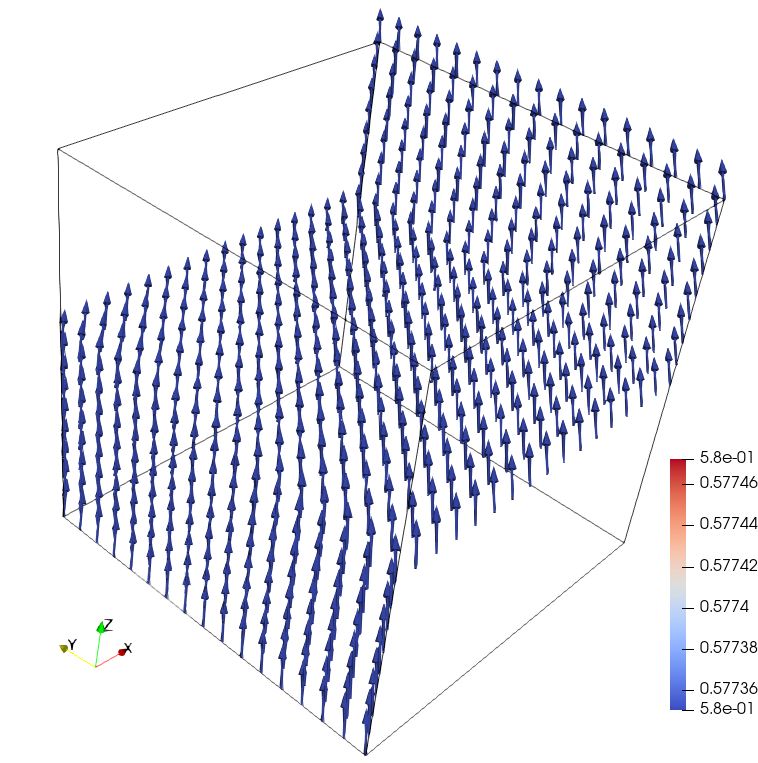}\quad
\includegraphics[width=0.3\textwidth]{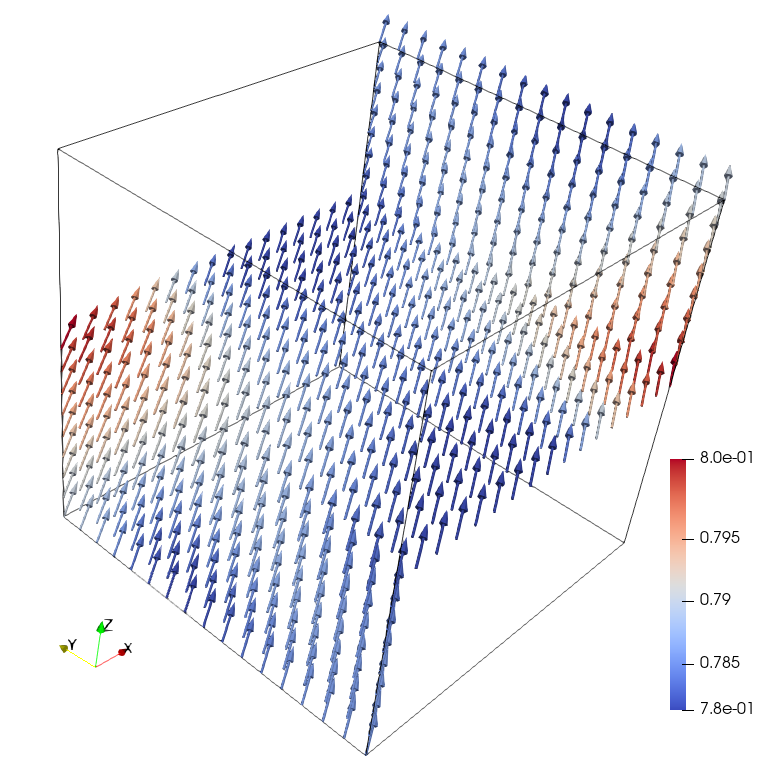}\quad
\includegraphics[width=0.3\textwidth]{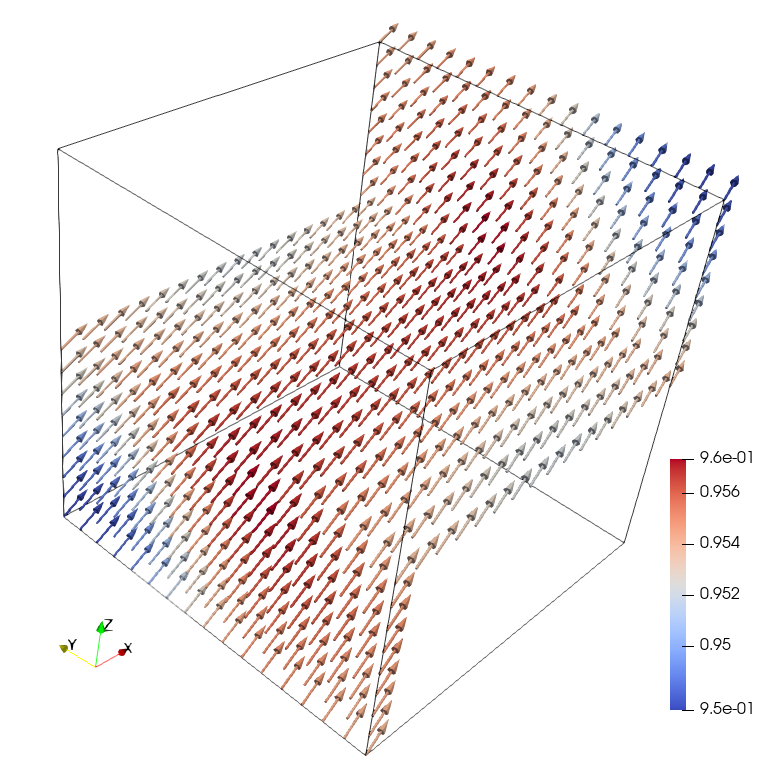}
\end{center}
\caption{Director at time $t=0,0.01,0.03$, colored by the magnitude of the $x$-component.}
\label{fig_appl3d_mh}
\end{figure}

\begin{figure}[!htp]
\begin{center}
\includegraphics[width=0.3\textwidth]{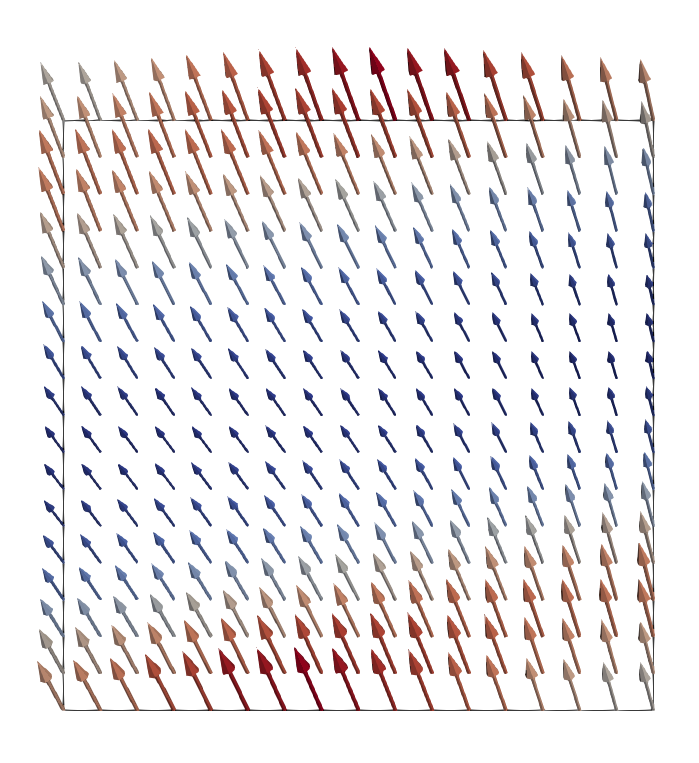}\quad
\includegraphics[width=0.3\textwidth]{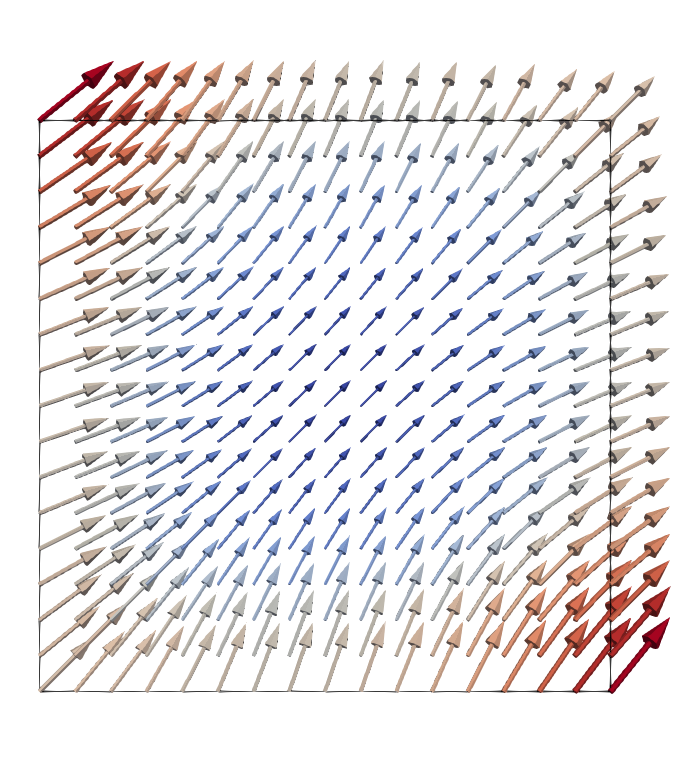}\quad
\includegraphics[width=0.3\textwidth]{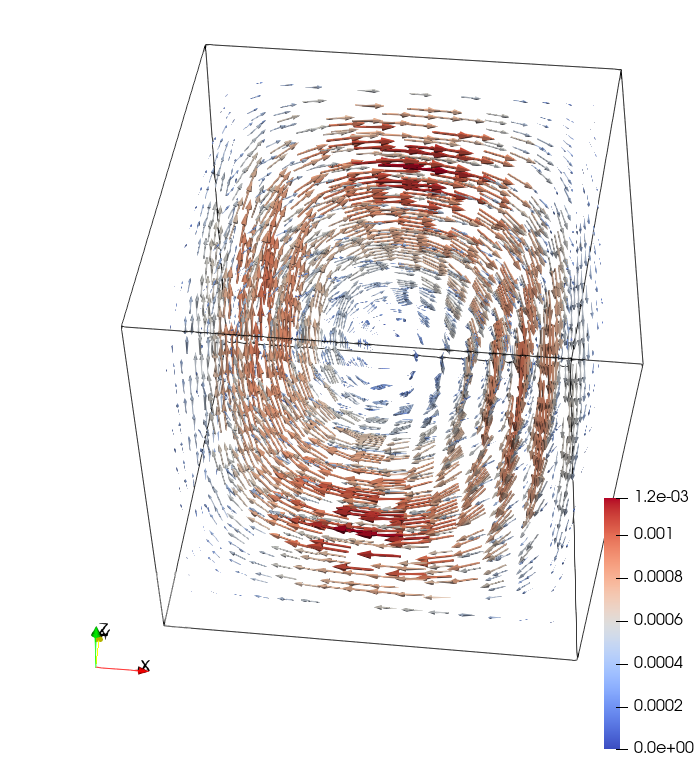}
\end{center}
\caption{The perpendicular component of the (negative) electric field ($\nabla \Phi= \nabla \widetilde{\Phi} - \f{E}_0$)
at $z=0$ and $y=0$ (left,middle), and the velocity (right) at time $t=0.03$.}
\label{fig_appl3d_phi_vel}
\end{figure}

\subsubsection{Effect of the director on the velocity}

In the next experiment we demonstrate how the director can be used to influence the velocity field which is generated by
an applied electric field.
We repeat the previous experiments with applied field $\f{E}_0 = (1,0,0)^T$, $h=2^{-5}$ 
and consider three different initial orientations of the director $\f{d}_0=(1,1,1)^T,\,\, (1,0,1)^T, \,\, (1,1,0)^T$.
The results displayed in Figure~\ref{fig_appl1_vel}
indicate that the direction of the induced velocity field is prescribed by the direction of the director field;
the respective velocities rotate around the respective directions $(0,-1,1)^T,\,\, (0,1,0)^T, \,\, (0,0,1)^T$ that are perpendicular to the 
respective initial conditions for the director.

\begin{figure}[!htp]
\begin{center}
\includegraphics[width=0.3\textwidth]{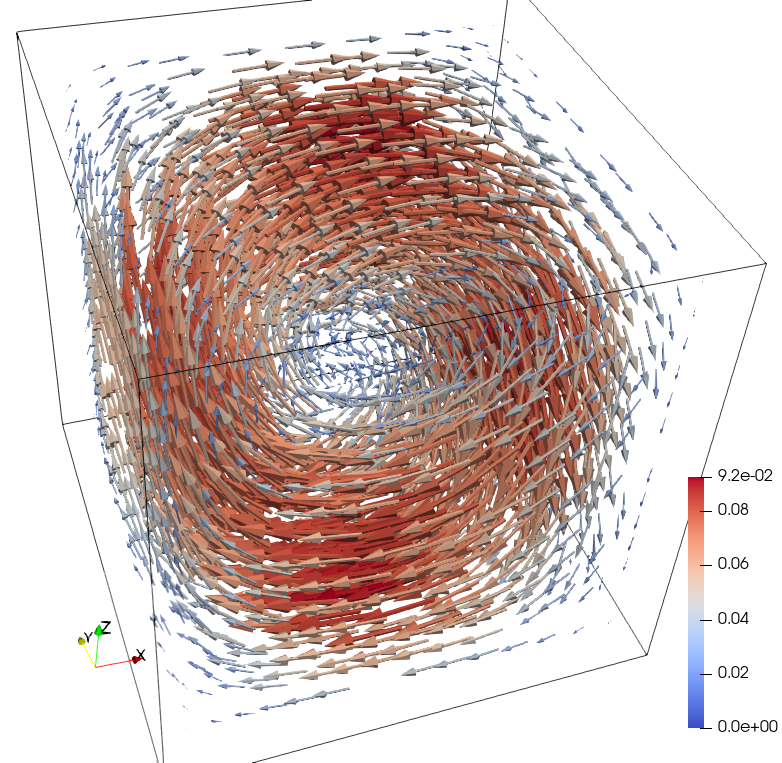}\quad
\includegraphics[width=0.3\textwidth]{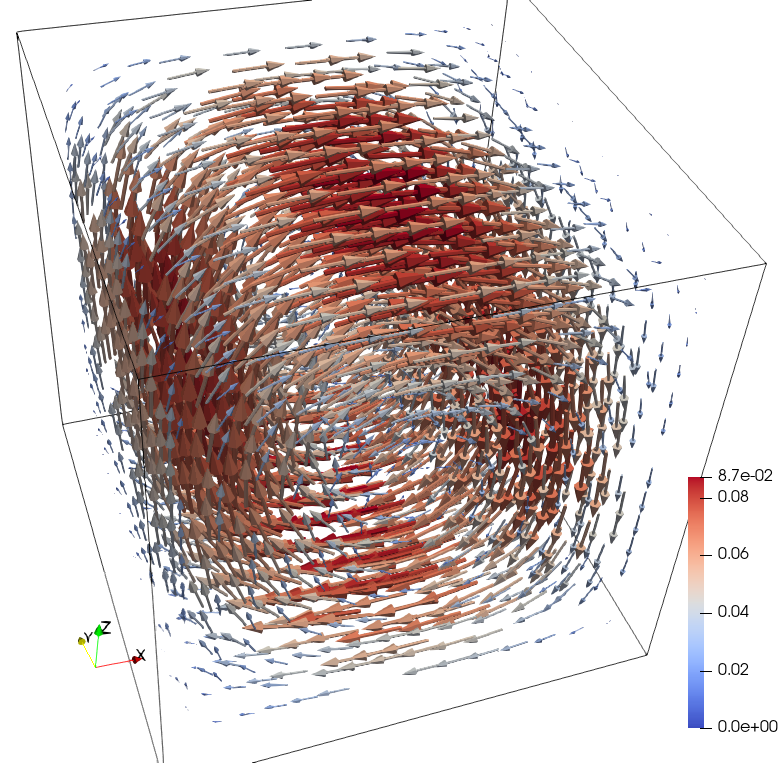}\quad
\includegraphics[width=0.3\textwidth]{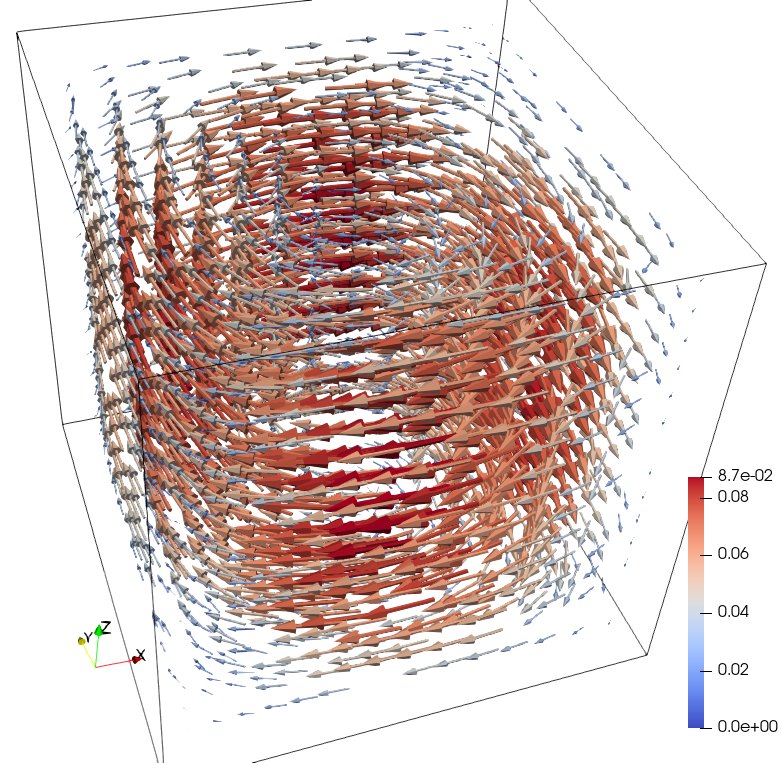}
\end{center}
\caption{(from left to right) Velocity induced by the applied field $\f{E}_0 = (1,0,0)$ at $t=0.04$ for $\f{d}_0=(1,1,1)^T$, $\f{d}_0=(1,0,1)^T$, and $\f{d}_0=(1,1,0)^T$.}
\label{fig_appl1_vel}
\end{figure}

When a constant electric field is applied, the induced velocity field quickly diminishes over time; see Figure~\ref{fig_appl1_vel_noac}.
By using an alternating electric field, it is possible to sustain the flow field over a longer period of time:
we consider an oscillating electric field $\f{E}_0(t) = (\cos(35 \pi t),0,0)$
and observe that the amplitude of the induced velocity oscillates in time 
but the flow retains its direction and persists over a longer time period; see Figure~\ref{fig_appl1_vel_ac} (the maximum amplitude of the
velocity is roughly half of the maximum amplitude in Figure~\ref{fig_appl1_vel_noac}).
Eventually the director aligns parallel to the applied electric field, the associated anisotropy effect vanishes,
and the induced velocity field becomes negligible.
We note that as long as the orientation of the director is fixed, it is possible to produce the desired flow pattern over
an arbitrary period of time; cf.~\cite{experi}.

\begin{figure}[!htp]
\begin{center}
\includegraphics[width=0.3\textwidth]{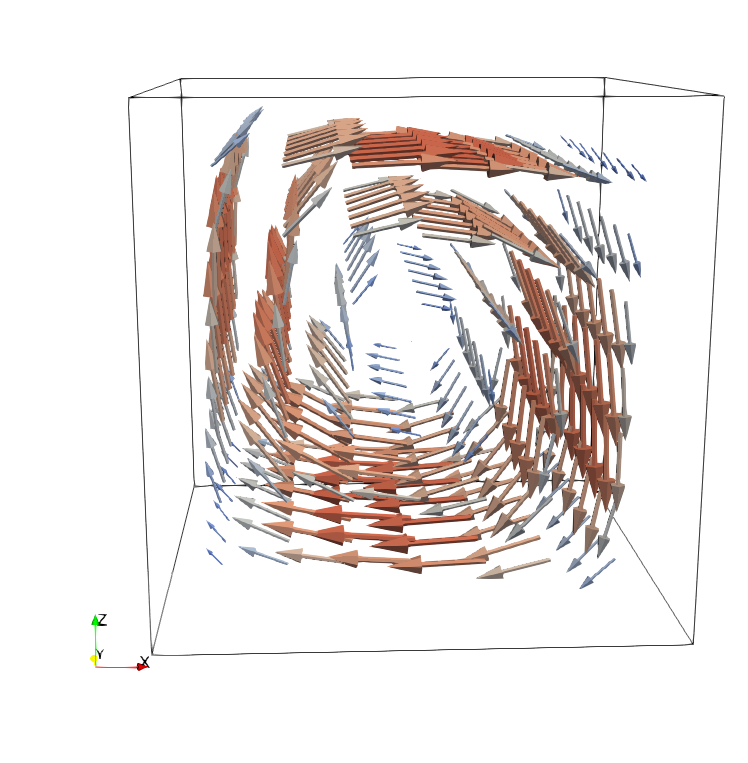}
\includegraphics[width=0.3\textwidth]{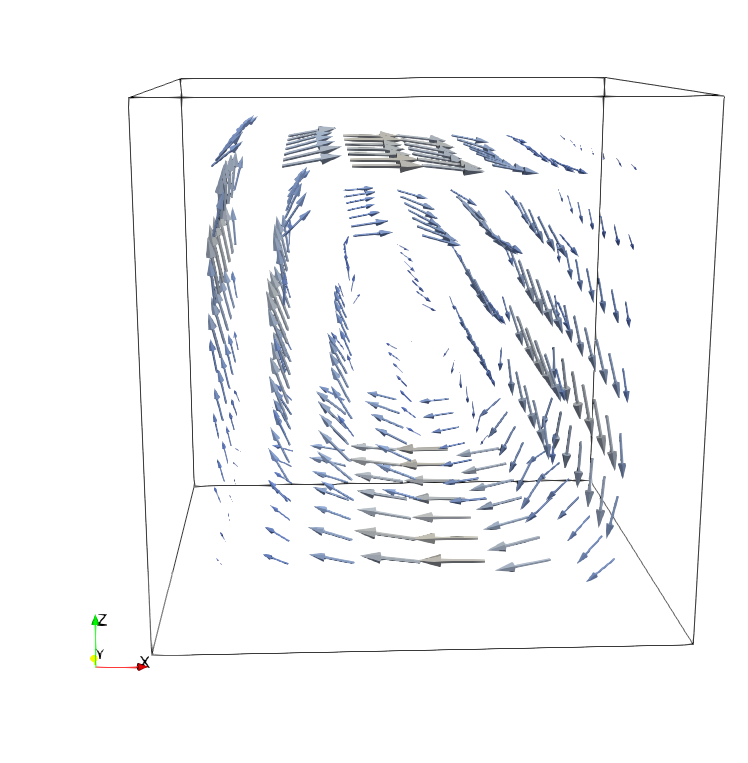}
\includegraphics[width=0.3\textwidth]{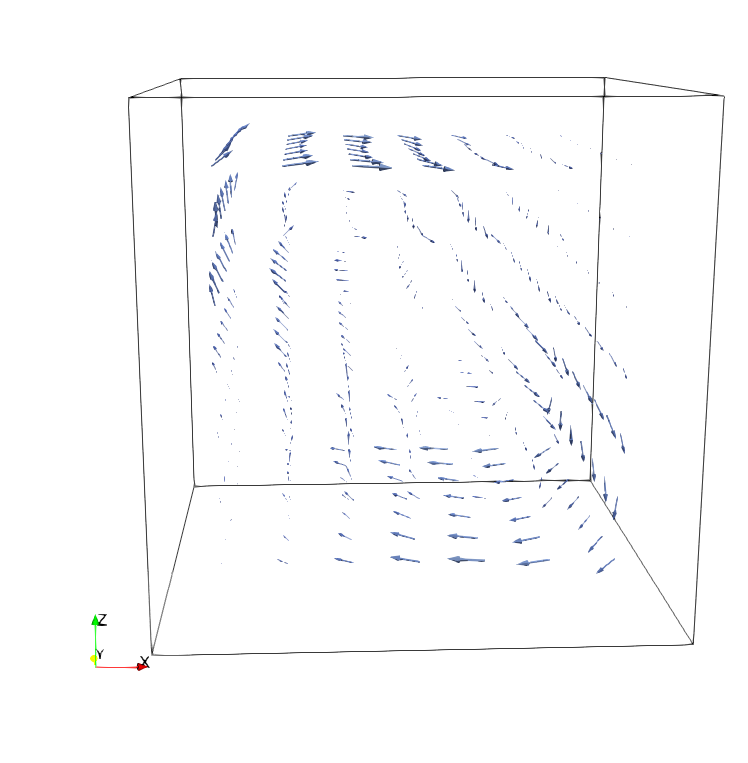}
\end{center}
\caption{Velocity field  at $t=0.04,0.07,0.1$ for $\f{d}_0=(1,0,1)^T$, and a constant applied field $\f{E}_0 = (1,0,0)$ (computed with $h=2^{-4}$).}
\label{fig_appl1_vel_noac}
\end{figure}

\begin{figure}[!htp]
\begin{center}
\includegraphics[width=0.24\textwidth]{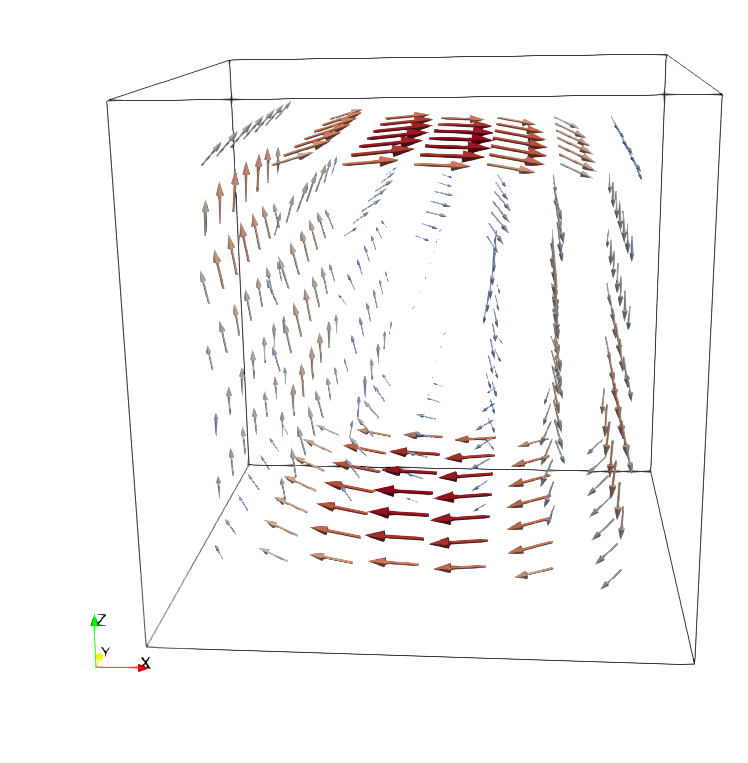}
\includegraphics[width=0.24\textwidth]{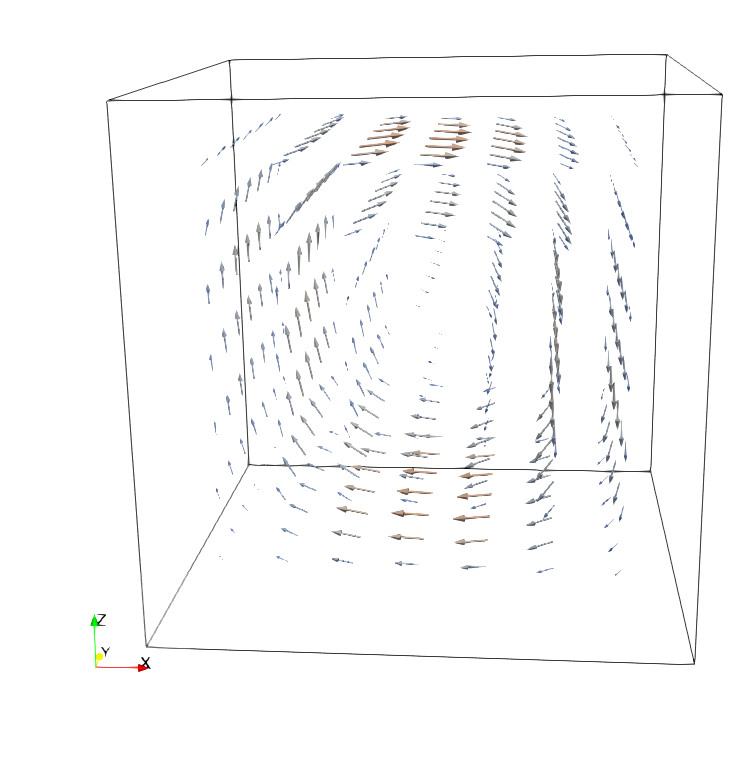}
\includegraphics[width=0.24\textwidth]{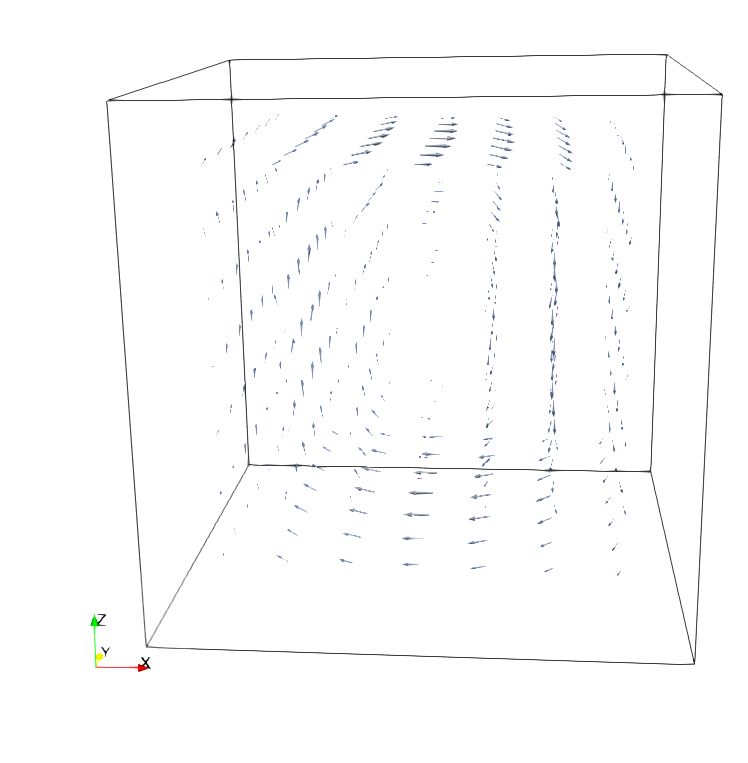}
\includegraphics[width=0.24\textwidth]{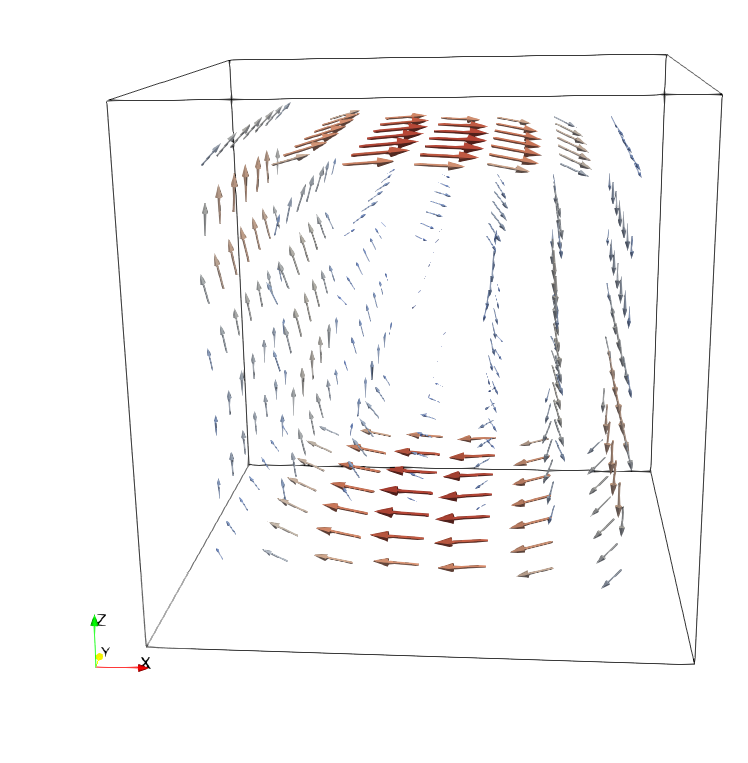}
\end{center}
\caption{Velocity field  at $t=0.12,0.13,0.14,0.15$ for $\f{d}_0=(1,0,1)^T$ and with oscillating applied field $\f{E}_0(t) = (\cos(35 \pi t),0,0)$ (computed with $h=2^{-4}$).}
\label{fig_appl1_vel_ac}
\end{figure}

\end{document}